%% file: main.tex
\documentclass[11pt,letterpaper]{article}
\usepackage[margin=1in]{geometry}
\usepackage[utf8]{inputenc}
\usepackage[giveninits=true,url=false,maxcitenames=10,bibstyle=numeric,maxnames=99,minnames=99,sorting=none]{biblatex}
\usepackage{float}
\usepackage{amsmath,amsthm,amssymb,mathtools,booktabs}
\usepackage[dvipsnames]{xcolor}
\usepackage[labelfont=normalfont,labelformat=simple]{subcaption}		

\usepackage{hyperref}
\newtheorem*{theorem*}{Theorem}
\newtheorem{theorem}{Theorem}[section]

\input{tikz_header_graph}

\title{Surfaces with given Automorphism Group}
\author{Reymond Akpanya and Tom Goertzen}
\date{\vspace{-5ex}}
\usepackage[colorinlistoftodos]{todonotes}
\newtheorem{lemma}{Lemma}
\newtheorem{prop}{Proposition}

\newtheorem{definition}{Definition}
\newtheorem{remark}{Remark}
\newtheorem{conjecture}{Conjecture}

\DeclareMathOperator{\Aut}{Aut}

\DeclareMathOperator{\ord}{ord}
\begin{document}

\maketitle


\section{Abstract}

Frucht showed that, for any finite group $G$, there exists a cubic graph such that its automorphism group is isomorphic to $G$. For groups generated by two elements we simplify his construction to a graph with fewer nodes. In the general case, we address an oversight in Frucht’s construction. We prove the existence of cycle double covers of the resulting graphs, leading to simplicial surfaces with given automorphism group. For almost all finite non-abelian simple groups we give alternative constructions based on graphic regular representations. 
In the general cases $C_n,D_n,A_5$ for $n\geq 4$, we provide alternative constructions of simplicial spheres. Furthermore, we embed these surfaces into the Euclidean 3-Space with equilateral triangles such that the automorphism group of the surface and the symmetry group of the corresponding polyhedron in $\mathrm{O}(3)$ are isomorphic.
\section{Introduction}

Combinatorial structures such as graphs and simplicial complexes are ubiquitous in mathematical research. The identification and study of these fundamental structures provides a unifying view of phenomena from a wide range of diverse mathematical disciplines. In particular, cubic graphs have been the focus of many studies in graph theory such as the cycle double cover conjecture, see Section \ref{preliminiaries}. In this paper, we investigate the relationship between cubic graphs and simplicial surfaces, with a focus on their respective automorphism groups. Simplicial surfaces describe the incidence relations of triangulated surfaces and they can be linked to cubic graphs by observing the incidence between faces and edges only.

In particular, we show in Section \ref{simplicial_surface_construction} that a construction by Frucht in \cite{frucht} yielding a cubic graph with given automorphism group leads to the following result:

\begin{theorem*}
    Let $G$ be a finite group generated by a set $S$. There exists a simplicial surface $X_{G,S}$ such that $Aut(X_{G,S})\cong G$.
\end{theorem*}

Furthermore, for cyclic groups and dihedral groups we show the following result in Section \ref{emb}:

\begin{theorem*}
    For $G=C_n$ with $n\geq 3$ or for $G=D_n$ with $n\geq 4$ there exists a simplicial surface $X_G$ with automorphism group isomorphic to $G$ and $X_G$ can be embedded into $\mathbb{R}^3$ with equilateral triangles.
\end{theorem*}

 As an example of a cubic graph consider the Petersen graph shown in Figure \ref{petersen}. A natural question that arises from examining such a graph is:
\begin{itemize}
    \item[Q1] \textit{How can we link a cubic graph to a simplicial surface, i.e.\ the combinatorics of triangulations of three dimensional polyhedra?}
\end{itemize}
This question was answered by Szekeres in \cite{szekeres} by interpreting a cubic graph as the graph describing the incidence structure between the faces and edges of a simplicial surface. In this case, the vertices of the surface correspond to the cycles in a so called cycle double cover of the cubic graph, i.e.\ a collection of cycles that passes every edge exactly twice.
Three such cycle double covers of the Petersen graph are given by
{\small \begin{align*}
    \{ (1,5,4,3,2), (1,6,9,7,2), (1,6,8,3,2,7,10,5), (4,9,6,8,10,5), (3,8,10,7,9,4) \}, \\
  \{ (1,5,4,3,2), (1,6,9,7,2), (2,7,10,5,4,9,6,8,3), (1,6,8,10,5), (3,8,10,7,9,4) \}, \\
  \{ (1,5,4,3,2), (1,6,9,7,2), (4,9,7,10,5), (3,8,6,9,4), (1,6,8,10,5), (2,7,10,8,3) \}.
\end{align*}}
Here, the third cycle double cover above corresponds to the surface illustrated in Figure \ref{petersen_surface}.
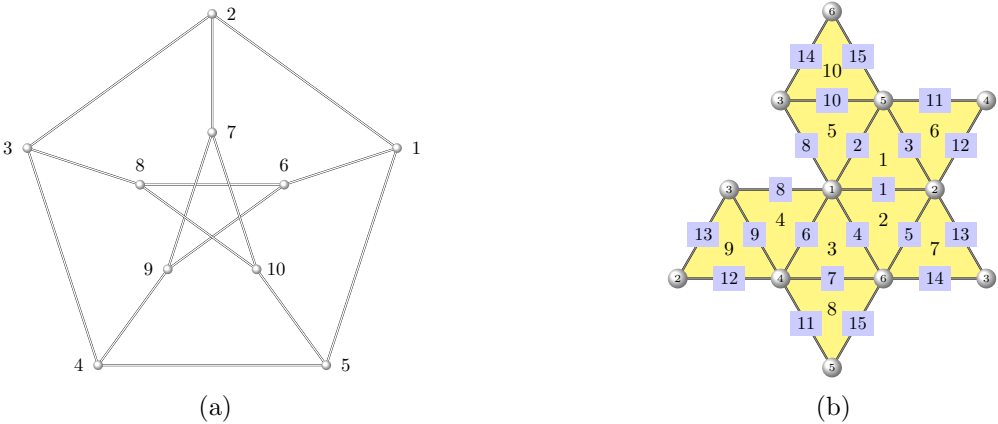
\begin{figure}[H]
\begin{minipage}{.49\textwidth}
    \centering
    \resizebox{!}{5cm}{\input{Petersen2}}
\subcaption{}
\label{petersen}
\end{minipage}
\begin{minipage}{.49\textwidth}
    \centering
    \resizebox{!}{5cm}{
    \input{PetersenSurface}
    }
    \subcaption{}
\label{petersen_surface}
\end{minipage}
\caption{(a) Petersen graph (b) Surface with facegraph isomorphic to the Petersen graph}
\end{figure}
If a surface can be constructed from computing a cycle double cover of a given cubic graph, the following question arises:
\begin{itemize}
    \item[Q2]  \textit{What is the resulting automorphism group of the simplicial surface?} 
\end{itemize}
  For the surface in Figure \ref{petersen_surface}, the automorphism group is isomorphic to $A_5$. This automorphism group corresponds to a subgroup of the automorphism group of the Petersen graph which leaves the corresponding cycle double cover invariant. Note, that the automorphism group of a surface always gives rise to a subgroup of the automorphism group of the underlying cubic graph, see Section \ref{surfaces_and_graphs} for a proof. As a next, step we can try to compute an embedding of the simplicial surface with equilateral triangles, i.e.\ realize the surface as a polyhedron built from equilateral triangles in the Euclidean 3-space.
Computing embeddings by solving a system of equations determined by a given simplicial surface turns out to be a task of high complexity.
In \cite{icosahedron} the authors elaborate on the complexity of solving this system of equations by computing all embeddings of a combinatorial icosahedron with equilateral triangles and non-trivial symmetry groups. Figure \ref{icosahedronfigure} shows three different embeddings of the icosahedron constructed from equilateral triangles.
\begin{figure}[H]
    \centering
\includegraphics[height=4cm]{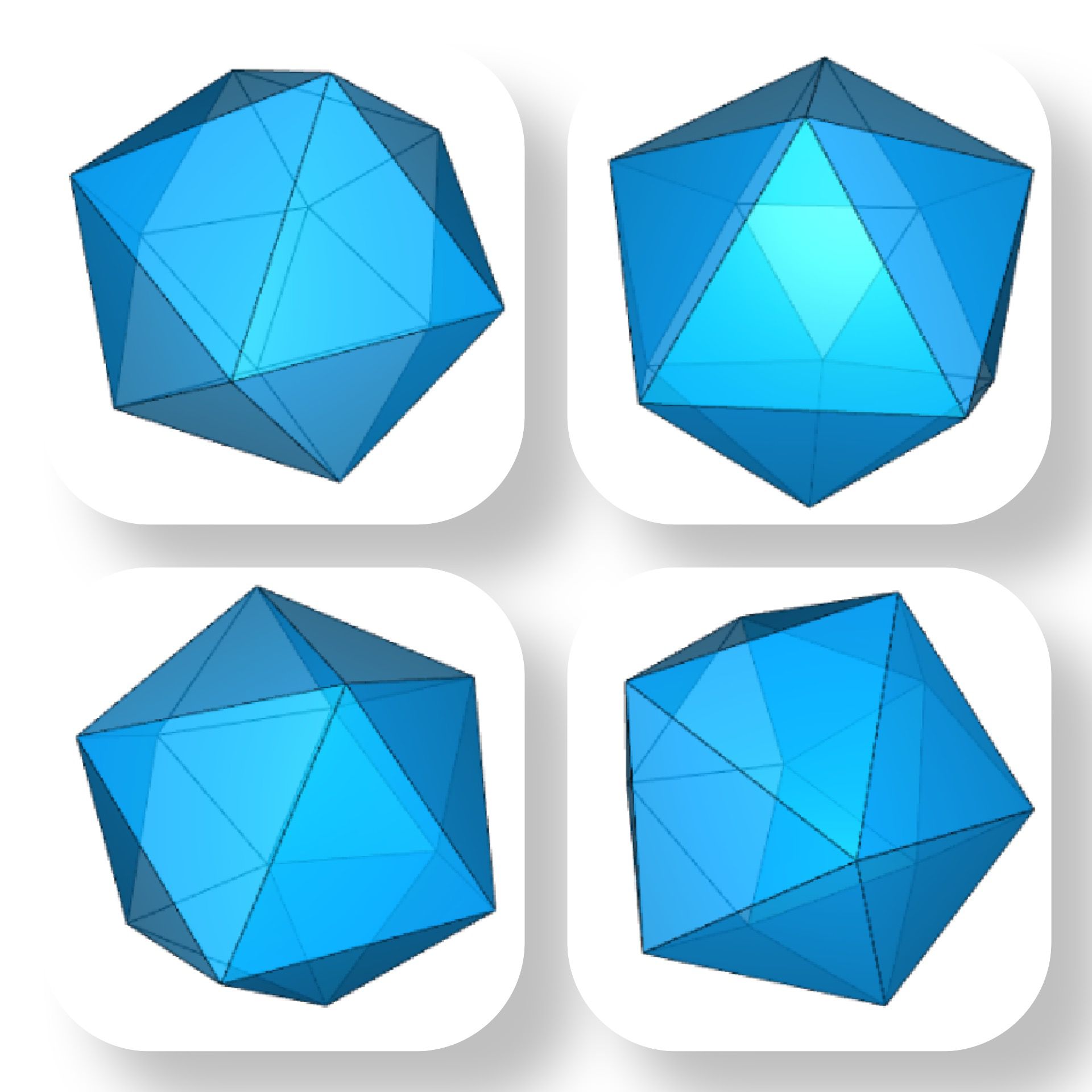}
\includegraphics[height=4cm]{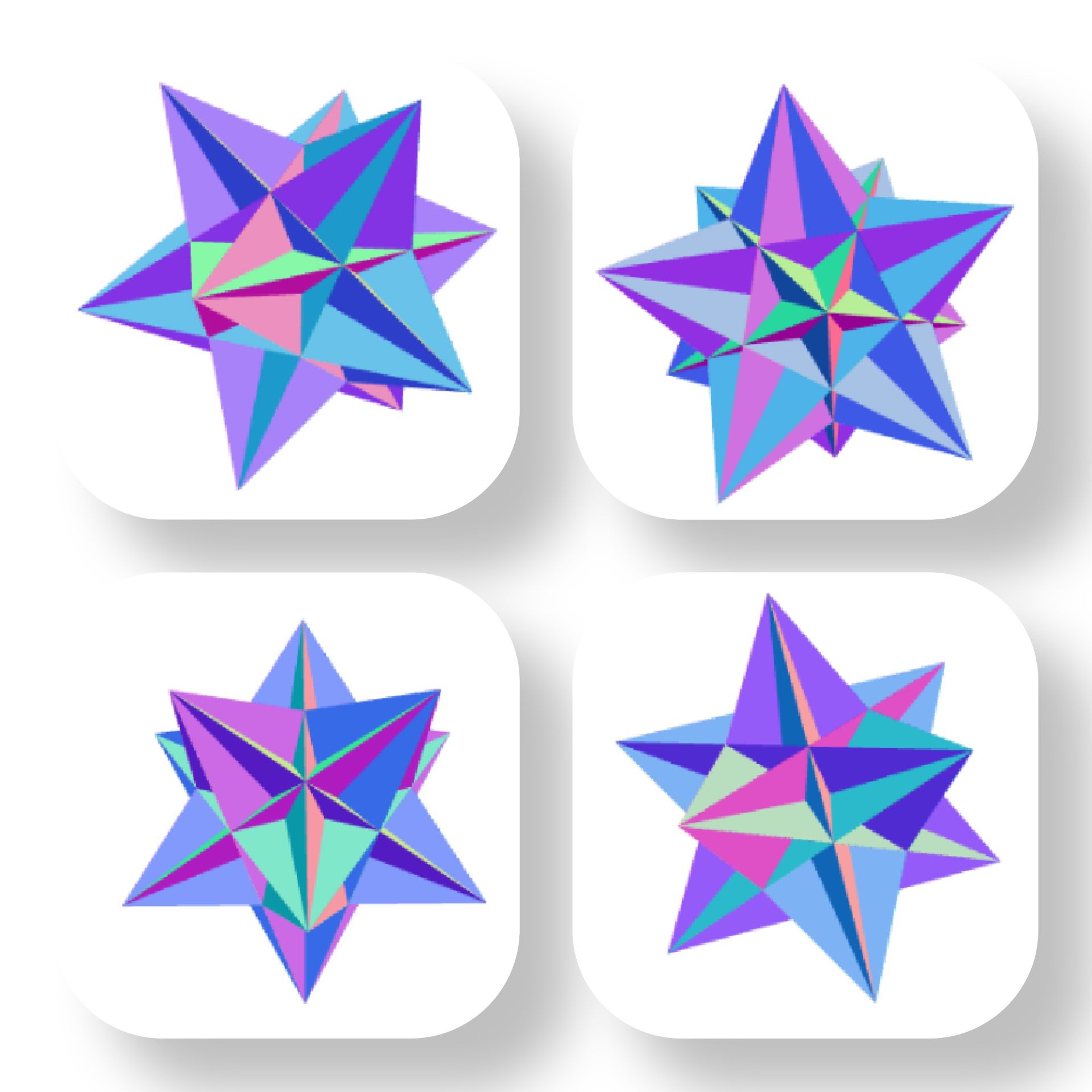}
\includegraphics[height=4cm]{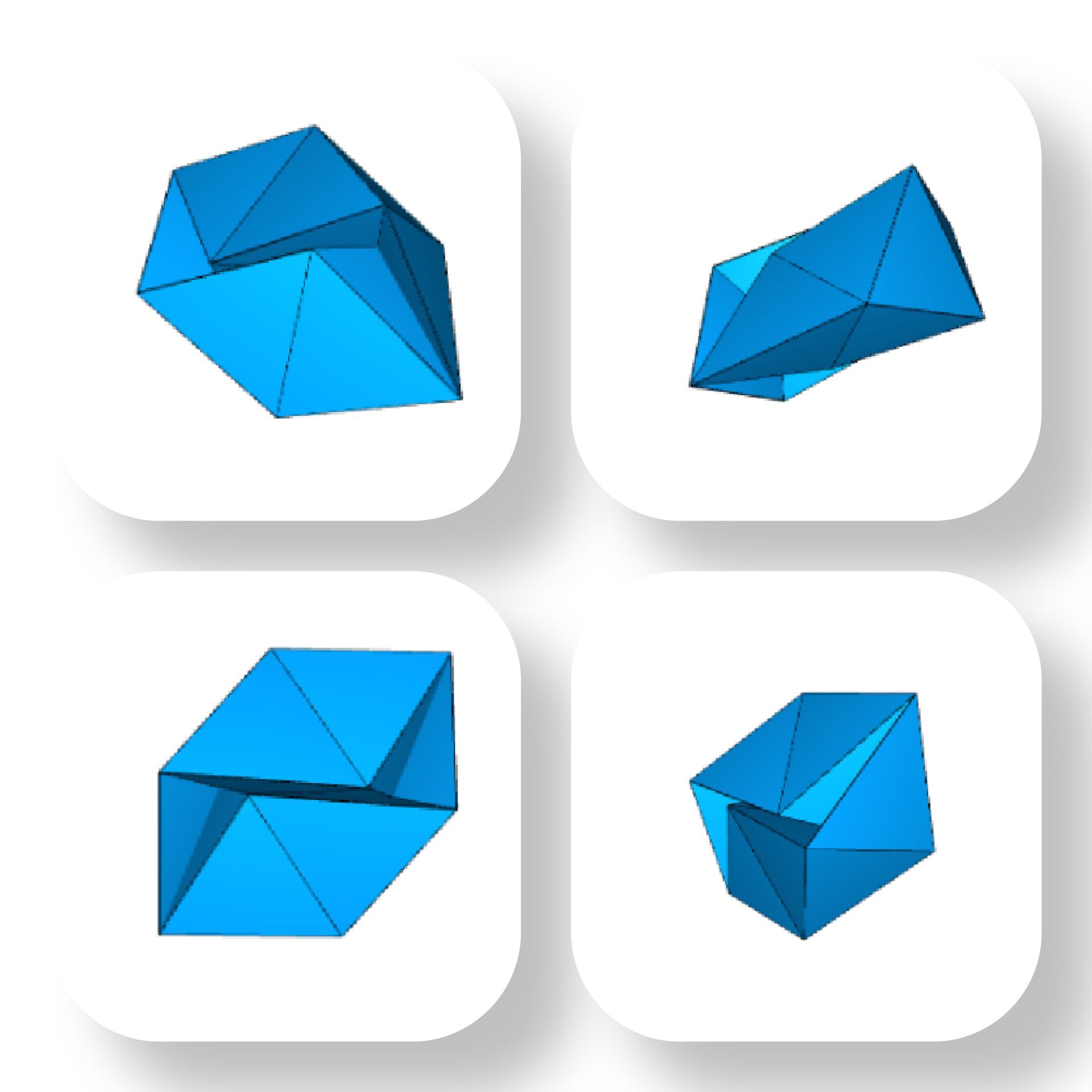}
    \caption{Different embeddings of the icosahedron: platonic solid (left), great icosahedron (middle) and a flexible icosahedron (right)}
    \label{icosahedronfigure}
\end{figure}
The symmetry group of an embedding of a surface is the subgroup in the Euclidean group $\mathrm{E}(3)$ leaving the embedding invariant. Such a symmetry group can be embedded into the group of orthogonal transformation $\mathrm{O}(3)$. Note that in general, the symmetry group of an embedded surface is a subgroup of the automorphism group of the underlying simplicial surface. Here, we aim to compute embeddings of a simplicial surface with a high number of symmetries as described in the following question: 
\begin{itemize}
    \item[Q3] \textit{Given a simplicial surface, can we compute an embedding into $\mathbb{R}^3$ such that the symmetry group of the surface is isomorphic to the automorphism group of the underlying simplicial surface?}
\end{itemize}
In this paper, we elaborate on the translation of cubic graphs with given automorphism group into a simplicial surfaces, analyze corresponding embeddings in some cases and therefore give answers to the questions Q1-Q3 for certain classes of cubic graphs and surfaces.

We show that there exists simplicial surface $X_G$ with $\Aut(X_G)\cong G$ by giving a $G$-invariant cycle double cover of a cubic graph based on Frucht's construction in \cite{frucht} or a cubic graph that forms a vertex transitive graph also known as generalized orbital graph. We define a 3-edge colouring to obtain this cycle double cover by applying the methods given in \cite{szekeres}.
Note, that cubic graphs do not neccesarily admit 3-edge colouring. For instance the Petersen graph, shown in Figure \ref{petersen}, is a well-known example of a vertex-transitive graph that does not admit a 3-edge colouring.
Furthermore, for $G=C_n,D_n,A_5$ we provide a simplicial surface $X_G$ with automorphism group isomorphic to $G$ and $X_G$ can be embedded with equilateral triangles. This is shown by exploiting the structure of the groups $C_n,D_n,A_5$ and their embeddings in $\mathrm{O}(3)$. For the cyclic and dihedral cases, we show in Section \ref{cyclic_embedding} that it suffices to consider group orbits of points in $\mathbb{R}^3$ in order to find embeddings with equilateral triangles.

In Section \ref{preliminiaries}, we introduce the theory of simplicial surfaces and their connections to cubic graphs. Moreover, we observe that a cubic graph has to be bridgeless in order to associate it to a surface. In the following three sections we present the graph constructions that are the focus of this paper. Section \ref{fruchts_graphs} deals with Frucht's cubic graph construction from \cite{frucht} yielding a cubic graph with given automorphism group. Here, we modify the construction for groups with $n>2$ generators, where Frucht missed a case. In Section \ref{cyclic_dihedral_graph} we present the graphs with cyclic automorphism group that arise from Frucht's construction and provide an alternative construction that yields cubic graphs with dihedral automorphism groups. We then introduce the construction of vertex-transitive cubic graphs in Section \ref{vertex-transitive}. In Section \ref{simplicial_surface_construction}, we show that we can associate a surface to each of the cubic graphs given in the previous sections. Finally in Section \ref{emb}, we compute infinite families of surfaces based on the cubic graphs presented in Section \ref{cyclic_dihedral_graph} such that the automorphism groups and symmetry groups are isomorphic of these surfaces. All the graph constructions discussed in this paper, along with the software for the visualization of surfaces and graphs, are implemented in the GAP4 package, SimplicialSurfaces \cite{simplicialsurfacegap}. We verify the automorphism groups of the given examples using the algorithms provided in \cite{MCKAY201494}, which are implemented in GAP4 \cite{GAP4}.

\section{Preliminaries} \label{preliminiaries}

In this section, we introduce the basic notion on simplicial surfaces, their relation to cubic graphs and their embeddings into $\mathbb{R}^3$.

\subsection{Simplicial Surfaces}
SSimplicial surfaces describe the incidence structure of triangulated surfaces. Compared to simplicial complex, where each element is uniquely described by its corresponding vertices, the definition of a simplicial surface below allows different faces and edges to have the same vertices by introducing an incidence relation between vertices, edges and faces. The following definitions are based on the work presented in \cite{simplicialsurfacesbook}.

\begin{definition} \label{simpl_surface}
A (closed) \emph{simplicial surface} $(X,<)$ is a countable set $X$ partitioned into non-empty sets $X_0$, $X_1$, and $X_2$
such that the relation $<$, called the \emph{incidence}, is a subset of the union  $X_0\! \times\! X_1\: \cup\: X_1\! \times\! X_2\: \cup \:X_0 \!\times\! X_2$, satisfying the following conditions.

\begin{enumerate}
\item
For each edge $e \in X_1$ there are exactly two vertices $V \in X_0$ with $V < e$.

\item
For each face $F \in X_2$ there are exactly three edges $e \in X_1$ with $e < F$ and three vertices $V \in X_0$ with $V < F$. Moreover, any of these three vertices is incident to exactly two of these three edges.

\item
For any edge $e \in X_1$ there are exactly two faces $F \in X_2$ with $e < F$. 

\item Umbrella condition: 
For any vertex $V \in X_0$, the number $n=\deg(V)$ of faces $F_i \in X_2$ with $V < F_i$ satisfies $3\leq \deg(V) <\infty$ and is called the \emph{degree} of the vertex $V$. The $F_i$ can be arranged in a sequence $(F_1, \ldots , F_n)$ such that $F_{i+1}$ and $F_i$ share a common edge $e_i$ with $V<e_i$ for $i=1, \ldots , n$ and we set $F_{n+1}=F_1$. This sequence can be viewed as a cycle $(F_1,\ldots ,F_n)$ called the \emph{umbrella} of $V$.
\end{enumerate}
We call the elements of $X_0$, $X_1$ and $X_2$, vertices, edges and faces, respectively.

\end{definition}

The above definition requires each edge to be incident to two faces. We can relax the definition by differentiating between closed surfaces, where each edge is incident to two faces and open surfaces, containing an edge which is only incident to one face. Furthermore, we can generalize the definition by allowing vertices with degree $2$.
From now on, when we speak of surfaces we mean simplicial surfaces in the sense above and omit the incidence relation, whenever it is clear from the context. The definition of simplicial surfaces above allows that two distinct edges share common vertices and  we reintroduce conventional simplicial complexes by introducing vertex-faithful simplicial surfaces as defined below.

\begin{definition}
A surface $X$ is called \emph{vertex-faithful} if its edges and faces are uniquely described by its incident vertices, i.e.\ the following map is injective
$$ X\to P(X_0)\coloneqq\{S\subseteq X_0 \},\,x\mapsto X_0(x)\coloneqq \{V\in X_0 \mid V<x \text{ or }V=x \}.$$    
\end{definition}

The image of the above map always yields a simplicial complex in the conventional way.

 We define properties such as orientability and the Euler characteristic in the usual sense.

\begin{definition}
For a simplicial surface $X$ we can compute its \emph{Euler characteristic} $\chi(X)$ as
$$ \chi(X)=|X_0|+|X_2|-|X_1|.$$
\end{definition}

For example, a simplicial sphere has Euler characteristic $2$ and a simplicial torus has Euler characteristic $0$.

\begin{definition}
    A \emph{homomorphism} between two simplicial surfaces $(X,<_X),(Y,<_Y)$ is a map $\pi:X\to Y$ satisfying the following two conditions.
\begin{enumerate}
    \item For $A,B\in X$ with $A<_X B$ we have that $\pi(A)<_Y\pi(B)$.
    \item For every face $F\in X_2$ the restriction of $\pi$ to the vertices and edges incident to $F$ is an isomorphism onto the set of vertices and edges that are incident to $\pi(F)$.
\end{enumerate}

Auto-, mono- and epimorphisms are defined in the usual way and we can define the category of simplicial surfaces.
\end{definition}

By observing that the degree of a vertex is invariant under an isomorphism of a simplicial surface, we can define invariants of isomorphic surfaces. One is the so-called \emph{vertex-counter} of a surface $X$ which counts the degrees of the vertices and is defined as a polynomial in the indeterminates $\{v_n\mid n\geq 3\}$, i.e.\
$$ \prod_{v\in X_0} v_{\deg_X(v)}.$$

\subsection{Simplicial Surfaces and Graphs} \label{surfaces_and_graphs}

There are several ways of linking a graph $\Gamma$ to a surface $X$. Next, we present some of these graphs based on partial incidence structures of a given simplicial surface, see also \cite{simplicialsurfacesbook}. In the rest of the paper, we refer to the vertices of a graph as nodes in order to distinguish between vertices of a surface and those of a graph.

\begin{definition}
    Let $(X,<)$ be a simplicial surface.
    \begin{enumerate} 
    
    \item The \emph{incidence graph}, denoted by $\mathcal{I}(X)$, of a surface $X$ has nodes $X$ and there exists an edge between two nodes $A,B\in X$, whenever we have $A<B$ in $X$.
    \item The \emph{vertex(-edge) graph} of $X$, denoted by $\mathcal{V}(X)$, has nodes $X_0$ and edges $X_1$ such that an edge $e\in X_1$ connects the two vertices in $X_0(e)$.
    \item The \emph{face(-edge) graph} of $X$, denoted by $\mathcal{F}(X)$, has nodes $X_2$ and edges $X_1$ such that an edge $e\in X_1$ connects the two faces in $X_2(e)=\{F \in X_2 \mid e<F \}$.
    \end{enumerate}
\end{definition}

Since each face of a (closed) surface $X$ has three edges, the face graph $\mathcal{F}(X)$ is cubic. 
Normally, mostly (2), is used to link a graph to a given surface, see for instance \cite{szekeres}. Here, we primarily make use of the graph in (3) as it gives us a strong connection between the theory of cubic graphs and simplicial surfaces.
Note that, the umbrella condition in Definition \ref{simpl_surface} enforces the face graph $\mathcal{F}(X)$ of a surface $X$ to be bridgeless, since each edge must lie on a cycle that is induced by the umbrella of a vertex $V\in X$. Furthermore, it is straightforward to show that a vertex-faithful surface yields a 3-connected graph. The following lemma demonstrates our key tool in establishing a connection between the theory of simplicial surfaces and cubic graphs.

\begin{lemma}
$\mathcal{F}(\cdot)$ is a functor between the category of (closed) simplicial surfaces and the category of bridgeless cubic graphs.
\end{lemma}

\begin{proof}
Let $X,Y$ be two simplicial surface and $\pi:X\to Y$ a homomorphism between them. By restricting $\pi$ to $X_1\cup X_2$ we get a map $\pi|_{X_1\cup X_2}:\mathcal{F}(X)\to \mathcal{F}(Y)$ which is a graph homomorphism since for $e\in X_1, f \in X_2$ with $e<f$ we have that $\pi(e)<\pi(f)$.
\end{proof}

Since an automorphism of a simplicial surface yields an automorphism of its face graph we obtain the following result.

\begin{lemma}
    For a simplicial surface $X$ we have that $\Aut(X)\hookrightarrow \Aut(\mathcal{F}(X))$. Moreover, we have $\Aut(X)\cong \Aut(\mathcal{F}(X))$ if and only if for all $\pi \in \Aut(\mathcal{F}(X))$ and all vertices $v\in X_0$, the isomorphism $\pi$ maps the umbrella $U=(F_1,\dots,F_n)$ at $v$ onto an umbrella of another vertex of $X$ via $\pi(U)=\pi((F_1,\dots,F_n))=(\pi(F_1),\dots,\pi(F_n))$. Thus, each automorphism of $\mathcal{F}(X)$ can uniquely be extended to an automorphism of $X$. 
\end{lemma}

Hence, the task of constructing a simplicial surface $X$ with given group $G$ as automorphism group can be reduced to finding a cubic graph $\Gamma$ with automorphism group isomorphic to $G$ and a surface $X$ with $\mathcal{F}(X)=\Gamma$ such that $\Aut(X)\cong \Aut(\mathcal{F}(X))$. In \cite{szekeres}, the author describes how to compute $\mathcal{F}^{-1}(\Gamma)$ for a cubic graph $\Gamma$, i.e.\, the simplicial surfaces with a given face graph. This can be achieved by computing a cycle double cover.

\begin{definition}
A \emph{cycle double cover} or \emph{polyhedral decomposition} (see \cite{szekeres}) of a cubic graph $\Gamma$ consists of a collection of cycles $(C_i)_{i\in I}$ such that for each node in $\Gamma$ there are exactly three cycles passing this node. This is equivalent to saying that there are exactly two cycles traversing each edge.  
\end{definition}

So a simplicial surface can be recovered from its face graph with a cycle double cover, where a cycle $C_i$ corresponds to the umbrella of a vertex of the  surface. It follows that a simplicial surface $X$ is vertex-faithful if and only if no two cycles of its corresponding cycle double cover share two common edges.
It is still an open problem, whether all cubic bridgeless graphs admit a cycle double cover.

\begin{conjecture}[Cycle double cover conjecture] \label{cdc_conjecture}
In \cite{szekeres}, the author conjectures that any cubic graph containing no bridges has a cycle double cover or equivalently is a face graph of a simplicial surface. Moreover, it is conjectured that any bridgeless graph admits a cycle double cover, see \cite{JAEGER19851}.
\end{conjecture}

For the graphs presented in the following sections, we always give a cycle double cover. One of the main tools for finding cycle double covers is the following observation for 3-edge colourable cubic graphs.

\begin{remark}
    In \cite{szekeres}, the author shows that a 3-edge colouring, also called \emph{Tait colouring}, of a cubic graph leads to a cycle double cover. Any cycle of the double cover is obtained by choosing two colour classes and alternating between them. The resulting surface, has a \emph{Gr\"unbaum colouring} corresponding to this 3-edge colouring.
\end{remark}

  To show Conjecture \ref{cdc_conjecture} it suffices to consider cubic graphs that do not admit a 3-edge colouring, see \cite{JAEGER19851}. Such graphs are known as Snarks in the literature. The Petersen graph and the graph shown in Figure \ref{a_5_graph} both are examples of such graphs. Note that the Petersen graph can be obtained from the graph in Figure \ref{a_5_graph} by contracting three cycles into one node. On the other hand, one can subdivide each face into three faces of the surface given in Figure \ref{petersen_surface} to obtain a surface with face graph shown in Figure \ref{a_5_graph}.

\begin{figure}[H]
\begin{minipage}{\textwidth}
    \centering
    \resizebox{!}{7cm}{\input{A5_16_3}}
\end{minipage}
\caption{Facegraph of a surface with automorphism group isomorphic to $A_5$}
\label{a_5_graph}
\end{figure}
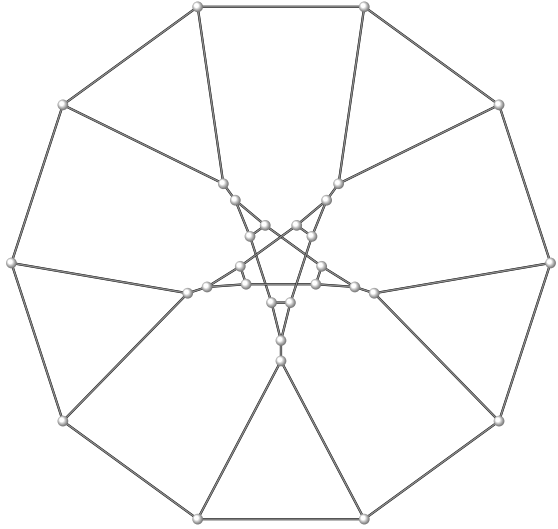

\subsection{Embeddings of Simplicial Surfaces}

In order to translate a simplicial surface into a polyhedron with equilateral triangles, one has to solve the \emph{embedding problem}. Solving the embedding problem is equivalent to finding an embedding of a simplicial surface as defined below.

\begin{definition}\label{embedding}
Let $X$ be a simplicial surface. A map $\phi : X_0 \mapsto \mathbb{R}^3$ such that all neighbouring vertices $v_1$ and $v_2$ satisfy
\[
\|\phi(v_1) -\phi(v_2)\|^2 = 1
\]
in the Euclidean norm is called an \emph{embedding} of $X$ with equilateral triangles.
\end{definition}

For 3-edge coloured simplicial surfaces, we can identify each colour with an edge length and ask for embeddings with a given congruence class of triangles.

\begin{remark}
  Assume that $X$ is a simplicial surface with edges coloured in red (r), green (g) and blue (b). Furthermore, let $\ell_r, \ell_g, \ell_b$ be non-negative real numbers satisfying the triangle inequality, i.e.\ there exists a triangle with those edge lengths. We seek to embed $X$ into the Euclidean 3-space, whereby the embedding is constructed out of congruent triangles with edge lengths $(\ell_r, \ell_b, \ell_g)$.  
\end{remark}

The question of computing such an embedding introduces a
system of quadratic equations.
In general, it remains an unsolved problem whether there exists an embedding of a given simplicial surface, constructed from triangles of a certain congruence type.
For instance, if $X$ is a simplicial surface, whose vertex graph $\mathcal{V}(X)$ contains a clique of size $n\geq 5$, then an embedding constructed from equilateral triangles does not exist. This can be seen as follows: Assume that a simplicial surface $X$ such that $\mathcal{V}(X)$ contains a clique of size $n\geq 5$ has an embedding constructed from equilateral triangles. Let $v_1,\ldots,v_5$ be 3-dimensional coordinates corresponding to vertices of the described clique. Enforcing the coordinates $v_1,v_2,v_3$ to satisfy the desired pairwise distance, gives rise to a equilateral triangle. Without loss of generality we can assume that this triangle with edge lengths $1$ is given by 
$$(v_1,v_2,v_3)=\left((0,0,0)^t,(1,0,0)^t,(\frac{1}{2},\frac{\sqrt{3}}{2},0)^t\right).$$
Clearly, $v_4$ and $v_5$ have to be on the intersection of the unit spheres with centres at the points $v_1,v_2,v_3$. As the intersection equals $$\{(\frac{1}{2},\frac{1}{2\sqrt{3}},\sqrt{\frac{2}{3}})^t,(\frac{1}{2},\frac{1}{2\sqrt{3}},-\sqrt{\frac{2}{3}})^t \}$$ and since $v_4\neq v_5$, we can assign $v_4$ to the first and $v_5$ to the second coordinate in the set. This leads to $\| v_4 -v_5 \|=2 \sqrt{\frac{2}{3}}$ which contradicts the existence of such an embedding.

For example, the minimal triangulation of the torus $T$ consisting of 7 vertices, 21 edges and 14 triangles and incidence relations depicted as in Figure \ref{foldingTorus} does not have an embedding consisting of equilateral triangle, since the graph $\mathcal{V}(T)$ is isomorphic to the complete graph on 7 vertices. Although this simplicial surface cannot be embedded into the Euclidean space as a polyhedron constructed from equilateral triangles, its combinatorial structure can be visualized in a folding plan, see Figure \ref{foldingTorus}.

\begin{figure}[H]
    \centering

\scalebox{0.75}{\input{torus}}
    \caption{A minimal triangulation of the torus}
    \label{foldingTorus}
\end{figure}
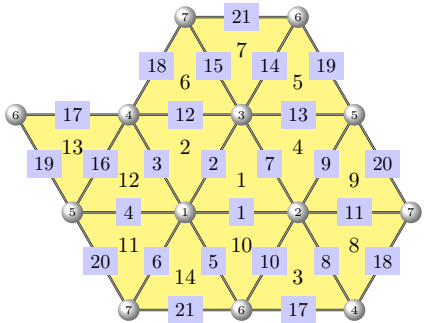
The question whether, a spherical simplicial surface has an embedding constructed from equilateral triangles remains unanswered.
But the existence of such an embedding has been studied for simplicial surfaces that form a combinatorial torus. 
For example, in \cite{tetrahedra_torus} the authors prove that by iteratively gluing together the faces of regular tetrahedra it is impossible to construct a torus built from equilateral triangles.

For a simplicial surface we  denote the symmetry group of the rigid body in $\mathbb{R}^3$ that arises from an embedding $\phi$ of $X$ by $\Aut(\phi(X))\leq \mathrm{O}(3)$. Note, that for a simplicial surface $X$ with an embedding $\phi$ the following holds:
\[
 \Aut(\phi(X))\hookrightarrow \Aut(X)\hookrightarrow \Aut(\mathcal{F}(X)).
\]

\section{Frucht's cubic graphs} \label{fruchts_graphs}

In this section, we recall Frucht's cubic graph construction introduced in \cite{frucht}. For a given group $G$ with generators $S=\{g_1,\dots,g_n \},$ Frucht constructs a cubic graph with the property that the automorphism group of the obtained graph is isomorphic to $G.$ In his construction he distinguishes between cyclic and non-cyclic groups. Here, we give a slightly simplified construction for a group $G$ generated by two elements. Furthermore, we present a modified construction for a group $G$ generated by more than three elements, as in this case, Frucht's construction in general does not yield a graph with the given group as automorphism group. Moreover, we show that our constructed graph has a three-edge colouring and hence yields a cycle double cover, as described in \cite{szekeres}. 
For a general group $G$ generated by $n$ elements $g_1,\dots,g_n\in G$, the cubic graph constructed by Frucht has $2(n+2)|G|$ nodes. For $n=2$, we can modify the graph by contracting 3-cycles such that the resulting graph has $2(2+1)|G|=6|G|$ nodes and still has the same automorphism group.

Frucht's construction is based on Cayley graphs, which are directed edge coloured graphs.
\begin{definition}
 Let $G$ be a finite group with generators $S=\{g_1,\dots,g_n\}$, such that $\langle S \rangle=G$. The nodes of the Cayley graph $\mathcal{C}_{G,S}$ are given by the elements of $G$ and its edges by $\{(g,gs) \mid g\in G,s\in S \}$. The $n$-edge colouring is given as follows:
    $\{(g,g\cdot s) \mid g\in G,s\in S \}\to \{1,\dots,n \}, (g,g\cdot g_i)\mapsto i$.   
\end{definition}

In Figure \ref{frucht4}, we see a corresponding uncoloured and undirected version of a Cayley graph with 2 generators.
The automorphism group of a Cayley graph corresponds to the left-action of the group $G$ on itself.
\begin{lemma}
  The automorphism group of a Cayley graph $\mathcal{C}_{G,S}$ respecting its edge colouring is isomorphic to $G$ and is given by the maps $\pi_g:G\to G, h\mapsto g\cdot h,$ for all $g\in G$.
\end{lemma}

For a finite group $G$ generated by $n$ elements, Frucht defines a cubic graph based on this constructions by spliting the nodes corresponding to a group element $g\in G$ into several nodes $x_{i,g},i=1,\dots,2n+4$ as shown in Figure \ref{frucht1}.

\begin{figure}[H]
\begin{minipage}{.49\textwidth}
    \centering
    \resizebox{!}{5cm}{\input{frucht_graph_part4}}
    \subcaption{}
    \label{frucht4}
\end{minipage}
\begin{minipage}{.49\textwidth}
    \centering
    \resizebox{!}{5cm}{\input{frucht_graph_part1}}
    \subcaption{}
\label{frucht1}

\end{minipage}
\caption{(a) Undirected Cayley graph (b) Part of Frucht's original cubic graph construction}
\end{figure}
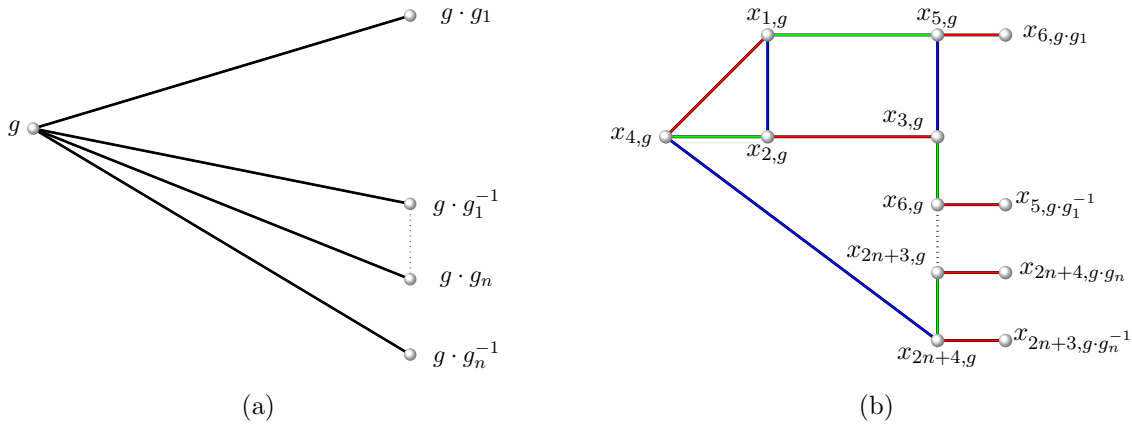

To formalise this, we make use of quadratic forms to define cubic graphs.
\begin{remark}
A cubic graph can be defined by a quadratic form which is related to the adjacency matrix of the graph in the following way: Let $x=(x_1,\dots,x_n)$ denote formal variables corresponding to the nodes of a graph $\Gamma$ with $n$ nodes. Let $A=(a_{ij})$ denote the adjacency matrix of $\Gamma$, and define $$ Q=\frac{1}{2}xAx^t .$$ Each monomial of $Q$ corresponds to an edge in $\Gamma$ and we can obtain $A$ from $Q$ via $a_{ij}=Q(e),$ with $e=e_i+e_j$, where $e_i$ denotes the vector with a $1$ at the $i$th position and $0$'s everywhere else.
The importance of this point of view is the observation that the automorphism group of the graph $\Gamma$ is isomorphic to the group of permutations of the variables that preserve the quadratic from $Q$.
\end{remark}

Let $h=|G|$ and $S=\{g_1,\dots,g_n\}$ such that $\langle S \rangle=G$. We denote the other elements of $G$ by ${g_{n+1},\dots,g_h}$ and define a graph via a quadratic form with nodes corresponding to the indeterminates $x_{i,g_j}$, with $i=1,\dots,2n+4,j=1,\dots,h$. For $i\neq j$ we define the quadratic form $$Q_{ij}=\sum_{k=1}^h x_{i,g_k}x_{j,g_k}.$$ Based on $Q_{ij}$, Frucht defines the quadratic forms $Q$ and $R$ that yield a cubic graph:
\begin{align*}
Q & ={\color{blue}{Q_{1,2}}{\color{black}{+}{\color{red}{Q_{1,4}}{\color{black}{+}{\color{ForestGreen}{Q_{1,5}}{\color{black}{+}{\color{red}{Q_{2,3}}{\color{black}{+}{\color{ForestGreen}{Q_{2,4}}{\color{black}{+}{\color{blue}{Q_{3,5}}{\color{black}{+}{\color{ForestGreen}{Q_{3,6}}{\color{black}{+}{\color{blue}{Q_{4,2n+4}}{\color{black}{+}{\color{ForestGreen}{\sum_{i=2}^{n}Q_{2i+3,2i+4}}{\color{black}{+}{\color{blue}{\sum_{i=2}^{n}Q_{2i+2,2i+3}}{\color{black}{+}{\color{red}{R},}}}}}}}}}}}}}}}}}}}}}\\
R & =\sum_{k=1}^{h}\sum_{j=1}^{n}x_{2j+3,g_{k}}x_{2j+4,g_{j}\cdot g_{k}}.
\end{align*}

For two generators $S=\{g_1,g_2 \}$, Frucht proves in \cite{frucht} that the automorphism group of the resulting cubic graph is isomorphic to $G$. However, for three generators this is not true in general. For instance, consider $G=A_5$ with generators $S=\{ g_1\coloneqq(1,5)(2,4), g_2\coloneqq(1,2,4,3,5), g_3\coloneqq(2,5,3)\}$. The automorphism group of the resulting cubic graph is isomorphic to $C_2\times A_5$ and thus does not let to a cubic graph with given automorphism group.
Note that the order of the generators matters in the construction by Frucht, i.e.\ in general reordering generators yields a non-isomorphic graph.

We modify Frucht's construction slightly to correct his proof and, in the case $|S|=2$, simplify the graph as illustrated in Figure \ref{frucht3} and Figure \ref{frucht2}. Both of these constructions are based on contracting 3-cycles.

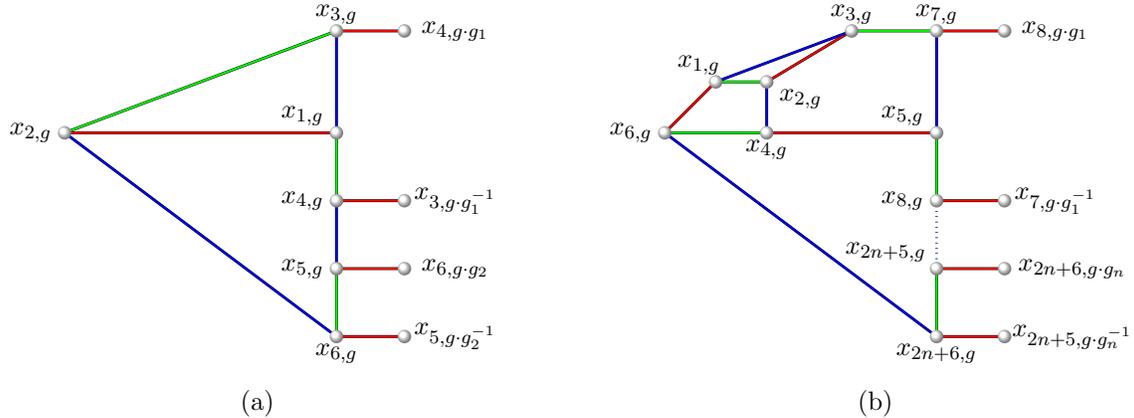
\begin{figure}[H]
\begin{minipage}{.49\textwidth} 
    \centering
    \resizebox{!}{5cm}{\input{frucht_graph_part2}}
    \subcaption{}
\label{frucht2}
\end{minipage}
\begin{minipage}{.49\textwidth} 
    \centering
    \resizebox{!}{5cm}{\input{frucht_graph_part3}}
    \subcaption{}
\label{frucht3}
\end{minipage}
\caption{(a) Simplified construction for two generators (b) Modified construction for more than two generators}
\end{figure}

In his proof, Frucht uses the notion of a cycle triplet of a node $v$ in a cubic graph $\Gamma$.

\begin{definition}
 Let $e_1,e_2,e_3$ be the edges incident to $v$ in $\Gamma$ and $c_{i,j}$ the length of a minimal cycle passing through the edges $e_i,e_j$. Then the cycle triplet of $v$ is defined by as the multi-set $\{c_{1,3},c_{2,3},c_{1,2}\}$.   
\end{definition} 

Next, Frucht shows that for a given finite group $G$ with generators $S$ the constructed graph $\Gamma_{G,S}$ has an automorphism group isomorphic to $G$ by using the following.

\begin{lemma}
For a given automorphism $\pi :\Gamma_{G,S}\to \Gamma_{G,S}$, nodes have to be mapped to nodes with the same cycle triplet, since cycles are mapped to cycles with same length by graph automorphisms.    
\end{lemma}

For $n=2$ resp.\ $n>2$ we can define the cubic graph obtained from the Cayley graph by substituting each node with a component of the form as shown in Figure \ref{frucht2} resp.\ Figure \ref{frucht3}.
We set $h=|G|$ and for $n=2$ resp.\ $n>2$ we define a graph with nodes corresponding to the variables $x_{i,g_j}$,with $j=1,\dots,h$ and $i=1,\dots,6$ resp.\ $i=1,\dots,2n+6$.

For $n=2$, we define a cubic graph $\Gamma_{G,S}$ by the quadratic forms
\begin{align*}
&Q_{simp}=\color{red}{Q_{2,1}}\color{black}{+}\color{ForestGreen}{Q_{2,3}}\color{black}{+} \color{blue}{Q_{1,3}} \color{black}{+}\color{ForestGreen}{Q_{1,4}}\color{black}{+}\color{blue}{Q_{2,6}}\color{black}{+} \color{ForestGreen}{Q_{5,6}} \color{black}{+}\color{blue}{Q_{4,5}} \color{black}{+}\color{red}{R_{simp}},\\ &R_{simp}=\sum_{k=1}^{h}x_{3,g_{k}}x_{4,g_{k}\cdot g_{1}}+x_{5,g_{k}}x_{6,g_{k}\cdot g_{2}}.    
\end{align*}

For $n>2$, the graph $\Gamma_{G,S}$ is given by the following quadratic forms
\begin{align*}
Q_{mod}= & {\color{red}{Q_{1,6}}{\color{black}{+}{\color{ForestGreen}{Q_{1,2}}{\color{black}{+}{\color{blue}{Q_{1,3}}{\color{black}{+}{\color{red}{Q_{2,3}}{\color{black}{+}{\color{ForestGreen}{Q_{3,7}}{\color{black}{+}{\color{blue}{Q_{2,4}}{\color{black}{+}{\color{red}{Q_{4,5}}{\color{black}{+}{\color{ForestGreen}{Q_{4,6}}}}}}}}}}}}}}}}}\\
 & {\color{black}{+}}{\color{blue}{Q_{5,7}}{\color{black}{+}{\color{ForestGreen}{Q_{5,8}}{\color{black}{+}{\color{blue}{Q_{6,2n+6}}{\color{black}{+}{\color{ForestGreen}{\sum_{i=2}^{n}Q_{2i+5,2i+6}}{\color{black}{+}{\color{blue}{\sum_{i=2}^{n}Q_{2i+4,2i+5}}{\color{black}{+}{\color{red}{R_{mod}}}}}}}}}}}}}\\
R_{mod}= & \sum_{k=1}^h \sum_{j=1}^{n} x_{2j+5,g_k}x_{2j+6,g_k\cdot g_j}.
\end{align*}

 In Figure \ref{frucht2}, both node $x_{1,g}$ and node $x_{2,g}$ have the same cycle triplet $\{3,5,6 \}$. However, node $x_{4,g}$ is incident to two nodes that lie on a three-cycle, i.e.\ $x_{1,g}$ and $x_{3,g\cdot g_1^{-1}}$, but node $x_{6,g}$ is only incident to a single node that lies on a three-cycle, i.e.\ node $x_{2,g}$. Thus, for an automorphism $\pi$ and a node $x_{i,g }$ we have that $\pi(x_{i,g })=x_{i,g'}$. Similarly, in Figure \ref{frucht3},  components are mapped onto component, since $x_{6,g}$ is the only node that has cycle triplet $\{4,6,2+2n \}$ and the rest follows, from the incidence structure of the underlying cubic graph. 
 The remaining part of the proof is analogous to the proof in given by Frucht in \cite{frucht} and is based on showing that the automorphism group of the cubic graphs are isomorphic to the automorphism group of the corresponding edge-coloured Cayley graphs. Altogether, we have the following Theorem due to Frucht.

\begin{theorem}\label{Frucht_cubic_theorem}
Let $G=\langle g_1,\dots,g_n \rangle$ be a finite group. Then there exists a cubic graph $\Gamma$ with $\Aut(\Gamma)\cong G$. Moreover, if $n=2$ or $n>2$ the graph $\Gamma$ has $6|G|$ or $(2n+6)|G|$ nodes, respectively.
\end{theorem}

Frucht's original cubic graphs can be 3-edge coloured with a function $c:E(\Gamma)\to \{r,g,b \}$ as follows, where we identify the quadratic form $Q$ with the edges of $\Gamma$:
\begin{align}\label{frucht_edge_colouring}
\begin{split}
& c\left(R\cup Q_{1,4}\cup Q_{2,3}\right)=r,\;
   c\left(Q_{1,5}\cup Q_{2,4}\cup Q_{3,6}\cup\bigcup_{i=2}^{n}Q_{2i+3,2i+4}\right)=g,\\
  & c\left(Q_{1,2}\cup Q_{3,5}\cup Q_{4,2n+4}\cup\bigcup_{i=2}^{n}Q_{2i+2,2i+3}\right)=b.
\end{split}
\end{align}

We can obtain 3-edge colouring for the simplified and modified graphs by contracting and adding 3-cycles, respectively. For our simplified graph, i.e.\ in the case for $n=2$ generators, this yields a 3-edge colouring $c_{simp}$ for $Q_{simp}$ as follows:
\begin{align*}
  c_{simp}\left(R_{simp}\cup Q_{2,1}\right)=r,\;
  c_{simp}\left(Q_{2,3}\cup Q_{1,4} \cup Q_{5,6}\right)=g,\;
  c_{simp}\left(Q_{1,3}\cup Q_{2,5}\cup Q_{4,5}\right)=b.
\end{align*}

It is straightforward to check that the quadratic forms $Q$ and $Q_{simp}$ define three-edge coloured cubic graphs with edge colouring $c$ and $c_{simp}$, respectively.

In Figure \ref{q_8_graph} an example of Frucht's original construction for $Q_8$ with two generators is shown.

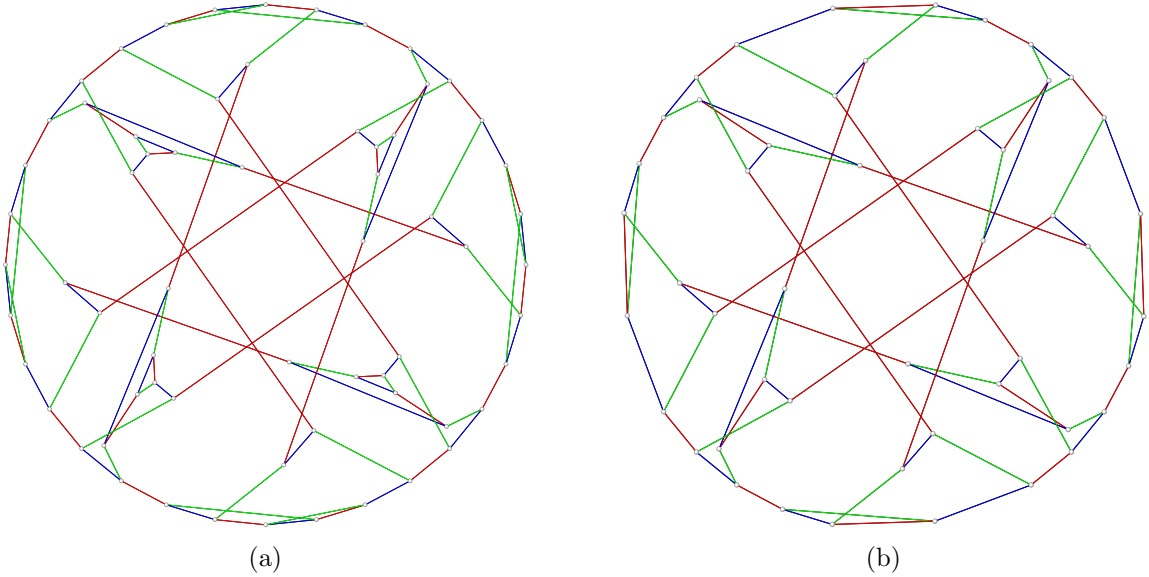
\begin{figure}[H]
\begin{minipage}{.49\textwidth}
    \centering
    \resizebox{!}{7cm}{\input{Q8}}
    \subcaption{}
    \label{q_8_graph}
\end{minipage}
\begin{minipage}{.49\textwidth}
    \centering
    \resizebox{!}{7cm}{\input{Q8_red}}
    \subcaption{}
    \label{Q8_red}
\end{minipage}
\caption{Cubic graphs with automorphism groups isomorphic to $Q_8$: (a)  Frucht's original construction on 64 nodes (b) Our modified construction on 48 nodes}
\end{figure}

The graph shown in Figure \ref{q_8_graph} has $8$ vertices of degree $3$ contracting these vertices corresponds to the simplification of Frucht's graph construction for groups generated by two elements. In Figure \ref{Q8_red}, we see the resulting cubic graph on $48$ nodes corresponding to our simplified construction with automorphism group isomorphic to $Q_8$. Iterating the process of contracting vertices of degree $3$ leads to a cubic graph with larger automorphism group on $32$ nodes.

\section{Cubic graphs with cyclic or dihedral automorphism group}\label{cyclic_dihedral_graph}
For a cyclic group $C_n$ with $n>2$ Frucht defines a graph $\Gamma_{C_n}$ with $6n$ vertices and cyclic automorphism group of order $n$. Let  $a_i,b_i,c_i,d_i,e_i,f_i$ for $i=1,\dots,n$ be indeterminates corresponding to the nodes of the graph $\Gamma_{C_n}$, where the edges are given by the quadratic form below:
$$Q_{C_n}=\sum_{i=1}^n (a_i b_i + a_i e_i + a_i f_i + b_i f_i + c_i d_i + c_i f_i + c_i e_i) +\sum_{j=1}^{n-1} (b_{j+1} e_j  + d_j d_{j+1}) + b_1 e_n  + d_1 d_n.$$

In Figure \ref{graph_c4} a planar embedding of this graph construction is shown for $n=4$.

\begin{theorem}[Frucht,\cite{frucht}]
For $n>2$, the quadratic form $Q_{C_n}$ defines a cubic graph $\Gamma_{C_n}$ with automorphism group isomorphic to $C_n$.
\end{theorem}

The dihedral groups $D_n$ are covered by Frucht's more general construction. However for $n\geq 4$, we define a cubic graph $\Gamma_{D_n}$ with $3n$ instead of $8n$ nodes and a dihedral automorphism group.
Let  $a_i,b_i,c_i,d_i$ for $i=1,\dots,n$ be indeterminates corresponding to the nodes of the graph $\Gamma_{D_n}$, where the edges are given by the following quadratic form:
$$Q_{D_n}=\sum_{i=1}^n ( a_i b_i  + a_i c_i + c_i d_i + b_i c_i) +\sum_{j=1}^{n-1} (a_j b_{j+1} + d_j d_{j+1}) + a_n b_1 + d_1 d_n.$$
\begin{theorem}
For $n\geq 4$, the automorphism group of the cubic graph $\Gamma_{D_n},$ defined by the above quadratic form, is isomorphic to $D_n$.
\end{theorem}
\begin{proof}
Let $\pi$ be an automorphism of $\Gamma_{D_n}.$
Note, that $\Gamma_{D_n}$ has exactly $n$ cycles of length 3, 
namely
\[
(a_1,b_1,c_1),\ldots,(a_n,b_n,c_n).
\]
We know, that the image of a cycle under $\pi$ has to be a cycle of the same length.
Therefore $\pi $ permutes the cycles of length 3 and there exists an $1\leq i \leq n$ with $\pi((a_1,b_1,c_1))=(a_i,b_i,c_i).$ 
Moreover, the vertices $d_1,\ldots,d_n$ must be permuted in such a way that $(\pi(d_1),\ldots,\pi(d_n))$ corresponds to the cycle $(d_1,\ldots,d_n).$ So, we conclude that $\pi(d_{1+j})=d_{i+j}$ or $\pi(d_{1+j})=d_{i-j}$, where $j \in \{1,\ldots,n\}$ and we read the subscripts modulo $n.$
If $\pi(d_{1+j})=d_{i+j}$, then $\pi(c_{1+j})$ is equal to $c_{i+j}$.
By examining the incidence structure of $\Gamma_{D_n}$, this results in $a_{i+j}=a_{i+j}$ and $b_{i+j}=b_{i+j}.$ So, $\pi$ can be written as the $i$-th power of the permutation $\omega=(a_1,\ldots,a_n)\ldots(f_1,\ldots,f_n).$ If $\pi(d_{1+j})=d_{i-j},$ we  deduce that $\pi$ maps $a_{i+j}$, $b_{i+j}$ resp.\ $c_{i+j}$ onto $b_{i-j}$, $a_{i-j}$ resp.\ $c_{i+j}$ by the same argument. 
Thus, $\pi$ is given by the product of $\omega^i$ and the involution $$s=\prod_{j=0}^{\lfloor\frac{n}{2}\rfloor} (a_{1-j},b_{1+j})(d_{1-j},d_{1+j})(c_{1-j},c_{1+j}) .$$ 
So, it follows that $\Aut(\Gamma_{D_n})=\langle \omega,s \rangle \cong D_n.$
\end{proof}

Next, we compute straight-line embeddings of the cubic graphs $\Gamma_{C_n}$ for $n\geq 3$ and $\Gamma_{D_n}$ for $n\geq 4$.  Let $n\geq 3$ be fixed. For simplicity we denote the set of vertices of the graph $\Gamma_{C_n}$ resp.\ $\Gamma_{D_n}$ by $V$ resp.\ $V'.$

We make use of the following remark to construct the corresponding embeddings of the above graphs.
\begin{remark}
Let $n\geq 3$  be a natural number. The permutation $\pi =(a_1,\ldots,a_n)\ldots(f_1,\ldots,f_n)$ gives rise to an automorphism of $\Gamma_{C_n}$ with $\pi^n=id.$ Thus, we conclude that 
$\Aut(\Gamma)=\langle \pi\rangle$ and that the automorphism group of the graph satisfies
$$\Aut(\Gamma_{C_n}) \cong \langle
    \underbrace{\left(\begin{array}{cc}
         \cos(\alpha)& -\sin(\alpha) \\
         \sin(\alpha)& \cos(\alpha)
    \end{array}\right)}_{M_\alpha\coloneqq}
\rangle \leq GL_2(\mathbb{R}),$$
where $\alpha=\frac{2\pi}{n}.$
\end{remark}

By assigning 2D-coordinates to the vertices in $V$, as described in the following remark, a desired embedding is constructed.
\begin{remark}
    Let $n\geq 3$ be a natural number and $v,w_1,w_2$ be defined by
    $v\coloneqq(1,0)^t,w_1\coloneqq(\cos(\alpha)-1,\sin(\alpha))^t,w_2\coloneqq(\sin(\alpha),1-\cos(\alpha))^t$. Furthermore, let $\phi: V\to \mathbb{R}^2$ be a function defined by
\begin{align*}
 &       \phi(a_1)=v+\frac{1}{2}w_1, \phi(b_1)=v+\frac{1}{4}w_1,\phi(c_1)=v+\frac{1}{2}w_1+2w_2,\\  
&\phi(d_1)=v+\frac{1}{2}w_1+3w_2, \phi(e_1)=v+\frac{3}{4}w_1,\phi(f_1)=v+\frac{1}{2}w_1+w_2,
    \end{align*}
    and 
    \begin{align*}
        &\phi (a_{i+1})={(M_\alpha)}^{i}\phi(a_1),
        \phi (b_{i+1})={(M_\alpha)}^{i}\phi(b_1),
        \phi (c_{i+1})={(M_\alpha)}^{i}\phi(c_1),\\&
       \phi (d_{i+1})={(M_\alpha)}^{i}\phi(d_1),
        \phi (e_{i+1})={(M_\alpha)}^{i}\phi(e_1),
        \phi (f_{i+1})={(M_\alpha)}^{i}\phi(f_1)
    \end{align*}
    where $\alpha=\frac{2\pi}{n}$ and $i \in \{1,\ldots,n-1\}.$
Then $\phi$ gives rise to a planar straight line embedding of $G.$
\end{remark}

Figure \ref{graph_c4} and \ref{graph_c7} show the constructed planar embedding for $n=4$ and $n=7.$
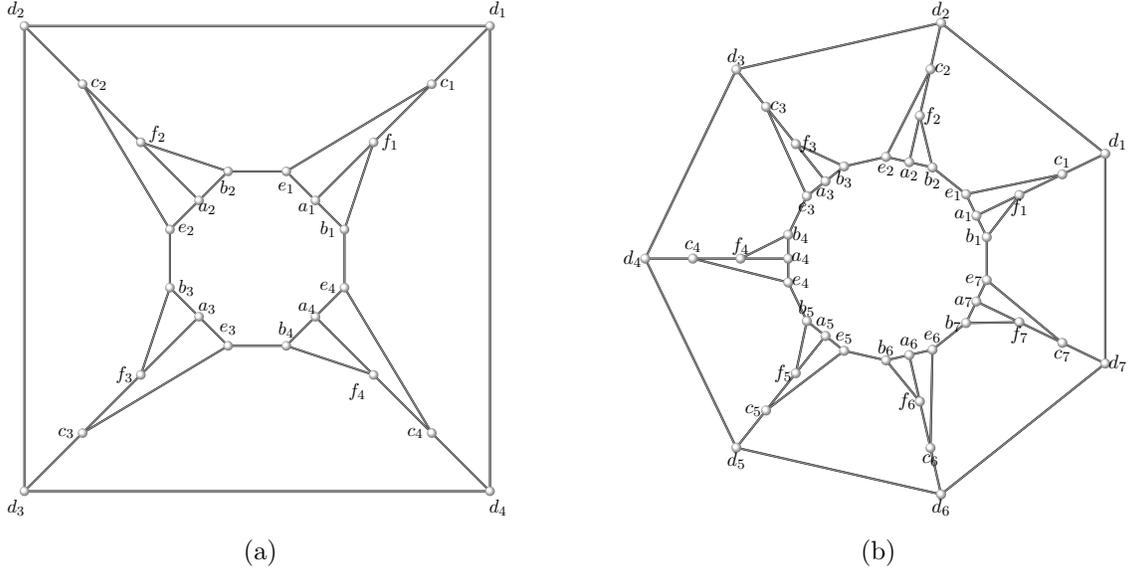
\begin{figure}[H]
\begin{minipage}{.49\textwidth}
    \centering
    \resizebox{!}{7cm}{\input{FacegraphC4}}
\subcaption{}
\label{graph_c4}
\end{minipage}
\begin{minipage}{.49\textwidth}
    \centering
    \resizebox{!}{7cm}{\input{FacegraphC7}}
\subcaption{}
\label{graph_c7}
\end{minipage}
\caption{(a) Planar embedding of $\Gamma_{C_4}$ (b) Planar embedding of $\Gamma_{C_7}$}
\end{figure}

Since the group $\langle M_{\alpha}\rangle$ can be embedded into the automorphism group of $\Gamma_{D_n},$ we can exploit the orthogonal matrix $M_{\alpha}$ to also construct a planar embedding of this graph, as seen in the cyclic case.

\begin{remark}
    Let $n\geq 4$ be a natural number and $v,w_1,w_2$ be defined by
 $ v\coloneqq(1,0)^t+,w_1\coloneqq(\cos(\alpha)-1,\sin(\alpha))^t,w_2\coloneqq(\sin(\alpha),1-\cos(\alpha))^t.$ Furthermore, let $\phi: V'\to \mathbb{R}^2$ be a function defined by
    \begin{align*}
  \phi(a_1)=v+\frac{3}{4}w_1, \phi(b_1)=v+\frac{1}{4}w_1,\phi(c_1)=v+\frac{1}{2}w_1+w_2,\phi(d_1)=v+\frac{1}{2}w_1+2w_2
    \end{align*}
    and 
    \begin{align*}
        & \phi (a_{i+1})={(M_\alpha)}^{i}\phi(a_1),
        \phi (b_{i+1})={(M_\alpha)}^{i}\phi(b_1),
        \phi (c_{i+1})={(M_\alpha)}^{i}\phi(c_1),
       \phi (d_{i+1})={(M_\alpha)}^{i}\phi(d_1)
    \end{align*}
    where $\alpha=\frac{2\pi}{n}$ and $i \in \{2,\ldots,n\}.$
Then $\phi$ gives rise to a planar straight line embedding of $G.$
\end{remark}

Figure \ref{graph_d4} shows the constructed planar embedding for $n=4$ and Figure \ref{graph_d7} shows the constructed embedding for $n=7.$

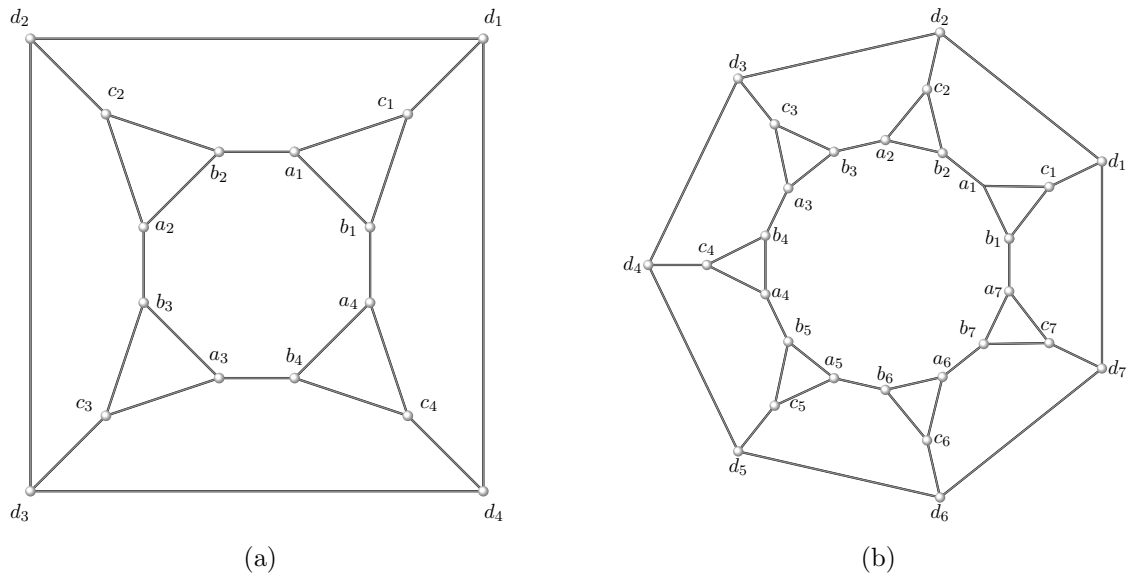
\begin{figure}[H]
\begin{minipage}{.49\textwidth}
    \centering
    \resizebox{!}{7cm}{\input{FacegraphD4}}
\subcaption{}
\label{graph_d4}
\end{minipage}
\begin{minipage}{.49\textwidth}
    \centering
    \resizebox{!}{7cm}{\input{FacegraphD7}}
\subcaption{}
\label{graph_d7}
\end{minipage}
\caption{(a) Planar embedding of $\Gamma_{D_4}$ (b) Planar embedding of $\Gamma_{D_7}$}
\end{figure}

Note, the embeddings of the graphs in Figure \ref{a_5_graph_cayley} and Figure \ref{a_5_graph_orb} were computed using a linear system of equation that gives rise to the \emph{Tutte-Embedding}. In \cite{tutte_embedding}, Tutte introduces for a planar 3-regular graph $\Gamma$ a crossing-free embedding of $\Gamma$ into the Euclidean plane, such that
\begin{itemize}
    \item the outer face of the graph forms a convex polygon, and 
    \item each vertex is positioned at the barycentre of its neighbouring vertices.
\end{itemize}
Tutte proves that the linear system of equations that arises from enforcing the above properties on the embedding of $\Gamma$ has a unique solution.
As an example, consider the graph $\Gamma=(V,E)$ given by 
\begin{align*}
    V=\{&1,2,3,4,5,6,7,8\}\\
    E=\{& \{1,2\},\{1,4\},\{1,5\},\{2,3\},\{2,6\},\{3,4\},\\
   & \{3,7\},\{4,8\},\{5,6\},\{5,8\},\{6,7\},\{7,8\}\}
\end{align*}
which forms the face graph of an octahedron. Enforcing the face given by the vertices $\{1,2,3,4\}$ to form a square, given by the 2D-coordinates $\left( 0,0\right)^t, \left( 0,1\right)^t, \left( 1,1\right)^t,\left( 1,0\right)^t,$ gives rise to the embedding identified by the following ordered list of 2D-coordinates: 
\begin{align*}
    \left[
     \left( 0,0\right)^t, \left( 0,1\right)^t, \left( 1,1\right)^t, \left( 1,0\right)^t,\left( \frac{1}{3},\frac{1}{3}\right)^t,\left( \frac{1}{3},\frac{2}{3}\right)^t,\left( \frac{2}{3},\frac{2}{3}\right)^t,\left( \frac{1}{3},\frac{2}{3}\right)^t
    \right].
\end{align*}

Here, the coordinate of vertex $i$ in the Tutte Embedding is given by the $i$-th entry in the above list, see Figure \ref{facegraph_octahedron}.
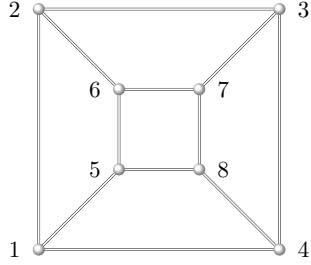
\begin{figure}[H]
\begin{minipage}{\textwidth}
    \centering
    \scalebox{0.8}{\input{facegraph_octahedron}}
\end{minipage}
\caption{Planar embedding of the face graph of the octahedron}
\label{facegraph_octahedron}
\end{figure}
For non planar graphs, see for instance Figure \ref{a_5_graph_cayley}, the existence of an unique embedding with all the vertices positioned at the barycentre of their neighbouring vertices is not established. In the figures showing non-planar graphs, the system of linear equations is still solvable and yields a highly symmetric embedding of the given graphs with low crossing numbers.

\section{Cubic Vertex-Transitive Graphs} \label{vertex-transitive}

Frucht's construction yields a cubic graph that has a given group as automorphism group. More precisely, for a given group $G$ generated by two resp.\ $n$ elements we obtain a cubic graph $\Gamma$ on $6|G|$-nodes resp.\ $(2n+6)|G|$ nodes with $\Aut(\Gamma)\cong G$. In general, these graphs are not $G$-vertex-transitive, i.e.\ $G$ does not act transitively on the vertices of $\Gamma.$ Therefore, we  focus on the construction of $G$-vertex-transitive graphs in this section. In the mathematical literature such graphs are also known as (generalised) orbital graphs, see \cite{lauri_scapellato_2016}.

If a group $G$ acts transitively on a given vertex-transitive graph $\Gamma$ it can be shown that $\Gamma$ is isomorphic to a Schreier coset graph, such that the nodes of $\Gamma$ correspond to the right-cosets $G/G_v$, where $G_v$ is the stabilizer of an arbitrary node in $\Gamma$. Since such a orbital graph has exactly $|G|/|G_v|=|G/G_v|$ vertices, it follows that this graph has considerably less nodes than the graph based on Frucht's construction. Note that, if $G_v$ is trivial the graph $\Gamma$ is a Cayley graph. In the following two sections, we distinguish between the two cases $G_v\cong \{ 1 \}$ and $G_v \not \cong \{ 1 \}$.
 
For a subgroup $H$ of a given group $G$ we define the generalized orbital graph as follows:

\begin{definition}
    Let $G$ be a finite group and $H\leq G$ be a subgroup of $G$. Let $S=\{g_1,\dots,g_n\}\subset G$. Then the graph with nodes $G/H$ and edges $\bigcup_i \{H,g_iH\}^G$ is called a generalized orbital graph, denoted by $\Gamma_{G/H,S}$. If $H=\{\text{id}\}$ we call it a Cayley graph. A Cayley graph with automorphism group isomorphic to $G$ is called a graphic regular representation (short GRR).
\end{definition}

\subsection{Cubic Cayley Graphs}

Here, we discuss possible constructions based on cubic Cayley graphs. Moreover, we illustrate these constructions for the group $G=A_5$ as an example.
\begin{theorem}[\cite{XIA2022256}]
    Except for a finite number of cases all finite non-abelian simple groups have a cubic GRR.
\end{theorem}

Combined with the following conjectures we can find cycle double covers of cubic GRRs by 3-edge colourings.

\begin{conjecture}[\cite{Hujdurović2014}]\label{conj_cayley_3edge}
    Every cubic Cayley graph admits a 3-edge colouring.
\end{conjecture}

Below, we see that Conjecture \ref{conj_cayley_3edge} can be easily verified in many cases. Alternatively, we can obtain cycle double covers of a Cayley graph as shown below.

\begin{remark}
    \begin{enumerate}
    \item Let $S=\{s_1,s_2,s_3\}\subset G$ be a set of three distinct involutions such that $G=\langle S \rangle$. By colouring the edges accordingly to the corresponding involutions, we see that the Cayley graph given by $G$ and $S$ is a cubic 3-edge colourable graph.
    \item Let $S=\{s,x,x^{-1}\}\subset G$ be a set of an involution $s$ and an element $x$ with even order such that $G=\langle S \rangle$. Then the Cayley graph given by $G$ and $S$ is 3-edge colourable by colouring the edges corresponding to the involution $s$ in one colour and the edges obtained by the even cycles of $x$ in alternating colours. More general, we can define a cycle double cover, via the cycles generated by $x$ and $sx$:
    for $g\in G$ we have that $g$ lies on the cycle $(g,g\cdot x,g\cdot x^2,\dots,g\cdot x^{-1})$, where $x^{-1}=x^{|x|-1}$ and on the cycle $(g,g\cdot s,g\cdot s\cdot x,\dots,g\cdot sx^{-1},g\cdot sx^{-1}\cdot s)$. Note that the union of these cycles for all elements $g\in G$ yields a cycle double cover invariant under the action of $G$.
    \end{enumerate}
\end{remark}

Let $G=A_5$ and $S=\{(1,2)(4,5),(1,5)(3,4),(1,5)(2,4) \}$. Then we have that $G=\langle S \rangle$ and we can verify that the Cayley graph $\mathcal{C}_{A_5,S}$ is a GRR, shown in Figure \ref{a_5_graph_cayley}. The 3-edge colouring obtained as in the Remark above leads to a simplicial surface with Euler characteristic $1$.

\subsection{Cubic Orbital Graphs}

Next, we focus on the case that $|G_v|>1$.  We construct a cubic vertex-transitive graph $\Gamma$ on $|G|/|H|$ nodes for subgroups $H\leq G$ giving rise to candidates for small cubic graphs with given automorphism group. However, in general we have that $G<\Aut(\Gamma)$. 

For generalized orbital graphs we need to be more careful since for $1\neq H\leq G$ it is possible that for an element $g\in G$, the orbit $\{H,gH \}^G$ contains other elements of the form $\{H,g'H \}$ for $g'\neq g\in G$. A (generalized) orbital graph is defined by the data $H\leq G$ and $S=\{g_1,\dots,g_n \}\subset G$ with vertices $G/H$ and edges given by the orbit of $G$ (with element-wise actions) on the set $\{\{H,g_1H\},\dots,\{H,g_nH\} \}.$ Let $g_1,\dots,g_m$ with $m$ minimal such that $\{\{H,g_1H\},\dots,\{H,g_nH\} \}^G=\{\{H,g_1H\},\dots,\{H,g_mH\} \}$. Then the resulting orbital graph is regular of degree $m$. We can use this construction, to obtain all vertex-transitive graphs $\Gamma$ with $G$ acting transitive on the nodes of $\Gamma$. It is still unknown whether any vertex-transitive graph $\Gamma$ admits a cycle double cover.
\begin{remark}
    Let $\Gamma$ denote a cubic vertex-transitive graphs as above. 
    If the edges of $\Gamma$ are given by three orbitals, we can define a 3-edge colouring by colouring each orbital in a different colour and thus obtain a cycle double cover by alternating colour-cycles. If the edges are given by less than three orbitals, it is no longer guaranteed that we obtain a 3-edge colourable graph.

\end{remark}

As an example, we consider the subgroup $H=\langle (2,3)(4,5)\rangle$ of $A_5=\langle S\rangle,$ with \\$S=\{(1,5)(3,4),(1,4,2,3,5) \}$. We obtain a cubic orbital graph $\Gamma$ on $30$ nodes with automorphism group isomorphic to $A_5$ which admits exactly one 3-edge colouring (up to interchanging colours) and thus it is not a snark, see Figure \ref{a_5_graph_orb}. By considering all vertex-transitive graphs of $A_5$, we find that this graph is the smallest vertex-transitive cubic graph with automorphism group isomorphic to $A_5$. Furthermore a $G$-invariant cycle double cover consisting of $10$ cycles of length $6$ and $6$ cycles of length $5$ which is respected by its automorphism group is given as follows:
{\small \begin{align*}
    (14,27,28,24,19,17), (9,30,25,18,13,10), (8,26,29,20,23,11), (7,15,16,21,22,12), (3,28,27,12,7,4), \\
  (3,13,18,23,20,4),  (2,11,8,17,14,5), (2,25,30,21,22,5),  (1,9,10,24,19,6), (1,16,15,29,26,6), \\
  (3,28,24,10,13),(2,25,18,23,11),(1,16,21,30,9), (6,26,8,17,19), (5,22,12,27,14), (4,20,29,15,7)
\end{align*}}

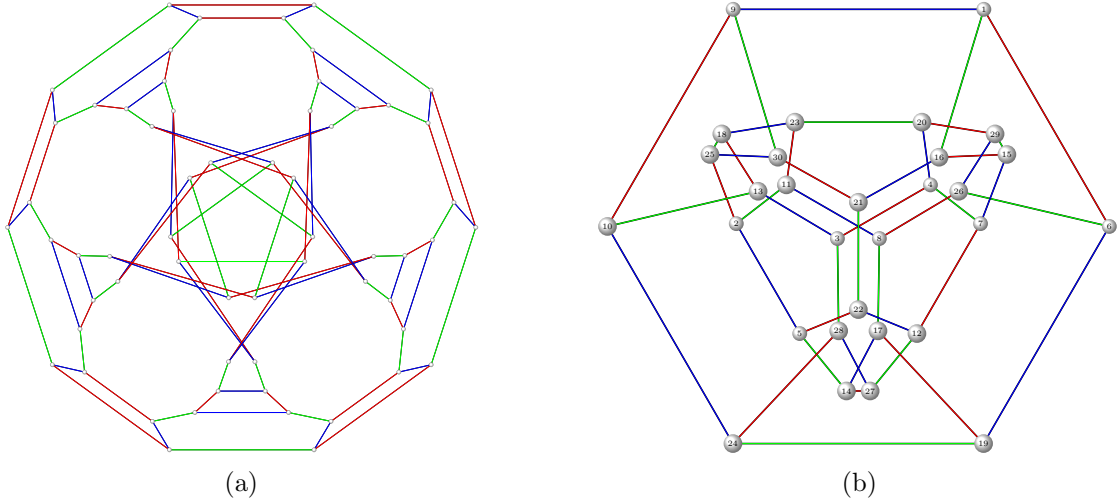
\begin{figure}[H]
\begin{minipage}{.49\textwidth}
    \centering
    \resizebox{!}{6cm}{\input{A5_60_31}}
    \subcaption{ }
\label{a_5_graph_cayley}
\end{minipage}
\begin{minipage}{.49\textwidth}
    \centering
    \resizebox{!}{6cm}{\input{A5_16_2}}
    \subcaption{}
\label{a_5_graph_orb}
\end{minipage}
\caption{Two vertex-transitive graphs with automorphism group $A_5$ (a) 3-edge colourable GRR (b) Orbitalgraph on $30$ nodes}
\end{figure}

\section{Cycle Double Covers and Simplicial Surface Constructions}\label{simplicial_surface_construction}

In this section, we show that we can obtain a simplicial surface with a given group $G$ as automorphism group. In both, Fruchts construction and the construction of cubic generalized orbital graphs, we achieve this by computing a $G$-invariant cycle double cover. Such a cycle double cover can be constructed by exploiting 3-edge colourings or the structure of the group $G$. Here, the number of faces (triangles) of the resulting surface equals the number of nodes of the given cubic graph. Furthermore, we can obtain smaller (in the number of nodes) surfaces by considering vertex-faithful surfaces for certain groups.

\subsection{Frucht surfaces}

We are now ready to prove the following.

\begin{theorem}
    Let $G$ be a finite group generated by a set $S$. There exists a simplicial surface $X_{G,S}$ such that $Aut(X_{G,S})\cong G$.
\end{theorem}

\begin{proof} 
From Theorem \ref{Frucht_cubic_theorem}, we obtain a cubic graph $\Gamma=\Gamma_{G,S}$ with $\Aut(\Gamma)\cong G$. Moreover, $\Gamma$ has a 3-edge colouring as shown in \eqref{frucht_edge_colouring}. This yields a cycle double cover by considering all cycles defined by alternating between two colours, see \cite{szekeres}.  We obtain a simplicial surface $X_{G,S}$, where faces and edges are given by the graph and vertices by the cycles such that $\mathcal{F}(X_{G,S})=\Gamma$ and $\Aut(X_{G,E})\leq \Aut(\Gamma)$. In order to show equality it suffices to show that the given cycle double cover is $G$-invariant. The automorphism group of the graph $\Aut(\Gamma)$ is isomorphic to $G$ via the automorphisms of type $$\sigma_m (x_{i,g_k})=x_{i,g_m g_k },$$ for $m=1,\dots,h=|G|.$ We have that $\sigma_m(Q_{i,j})=Q_{i,j}$ and for $|S|=2$ resp.\ $|S|>2$ we have $\sigma_m(R_{simp})=R_{simp}$ resp.\ $\sigma_m(R_{mod})=R_{mod}$. Hence, the cycle double cover is invariant under $\sigma_m$ and thus vertices of $X_{G,S}$ are mapped onto each other. It follows that $$\Aut(X_{G,S}) \cong G .$$
\end{proof}

For the surface construction based on the 3-edge coloring of the graphs given in Figure \ref{frucht2} and Figure \ref{frucht3}, we can give explicitly give the vertex degrees corresponding to the cycle-lengths.

\begin{remark}
Let $G$ be a finite group with $n$ generators. Then we can give the length of the cycles in the cycle double cover above as follows:
For $n=2$ or $n>2$ the blue-green cycles yield $G$ cycles of length  $6$ or $6+2n$, respectively.  For the blue-red cycles we have  $|G|/\langle g_1\cdot g_2\rangle$ cycles of length  $|g_1\cdot g_2|\cdot 6$ for $n=2$ and $|G|/\langle g_1\cdot \dots \cdot g_n\rangle$ cycles of length $|g_1\cdot \dots \cdot g_n|\cdot 6$ for $n>2$.
The green-red cycles yield $|G|/\langle g_i\rangle$ cycles of length $|g_i|\cdot 2$ for $i=2,\dots,n$ and $|G|/\langle g_1\rangle$ cycles of length $|g_1|\cdot 4$.
In total we have the vertex counters
$$ v_{6}^{|G|} v_{|g_1\cdot g_2|\cdot 6}^{|G|/| g_1\cdot g_2|}  v_{|g_1|\cdot 4}^{|G|/| g_1|} v_{|g_2|\cdot 2}^{|G|/| g_2|},$$ for $n=2$ and

$$v_{2n+6}^{|G|} v_{|g_1\cdots g_n|\cdot 6}^{|G|/|g_1\cdots g_n|}v_{|g_1|\cdot 4}^{|G|/| g_1|} \prod_{i=2}^{n} v_{|g_i|\cdot 2}^{|G|/| g_i|},$$ for $n>2$.
\end{remark}

The surface obtained from the 3-edge colouring are not vertex-faithful, i.e.\ do not describe a simplicial complex in the usual sense, since the blue-red cycle shares at least two edges with every green-red cycle. Likewise the blue-green cycle shares all its blue-edges with the blue-red cycle. 

However, we can define an alternative cycle double cover leading to a vertex-faithful surface in many cases

\begin{remark}
Let $G$ be a group with generators $\{g_1,\dots,g_n \}$ such that for all $i=1,\dots,n$ we have that $\langle g_i \rangle \cap \langle g_1 \cdots g_n \rangle = 1$ is the trivial group. Considering Figure \ref{frucht2}, we can define a vertex-faithful surface $X_G$ with given automorphism group $G$. For this we give a cycle double cover of the graphs given in Section \ref{fruchts_graphs} and show that the group of graphs automorphisms which are isomorphic to $G$ leave this cover invariant. Furthermore, the cover has the additional property that no two cycles share two edges. For the rank $2$ case, the cover is given as follows:
\begin{align*}
    \big(x_{1,g},x_{2,g},x_{3,g}\big),
\big(x_{1,g},x_{2,g},x_{6,g},x_{5,g},x_{4,g}\big),
\end{align*}
for all $g\in G$ and
\begin{align*}    \big(x_{6,g},x_{5,g},x_{6,g\cdot g_2},x_{5,g\cdot g_2},\dots,x_{6,g\cdot g_2^{|g_2|-1}},x_{5,g\cdot g_2^{|g_2|-1}}\big),
\end{align*}\ for all orbit representatives $g$ of the right action of $\langle g_2 \rangle$ on $G$,
\begin{align*}    \big(x_{4,g},x_{1,g},x_{3,g},x_{4,g*g_1},x_{1,g\cdot g_1},x_{3,g\cdot g_1},\dots,x_{4,g\cdot g_1^{|g_1|-1}},x_{1,g\cdot g_1^{|g_1|-1}},x_{3,g\cdot g_1^{|g_1|-1}}\big),
\end{align*}
 for all orbit representatives $g$ of the right action of $\langle g_1 \rangle$ on $G$, and the cycles 
 \begin{align*}   \big(x_{6,g},x_{2,g},x_{3,g},x_{4,g\cdot g_1},x_{5,g\cdot g_1},x_{6,g\cdot g_1\cdot g_2},\dots,x_{6,g}\big),
 \end{align*}
  for all orbit representatives $g$ of the right action of $\langle g_1\cdot g_2 \rangle$ on $G$. It is straightforward to show that automorphisms of the underlying graph leave this cycle basis invariant since elements of the form $x_{i,g}$ are mapped onto elements of the form $x_{i,g\cdot h}$ for $h\in G$. We can also describe the vertex-counter of the resulting surface as follows: 
\begin{align*}
    v_3^{|G|} v_5^{|G|} v_{2|g_2|}^{|G|/|g_2|} v_{4|g_1|}^{|G|/|g_1|} v_{5|g_1 g_2^{-1}|}^{|G|/|g_1 g_2|}.
\end{align*}    
\end{remark}

For the rank $n>2$ case we can analogously define a cycle double cover and obtain the following result.
\begin{theorem}
    For every finite group $G=\langle g_1,\dots,g_n\rangle$ there exists a simplicial surface $X$ with $\Aut(X)\cong G$. Furthermore, $X$ is vertex-faithful (a simplicial complex) if for all $i=1,\dots,n$ we have that $\langle g_i \rangle \cap \langle g_1\cdots g_n \rangle = \{1 \}$.
\end{theorem}

In some cases, we can give alternative cycle double covers of Frucht's graph to obtain a vertex-faithful surface.  For example, consider the quaternion group $Q_8=\langle i,j| i^4=j^4=1, i^2=j^2, ij=j^{-1} i\rangle$ with generators $g_1\coloneqq i,g_2\coloneqq j$. In \cite{babai}, it is shown that no graph with automorphism group $Q_8$ can be embedded onto a sphere. Below, we describe a cycle double cover of the graph shown in Figure \ref{Q8_red} yielding a vertex-faithful simplicial surface of Euler characteristic $0$:
\begin{align*}    \big(x_{1,g},x_{2,g},x_{3,g}\big),
\big(x_{1,g},x_{2,g},x_{6,g},x_{5,g},x_{4,g}\big),
\end{align*}
for all $g\in Q_8$ and
\begin{align*}
\big(x_{3,g},x_{1,g}, x_{4,g}, x_{3,g\cdot i^{-1}}, x_{2,g\cdot i^{-1}}, x_{6,g\cdot i^{-1}}, x_{5,g\cdot i^{-1}\cdot j^{-1}}, x_{6,g\cdot i^{-1}\cdot j^{-1}}, x_{5,g\cdot i^{-1}\cdot j^{-1}\cdot j^{-1}},x_{4,g\cdot i^{-1}\cdot j^{-1}\cdot j^{-1}}\big),
\end{align*}
for all $g\in Q_8$ (note that some elements $g\in Q_8$ are yielding the same cycle).

\subsection{Surface with Cyclic or Dihedral Automorphism Group}

In Section \ref{cyclic_dihedral_graph}, we introduced cubic graphs with cyclic or dihedral automorphism group. Here, we present cycle double covers for the two families of cubic graphs giving rise to spherical surfaces. Moreover these construction show that the underlying the constructed cubic graphs are planar.
\begin{theorem}
For $n>2$, there exists a vertex-faithful simplicial sphere $X_{C_n}$ with face graph $\Gamma_{C_n}$ and automorphism group isomorphic to $C_n$. Furthermore, the graph $\Gamma_{C_n}$ is a 3-connected planar cubic graph.
\end{theorem}

\begin{proof}

We fix $n>2$. For $i=1\ldots,n$, we define a cycle double cover of the graph $\Gamma_{C_n}$ as follows:
$$(a_i,b_i,f_i),(a_i,f_i,c_i,e_i),(e_i,c_i,d_i,d_{i+1},c_{i+1},f_{i+1},b_{i+1}),(d_1,\dots,d_n),(b_1,a_1,e_1,b_2,a_2,e_2,\dots,b_n,a_n,e_n),$$ with addition modulo $n$.
The automorphism group of the graph  is leaving the cycle cover invariant and thus the resulting surface has the same automorphism group as their underlying face graph. We obtain a surface $X_{C_n}$ with automorphism group isomorphic to $C_n$.
The corresponding vertex-counter of $X_{C_n}$ is then given by 
$$v_3^nv_4^nv_7^nv_nv_{3n}.$$
It follows that the surface $X_{C_n}$ is spherical, since its Euler characteristic is given by $$(n+n+n+2)- 9n + 6n=2$$
and thus $X_{C_n}$ is a vertex-faithful simplicial sphere. Steinitz Theorem says that a 3-connected cubic graph is planar if and only it is the face graph of a simplicial sphere. It follows that the graphs $\Gamma_{C_n}$ is planar.
\end{proof}

For $n>3$ we give an analogous result for the dihedral group $D_n$.

\begin{theorem}
For $n>3$, there exists a vertex-faithful simplicial sphere $X_{D_n}$ with face graph $\Gamma_{D_n}$ and automorphism group isomorphic to $D_n$. Furthermore, the graph $\Gamma_{D_n}$ is a 3-connected planar cubic graph.
\end{theorem}

\begin{proof}
    Analogously to the proof above, we construct a cyclic double cover
    $$(a_i,b_i,c_i),(a_i,c_i,d_i,d_{i+1},c_{i+1},f_{i+1},b_{i+1}),(d_1,\dots,d_n),(a_1,b_1,a_2,b_2,\dots,a_n,b_n),$$ for $i=1,\dots,n$ with addition modulo $n$. We obtain a simplicial surface $X_{D_n}$ with vertex counter $$v_3^nv_6^nv_nv_{2n}$$ and due to $$4n+(n+n+2)-6n=2$$ it follows that $X_{D_n}$ is a vertex-faithful simplicial sphere and thus the graph $\Gamma_{D_n}$ is planar.
\end{proof}

\subsection{Face-Transitive Surfaces}
The cubic graphs that arise from Frucht's construction yield simplicial surfaces by introducing suitable cycle double covers. Note, that in general the automorphism group of such a surface is a proper subgroup of the automorphism group of the underlying face graph. Thus, we do not necessarily obtain face-transitive surfaces, i.e.\ surfaces whose automorphism group acts transitive on the set of faces of the surface. However, for a graphic regular representation of a given group $G$ we can always associate a face-transitive surface. Here, we give a example construction for $G=A_5$ that leads to a surface $X_G$ based on such a GRR with $Aut(X_G)=G$.

In \cite{Hoffman1991} the authors show that Cayley graphs posses a cycle double cover. In the cubic case, this can be linked to 3-edge colourings as follows:
\begin{remark}
    Let $G$ be a finite group with generators $S$ and assume that $\mathcal{C}_{G,S}$ is a cubic Cayley graph. Then one of the following cases is true:
    \begin{enumerate}
        \item The group $G$ is generated by three distinct involutions $S=\{s_1,s_2,s_3\}$ corresponding to a 3-edge colouring (each involution has a distinct colour). As before, we can obtain a cycle double cover of $\mathcal{C}_{G,S}$ by alternating between colours.
        \item The group $G$ is generated by one involution $s$ and another element $x$ with $|x|>2$. We have seen that if $|x|$ is even, we can obtain a 3-edge colouring of $\mathcal{C}_{G,S}$. However, we can always obtain a cycle double cover via the cycles $(g,g\cdot x,\dots,g\cdot x^{|x|-1})$ and $(g,g\cdot s, g \cdot sx,\dots,g \cdot (sx)^{|sx|-1},g \cdot (sx)^{|sx|-1}\cdot s)$, for $g\in G.$
    \end{enumerate}
    
\end{remark}

In both cases the defined cycle double covers are invariant under the action of $G\leq \Aut(\mathcal{C}_{G,S})$ and in the case $\Aut(\mathcal{C}_{G,S})\cong G$ we always obtain a simplicial surface $X_{G,S}$ with $|G|$ faces and automorphism group isomorphic to $G$.

\begin{remark}
 To obtain a simplicial surface from a 3-edge coloured Cayley graph we can use the same construction as in the previous chapters. Note that, for involutions $s_1,s_2,s_3$, the length of the cycles are given by $\lvert s_i \cdot s_j \rvert \cdot 2 $. In the case that $S=\{s,x,x^{-1}\}$, with $|s|=2$ and $|x|=2n$ for $n>1$ the lengths of the cycles are given by $|x|,|s\cdot x|\cdot 2, |s\cdot x^{-1}|\cdot 2$.
\end{remark}

We can give a criterion whether or not the $X_{G,S}$ is vertex-faithful.

\begin{theorem}
    The surface $X=X_{G,S}$ is vertex-faithful if and only if
    \begin{enumerate}
        \item We have $S=\{s_1,s_2,s_3\}$ with three distinct involutions or
        \item $S=\{s,x,x^{-1} \}$ as above and $\langle sx \rangle \cap \langle x \rangle = \{1 \}.$
    \end{enumerate}
\end{theorem}

\begin{proof}
    Let $\Gamma_{S}$ be a Cayley graph $\mathcal{C}_{G,S}$. The surface $X$ is vertex-faithful if and only no two cycles share two edges. Since the Cayley graph is vertex-transitive we can assume that the identity element lies on such an edge. We need to proof that no cycles share two common edges. Let $S=\{s_1,s_2,s_3\}$ and and assume that the cycles corresponding to $s_1,s_2$ and $s_1,s_3$ share two common edges. It follows that  $|\langle s_1,s_2 \rangle \cap \langle s_1,s_3 \rangle|>\langle s_1 \rangle $ and hence $\langle x,y \rangle = \langle x,z \rangle$ contradicting the assumption that the three involutions are distinct.
\end{proof}

Using the techniques from the previous sections we get a surface with Euler Characteristic $1$ and automorphism group isomorphic to $A_5$. Its facegraph is shown in Figure \ref{a_5_graph_cayley}. For general vertex-transitive graphs which are not Cayley graphs it is in general an open problem to determine a cycle double cover. In the previous section, we see an example in Figure \ref{a_5_graph_orb} of such a graph yielding a surface on $30$ vertices with automorphism group isomorphic to $A_5$.

\subsection{Surfaces with Automorphism Group  \texorpdfstring{$A_5$}{Alternating Group of Degree 5}}
In the previous sections, three surfaces with automorphism group isomorphic to $A_5$ are shown. The corresponding face graphs are given in Figure \ref{petersen}, Figure \ref{a_5_graph_cayley} and Figure \ref{a_5_graph_orb}. The resulting surfaces are all face-transitive and of Euler characteristic $1$. In this section, we see that there exists a simplicial sphere with automorphism group isomorphic to $A_5$ together with an embedding with equilateral triangles such that resulting automorphism group in $\mathrm{O}(3)$ is isomorphic to $A_5$.
The well-known Snub dodecahedron, an Archimedean solid, has an automorphism group isomorphic to $G=A_5$ and its vertex-edge graph is isomorphic to a Cayley graph of $A_5$, obtained via the generators $S=\{(1,2)(3,4), (1,2,3,4,5), (1,3,5)\}$. Replacing the pentagonal faces with $5$ triangles, see kis operator, see \cite{conway2008}, we get a triangulated sphere with automorphism group isomorphic to $A_5$, called pentakis snub dodecahedron, see Figure \ref{snub_dodecahedron}. In Figure \ref{a_5_graph_snub} we see the facegraph of the pentakis snub dodecahedron, a planar cubic graph with automorphism group isomorphic to $A_5$.

\begin{figure}[H]
\begin{minipage}{.49\textwidth}
    \centering
    \includegraphics[height=4cm]{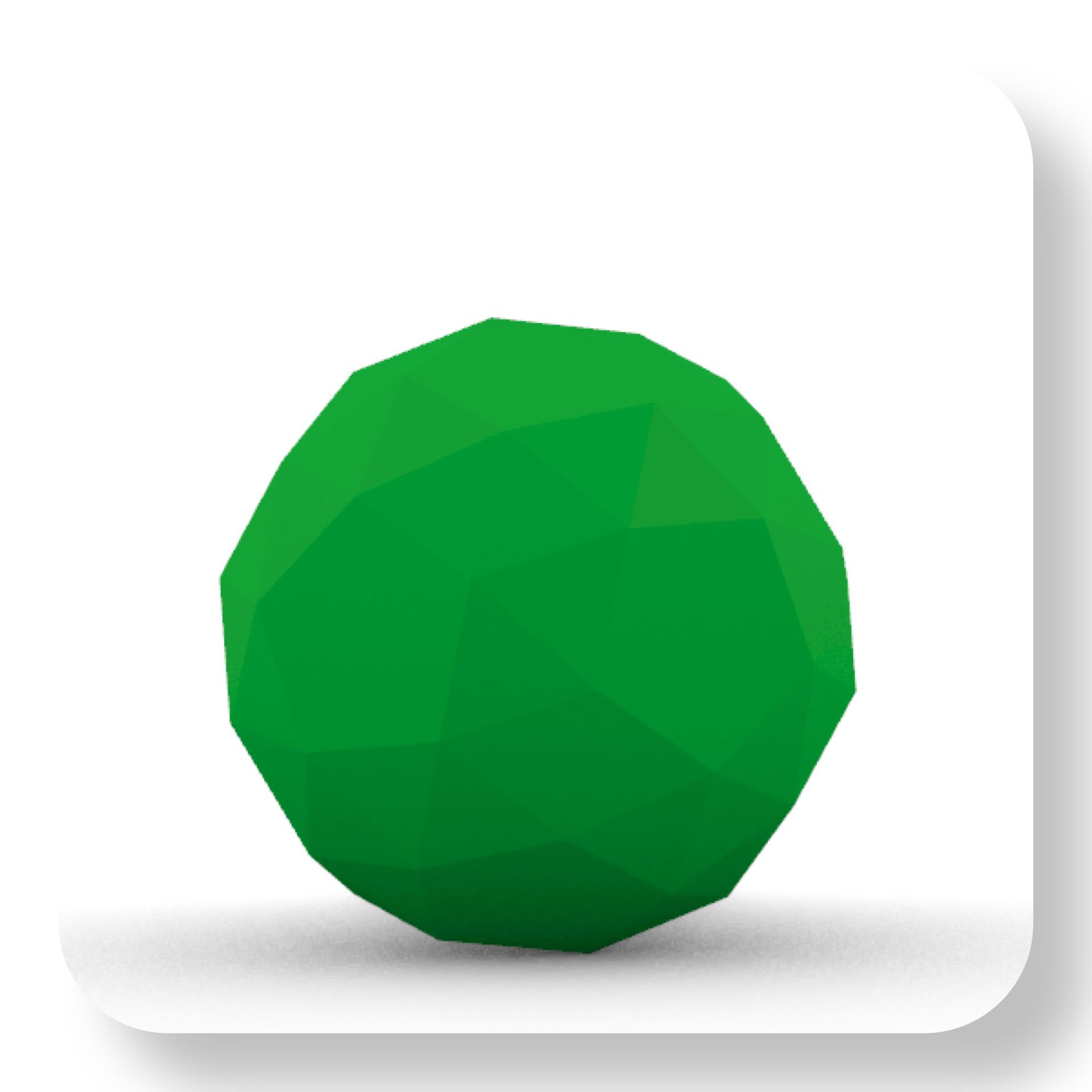}
    
    \includegraphics[height=4cm]{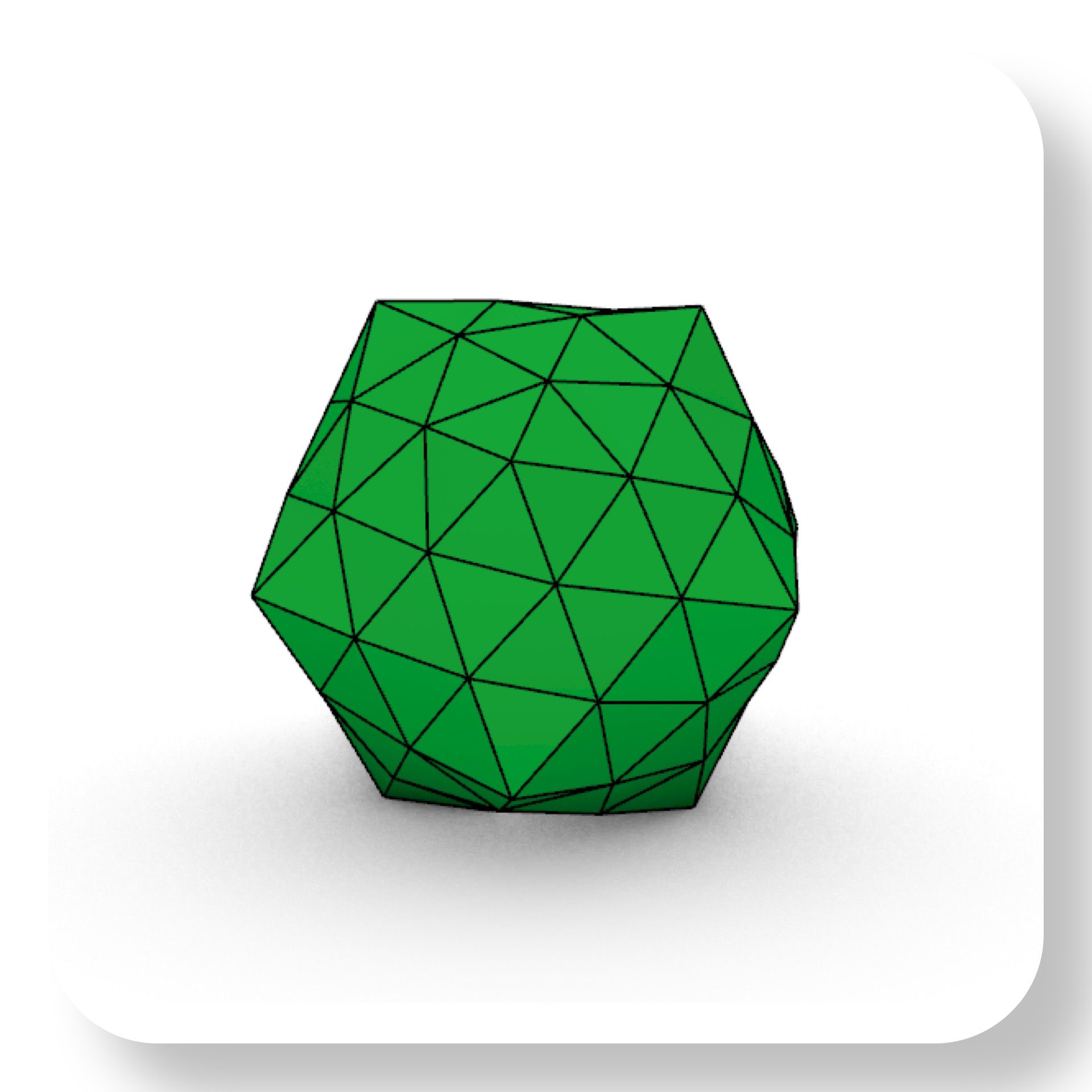}
\subcaption{}
\label{snub_dodecahedron}
\end{minipage}
\begin{minipage}{.4\textwidth}
    \centering
    \vspace{.5cm}
    \resizebox{!}{7.5cm}{\input{SnubDodecahedronExtended}}
\subcaption{}
\label{a_5_graph_snub}
\end{minipage}
\caption{(a) Snub dodecahedron and pentakis snub dodecahedron (b) Facegraph of the \\pentakis snub dodecahedron}
\end{figure}

\section{Embeddings of Simplicial Surfaces}\label{emb}
Inspired by the simplicial surfaces with cyclic or dihedral automorphism groups constructed in section \ref{simplicial_surface_construction}, we construct infinite families of simplicial surfaces and provide corresponding embeddings. Here, we focus on computing embeddings with all edge lengths $1$. In particular, we show the following: 
\begin{theorem}
    For $G=C_n$ with $n\geq 3$ or for $G=D_n$ with $n \geq 4$ there exists a simplicial surface $X_G$ with automorphism group isomorphic to $G$ and $X_G$ can be embedded into $\mathbb{R}^3$ with equilateral triangles.
\end{theorem}
More precisely, for $n \geq 3$ resp. $n\geq 4$ we construct families $(X^{(n,k)})_{k\in \mathbb{N}_0}$ resp. $(Y^{(n,k)})_{k\in \mathbb{N}_0}$ such that 
\begin{align*}
 & C_n \cong \Aut(\phi(X^{(n,k)})) \cong \Aut(X^{(n,k)})\hookrightarrow\Aut(\mathcal{F}(X^{(n,k)}))\\
   & D_n\cong \Aut(\phi(X^{(n,k)}))\cong \Aut(Y^{(n,k)})\hookrightarrow \Aut(\mathcal{F}(Y^{(n,k)}))
\end{align*}
The idea behind the construction is to combine tetrahedra, $n$-pyramids and $n$-strips (Figure \ref{pyramid}) in such a way that they preserve the desired cyclic or dihedral symmetry.
\begin{figure}[H]
\begin{minipage}{.5\textwidth}
    \centering
    \includegraphics[height=4cm]{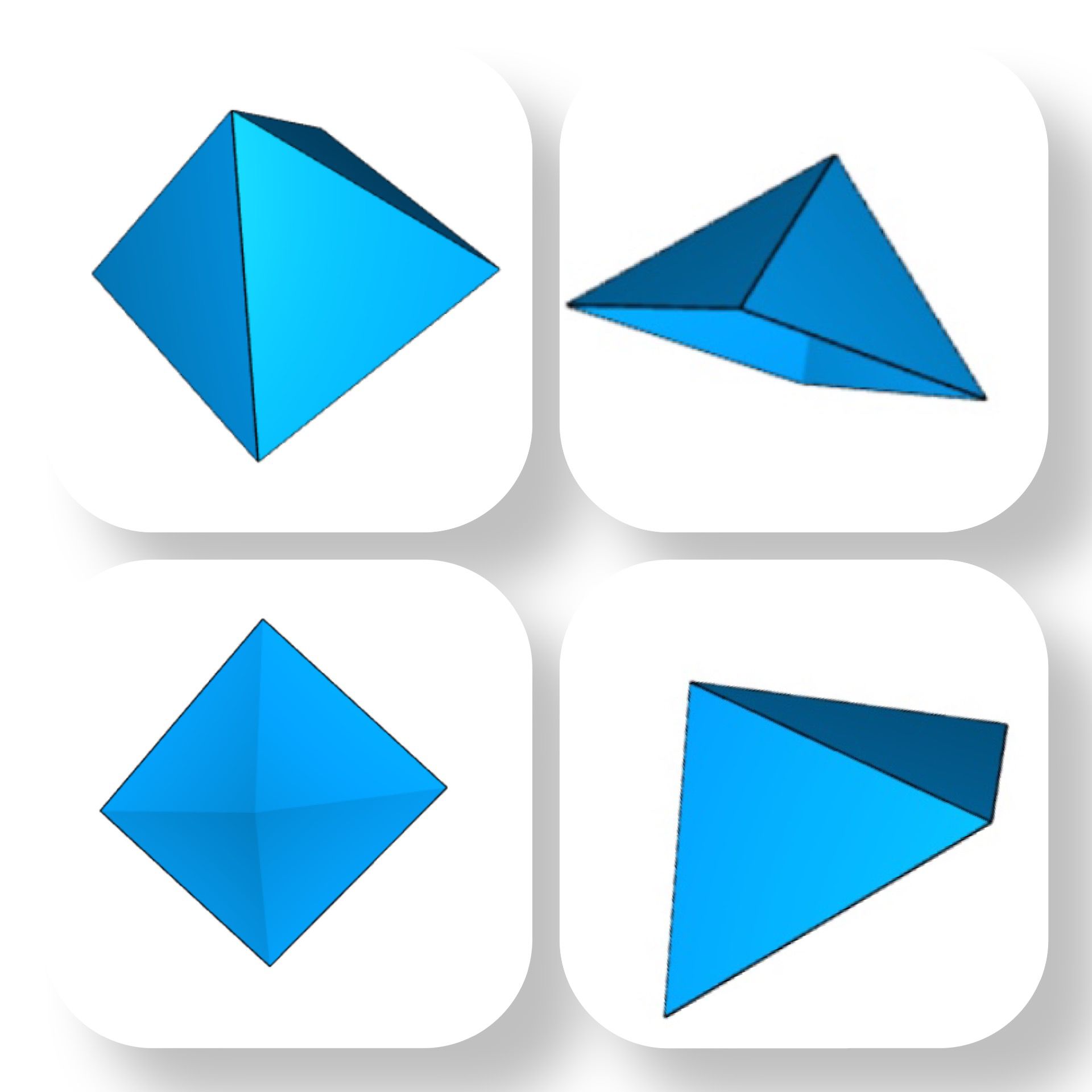}
    \includegraphics[height=4cm]{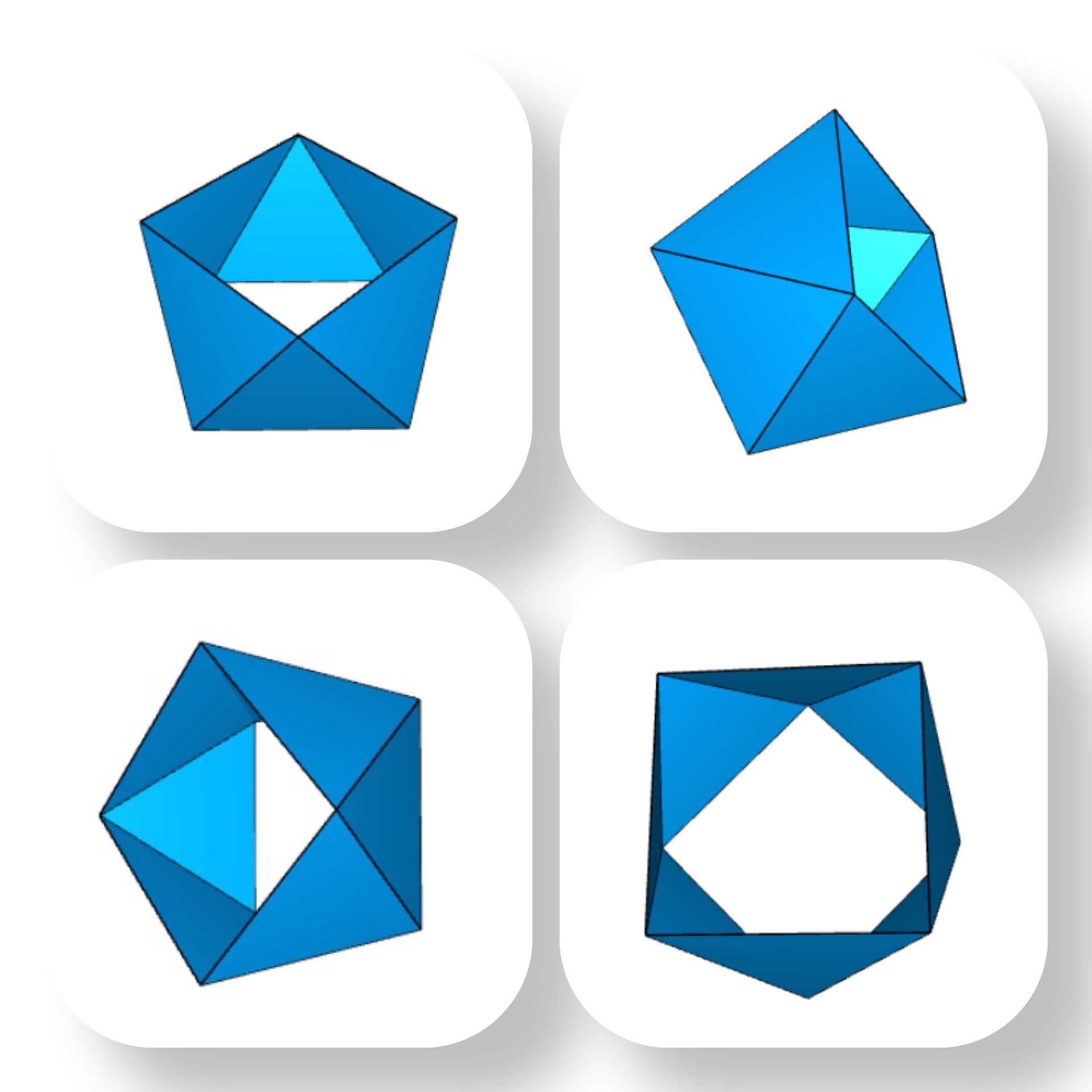}
\vspace{0.4cm}
\subcaption{}
\label{pyramid}
\end{minipage}
\begin{minipage}{.5\textwidth}
    \centering
    \vspace{.5cm}
    \resizebox{!}{4cm}{\input{octa}}
\subcaption{}
\label{octa}
\end{minipage}
\caption{(a) A simplicial 4-pyramid (left); A simplicial 4-strip (right) (b) Folding plan of a octahedron}
\end{figure}
 Since we aim to construct vertex-faithful surfaces, we can define the surfaces by constructing a subset $S$ of the faces, such that the remaining faces are given by a suitable group action on the vertices.  For example, the group $G=\langle (2,3,4,5) \rangle \leq S_6$ gives rise to a combinatorial Octahedron by computing the union of the following orbits:
\begin{align*}
&\{ 1,2,3\}^G\cup\{2,3,6\}^G=\{\{1,2,3\},\{1,3,4\},\{1,4,5\},\{1,2,5\},\{2,3,6\},\{3,4,6\},\{4,5,6\},\{2,5,6\}\}
\end{align*}
Here, the group $G$ acts on a subset by permuting the elements of the subset. 
\begin{remark}
     For simplicity, we identify a face $f$ of a vertex faithful simplicial surface with its set of incident vertices, i.e.\ we write $f=\{v_1,v_2,v_3\},$ if $v_1,v_2,v_3$ are the three vertices in the surface that are incident to $f.$  
\end{remark}
\subsection{Construction of cyclic family}\label{cyclic_embedding}
In this subsection, we give a detailed construction of the simplicial surfaces $X_{C_n}$ whose face graphs have cyclic automorphism groups. Here, we define the surfaces by computing the corresponding vertices of faces instead of a cycle double cover, as seen in section \ref{simplicial_surface_construction}. Furthermore we present embeddings of these surfaces consisting of equilateral triangles and prove that these surfaces and their computed embeddings also allow cyclic symmetries.   

Let therefore $n\geq 3$ be a natural number, $k$ be a non-negative integer and $\Gamma=G_{C_n}$ be the cubic graph, constructed in \ref{cyclic_dihedral_graph}, satisfying $\Aut(\Gamma)\cong C_n$. Since $\pi=(a_1,\ldots,a_n)\ldots(f_1,\ldots,f_n)$ is an automorphism of $\Gamma$ with $\ord(\pi)=n,$ the automorphism group of $\Gamma$ is given by $\langle \pi \rangle.$

In the following, we use the cyclic group
\begin{align*}
    G=\langle\, \underbrace{(1,\ldots,n)\ldots((k+2)n+1,\ldots,(k+3)n)}_{g\coloneqq}\,\rangle .
\end{align*}
 to obtain the vertices of faces of the  desired surface $X^{(n,k)}$ which fully determine the incidence structure of the surface. In the following, we define the set of faces of our surface $X^{(n,k)}$ by   \[
\{ f_{a_i},f_{b_i},f_{c_i},f_{d_i},f_{e_i},f_{f_i}\mid i=1,\ldots,n\}\cup \{ f_{i,j}, \overline{f_{i,j}}\mid i=1,\ldots,n,j=1,\ldots,k\}. 
\]

The vertices incident to the faces 
$f_{a_1},f_{b_1},f_{c_1},f_{d_1},f_{e_1},f_{f_1}$ are given by
\begin{align*}
&f_{a_1}=\{(k+1)n+1,(k+2)n+1,(k+3)n+2\},\\
&f_{b_1}=\{kn+1,(k+2)n+1,(k+3)n+2\},
 f_{c_1}=\{kn+1,kn+2,(k+1)n+1\},\\
&f_{d_1}=\{1,2,(k+3)n+1\},
    f_{e_1}=\{kn+2,(k+1)n+1,(k+3)n+2\},\\
    &f_{f_1}=\{kn+1,(k+1)n+1,(k+2)n+1\}
\end{align*}
and for $0< j \leq k$ the vertices of faces of $f_{1,j}$ and $\overline{f_{1,j}}$ are given by 
\[
f_{1,j}=\{(j-1)n+1,(j-1)n+2,jn+1\},\overline{f_{1,j}}=\{(j-1)n+2,jn+1,jn+2\}.
\]
Then, we obtain the remaining vertices of faces by the action of $G$ on the above faces. In particular, the vertices of the faces $f_{a_{i+1}},\ldots,f_{f_{i+1}},$ are given by ${(f_{a_1})}^{g^i},\ldots,{(f_{f_1})}^{g^i}$ respectively.  
In the case that $k>0,$ we define the vertices of the faces $f_{i+1,j}$ resp. $\overline{f_{i+1,j}}$ by ${({f_{1,j}})}^{g^i}$ resp. ${(\overline{{f_{1,j}}})}^{g^i}.$   
For example, for $n=4$ and $k=0$, we obtain the following vertices of faces by considering the $G$-orbits of $f_{a_1},\ldots,f_{f_1}$:
\begin{align*}
\{\{5,9,14\},\{6,10,14\},\{7,11,14\},\{8,12,14\}\} &\cup
\{\{1,9,14\},\{2,10,14\},\{3,10,14\},\{4,11,14\} \}\cup \\
\{\{1,2,5\},\{2,3,6\},\{3,4,7\},\{1,4,8\}\}&\cup
\{\{1,2,13\},\{2,3,13\},\{3,4,13\},\{1,4,13\}\}\cup\\
\{\{2,5,14\},\{3,6,14\},\{4,7,14\},\{1,8,14\}\}&\cup
\{\{1,5,9\},\{2,6,10\},\{3,7,11\},\{4,8,12\}\}.
\end{align*}
\begin{remark}
The simplicial surface $X^{(n,k)}$ has the following properties: 
    \begin{itemize}
        \item The surface $X^{(n,k)}$ has Euler-Characteristic 2.
        \item The face graph of the simplicial surface $X^{(n,0)}$ is isomorphic to $G_{C_n}$.
    \end{itemize}
\end{remark}
Next, we seek to embed these simplicial surfaces into $\mathbb{R}^3$ as polyhedra constructed from equilateral triangles, see Definition \ref{embedding}. To achieve this, we assign 3D-coordinates to the vertices of $X^{(n,k)}$ described as follows:
Let $l$ be a natural number with $\gcd(n,l)=1$ and $\cos(\alpha)\leq \frac{1}{2},$ where $\alpha\coloneqq\frac{2\pi l}{n}$. Furthermore, let $h$ and $\rho$ be scalars defined by 
\[
h\coloneqq\sqrt{1+4\sin^2\left(\frac{\alpha}{4}\right)},
\]
\[
\rho\coloneqq\lVert\left(
\begin{array}{c}
     \cos(\alpha) \\
     \sin (\alpha) 
\end{array}\right)-\left(
\begin{array}{c}
     \cos(2\alpha) \\
     \sin (2\alpha) 
\end{array}\right)\rVert
=2\left(1-\cos(\alpha)\right).
\]
Similar to the construction of the vertices of faces of $X^{(n,k)}$, we define the 3D-coordinates of a subset of the vertices of the surface such that the remaining 3D-coordinates are then given by a suitable group action.
 An embedding $\phi:X^{(n,k)}_0\to \mathbb{R}^3$ of the simplicial surface $X^{(n,k)}$ which gives rise to a polyhedron constructed from equilateral triangles is then given by 
\begin{align*}
&    \phi(jn+1)=-jh \left(
    \begin{array}{ccc}
         0&0  &1
    \end{array}\right)^t+\left(
    \begin{array}{ccc}
         \cos(i\frac{\alpha}{2})&\sin(j\frac{\alpha}{2})  &0
    \end{array}\right)^t\\
    &    \phi(jn+2)=-jh \left(
    \begin{array}{ccc}
         0&0  &1
    \end{array}\right)^t+\left(
    \begin{array}{ccc}
         \cos((j+1)\frac{\alpha}{2})&\sin((j+1)\frac{\alpha}{2})  &0
    \end{array}\right)^t,
\end{align*}
for $j=0\ldots, k$ and 
\begin{align*}
    \phi((k+3)n+1)=&\sqrt{1-\rho^2} \left(
    \begin{array}{ccc}
         0&0  &1
    \end{array}
\right)^t,
    \phi((k+3)n+2)=-(kh+\sqrt{1-\rho^2}) \left(
    \begin{array}{ccc}
         0&0  &1
    \end{array}
\right)^t\\
\phi((k+1)n+1)=&\frac{1}{3}(\phi(kn+1)+\phi(kn+2)+\phi((k+3)n+2))\\
&+\frac{\phi((k+3)n+2)-\phi(kn+1))\times \phi(kn+2)-\phi((k+3)n+2))}{\|\phi((k+3)n+2)-\phi(kn+1))\times \phi(kn+2)-\phi((k+3)n+2))\|}\\
\phi((k+2)n+1)=&\frac{1}{3}(\phi(kn+1)+\phi(((k+1)n+1)+\phi((k+3)n+2))\\
&+\frac{\phi(kn+1)-\phi((k+3)n+2))\times \phi(kn+1)-\phi((kn+1)n+1))}{\|\phi(kn+1)-\phi((k+3)n+2))\times \phi(kn+1)-\phi((kn+1)n+1))\|}.
\end{align*}
The images of the remaining vertices of $X^{(n,k)}$ under $\phi$ are then given by 
\begin{align*}
& \phi ((jn+(i+1)))=\phi ((jn+1)^{g^i})={(M_\alpha)}^i \phi(jn+1),
\end{align*}
where $i=1,\ldots,n, j=0,\ldots,k$ and $M_\alpha$ is given by
\[
\left(\begin{array}{ccc}
     \cos(\alpha)&-\sin(\alpha) & 0 \\
     \sin(\alpha)&\cos(\alpha) &0\\
     0&0&1
\end{array}
\right).
\]
Since the euclidean norm is invariant under the multiplication with orthogonal matrices, it suffices to show the edges that are incident to the faces $f_{a_1},\ldots f_{f_1},f_{1,j},\overline{f_{1,j}}$ all have length 1. Since these edges have length 1 by construction, $\phi$ gives rise to an embedding of $X^{(n,k)}$ constructed from congruent triangles.
For simplicity we denote this embedding of $X^{(n,k)}$ by $\phi^{(n,k)}_l=\phi_l.$ 
Figure \ref{embc45} shows polyhedra that arise from the embedding $\phi_1$ of $X^{(4,0)}$ and the embedding $\phi_1$ of $X^{(5,0)}.$
\begin{figure}[H]
\begin{minipage}{.5\textwidth}
    \centering
    \includegraphics[height=4cm]{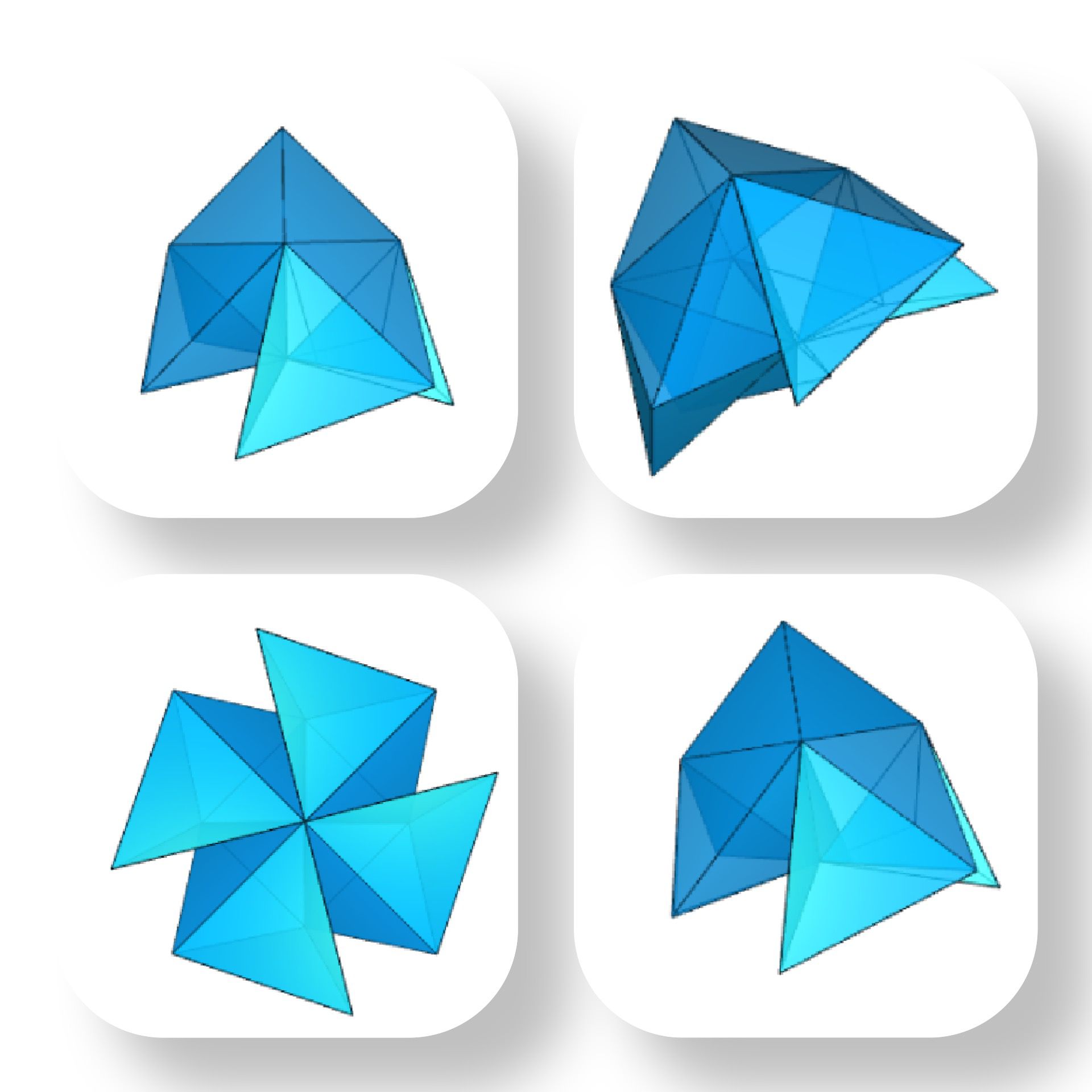}
    \includegraphics[height=4cm]{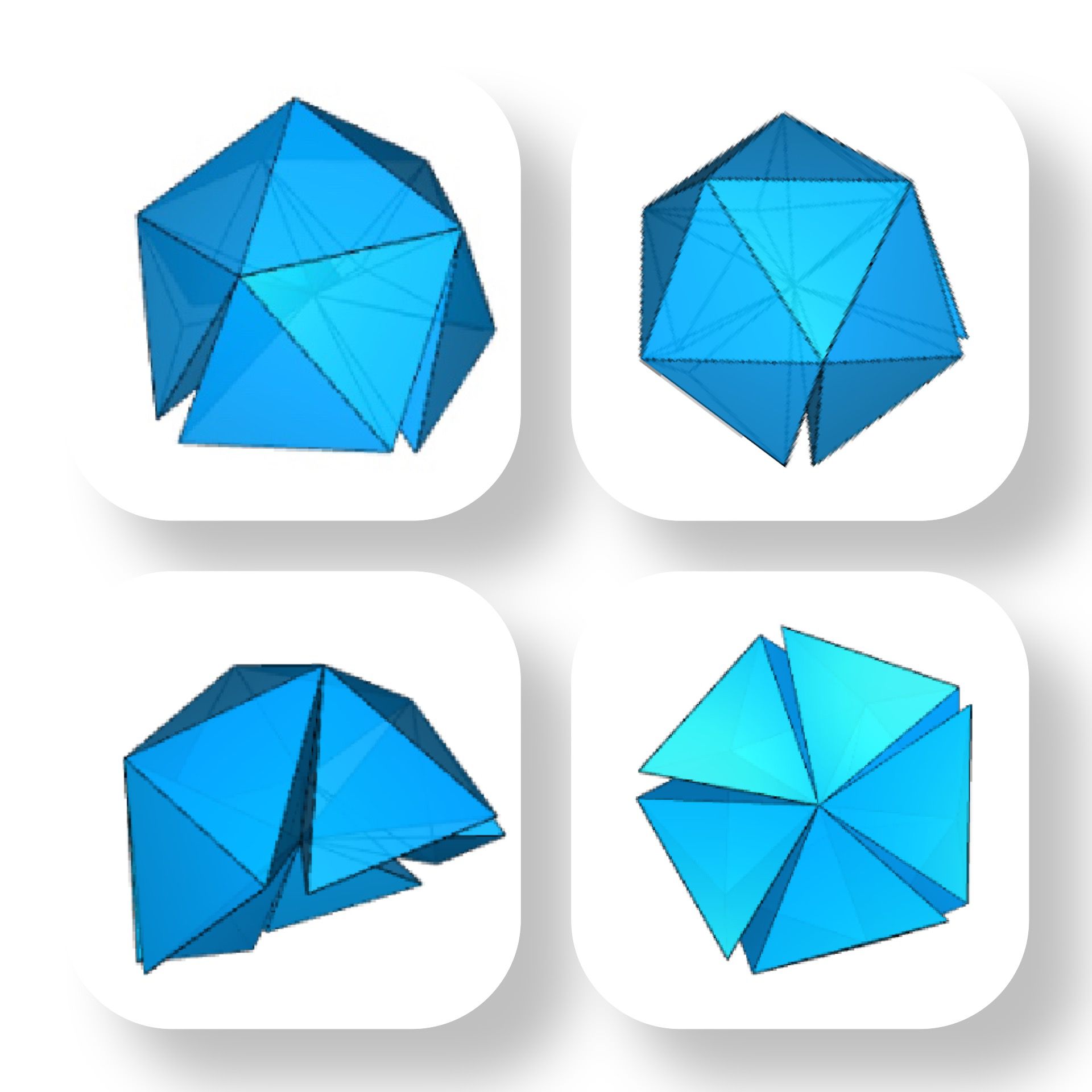}
\vspace{0.4cm}
\subcaption{}
\label{embc45}
\end{minipage}
\begin{minipage}{.5\textwidth}
    \centering
    \includegraphics[height=4cm]{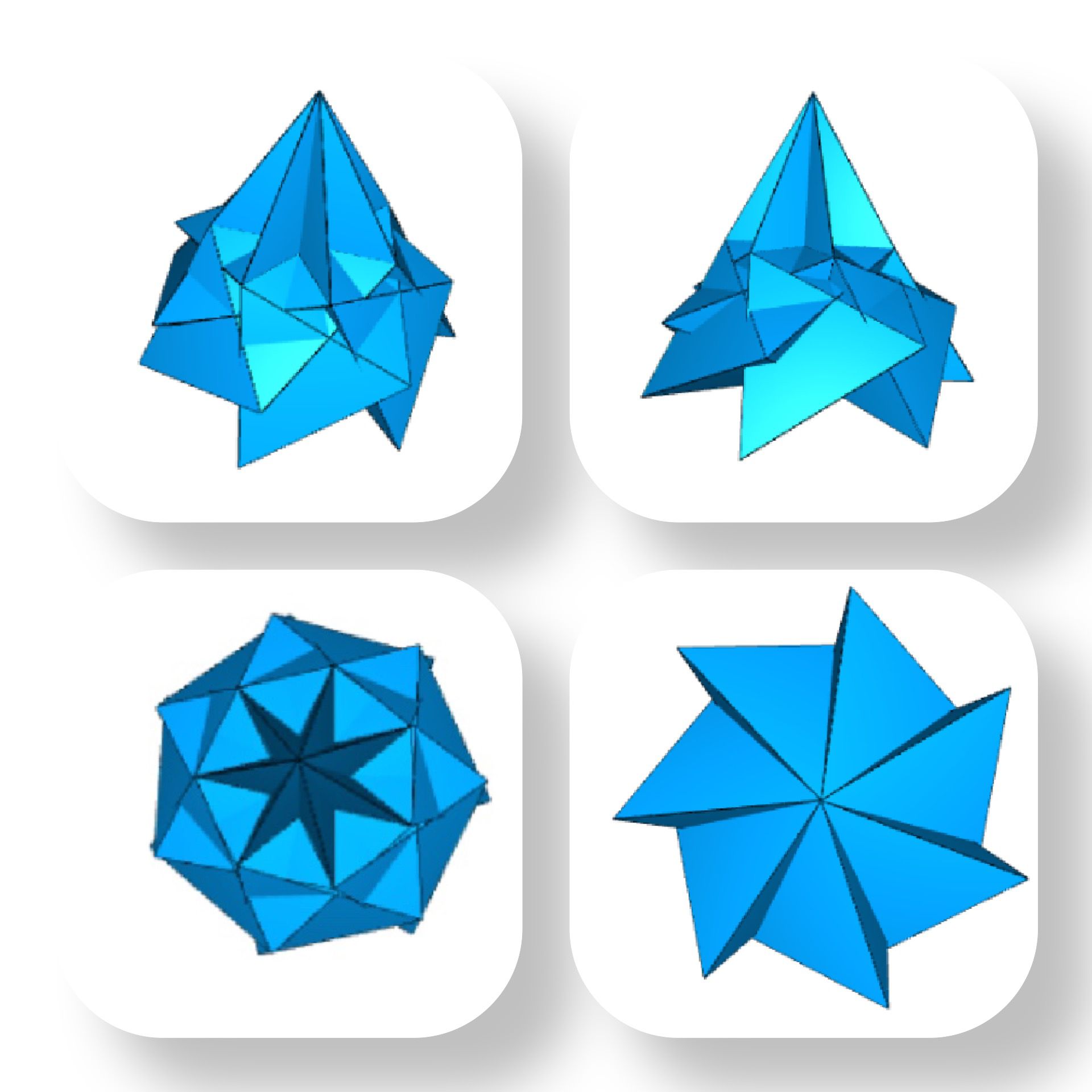}
    \includegraphics[height=4cm]{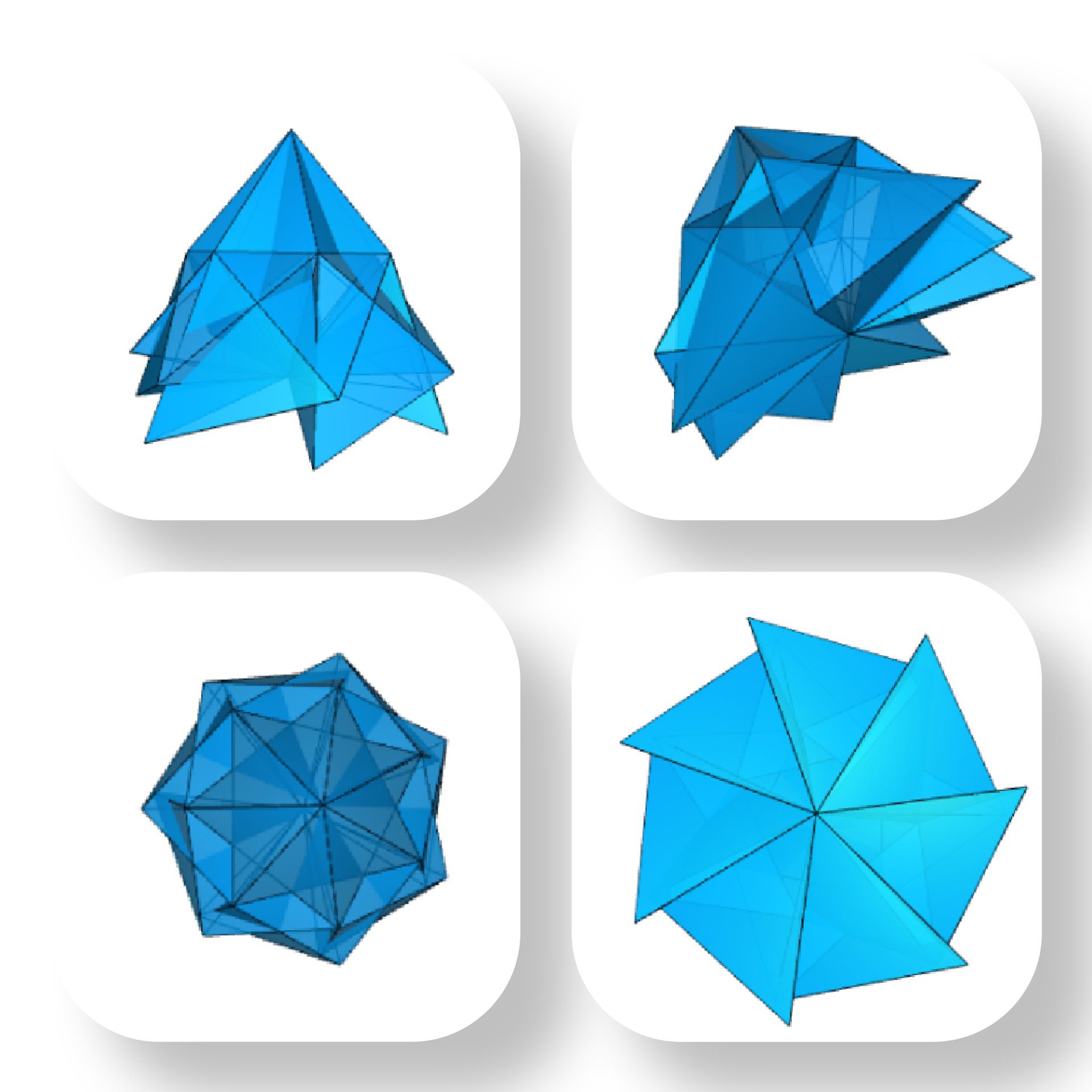}
\vspace{0.4cm}
\subcaption{}
\label{embc7}
\end{minipage}\caption{(a) Embeddings of the surfaces $X^{(4,0)}$ (left) and $X^{(5,0)}$ (right) (b) Two different embeddings of the surface $X^{(7,0)}$}
\end{figure}
Note, for $n\geq 3$ the above procedure yields 
\[
\vert \{l\mid 1\leq l \leq \frac{n}{2},\, \gcd(n,l)=1, \cos\left(\frac{2\pi l}{n}\right)\leq \frac{1}{2}\} \vert
\]
different embeddings of $X^{(n,k)}$ that cannot be transformed into each other by using rigid Euclidean motions only. Hence, the above procedure yields two different embeddings of the surface $X^{(7,0)}$ as polyhedra consisting of equilateral triangles, see Figure \ref{embc7}. In Figure \ref{embc44}, the construction of the family $(X^{(4,k)})_{k \in \mathbb{N}_0}$ is illustrated by showing the embeddings of the surfaces $X^{(4,1)},X^{(4,2)}$ and $X^{(4,3)}.$
\begin{figure}[H]
    \centering
    \includegraphics[height=4cm]{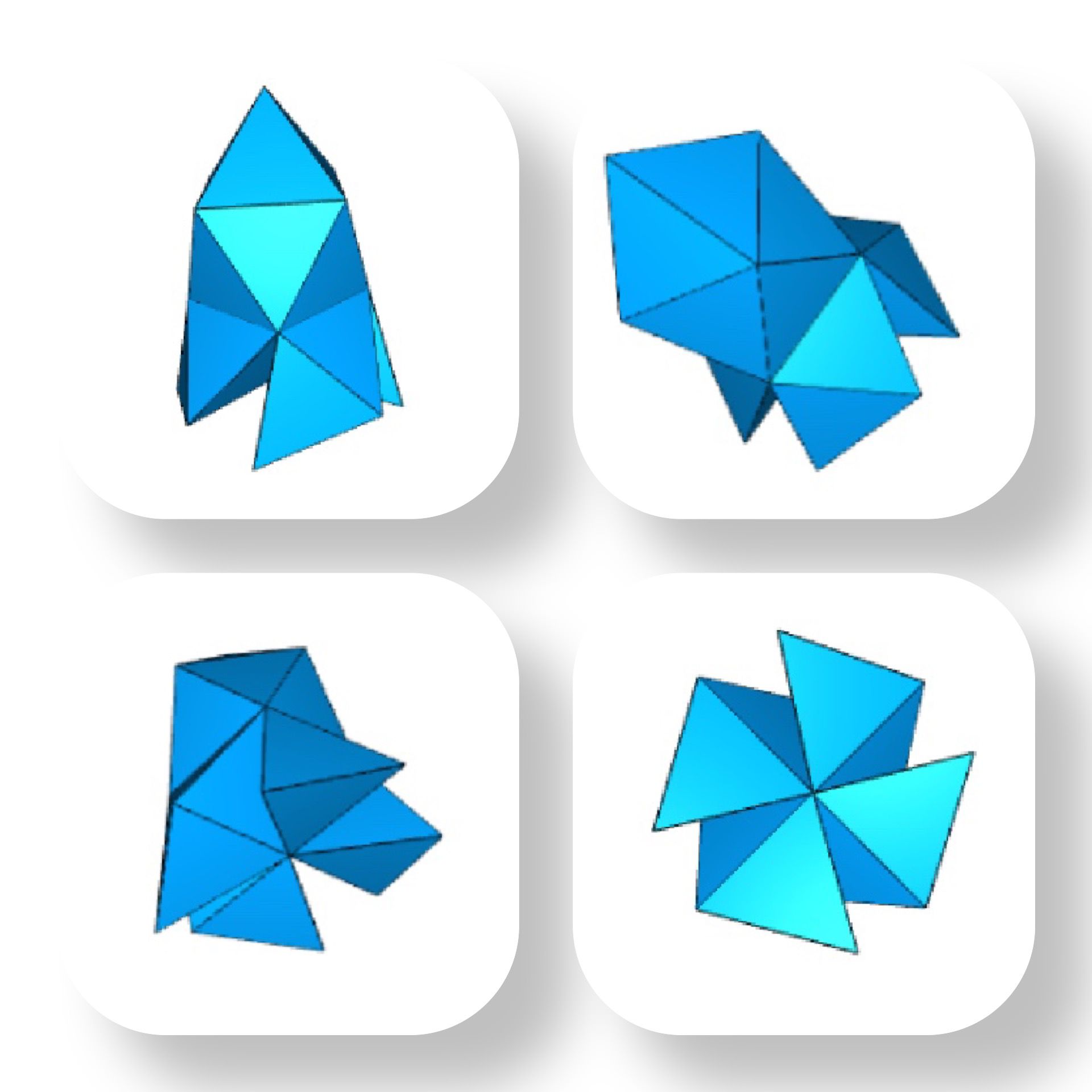}
    \includegraphics[height=4cm]{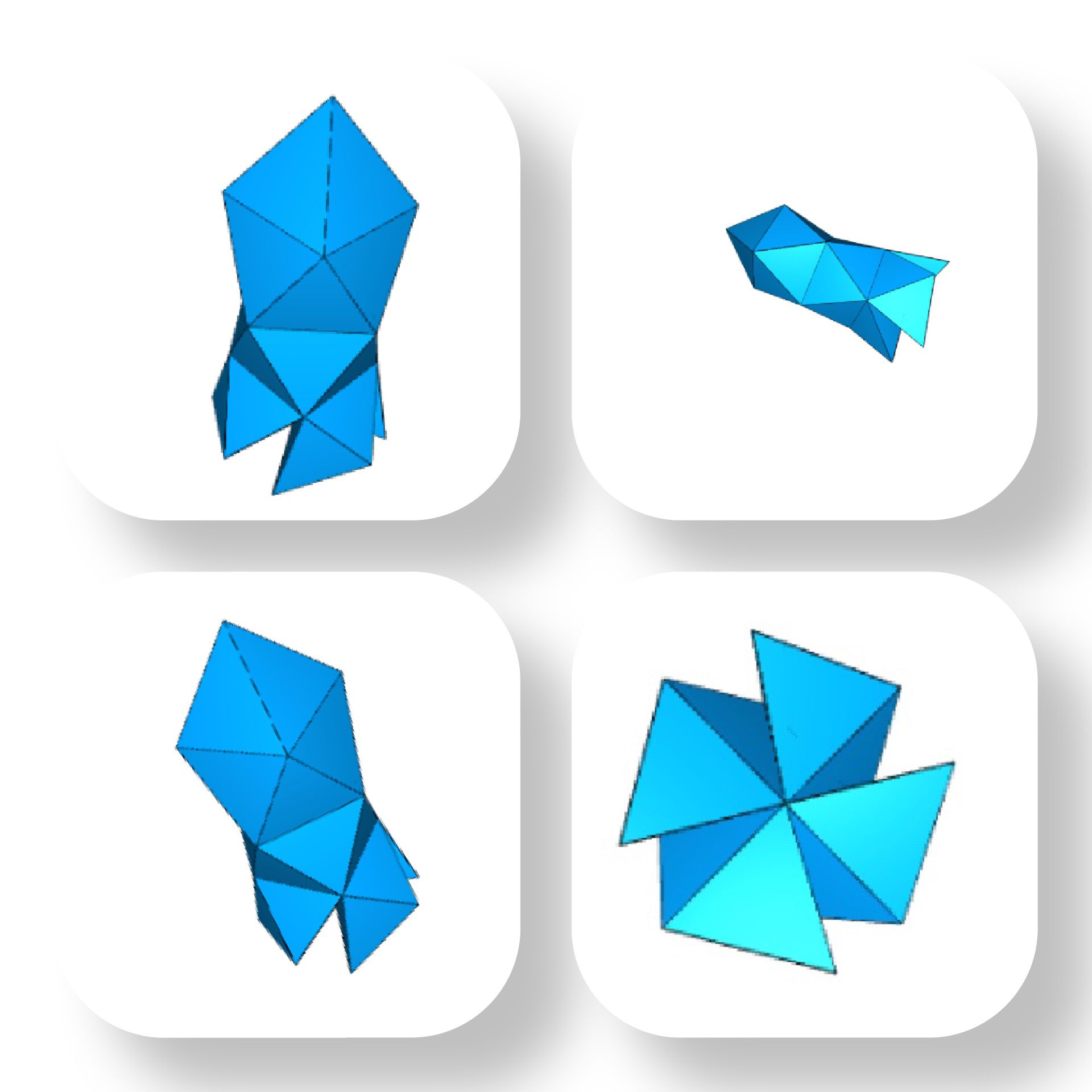}
    \includegraphics[height=4cm]{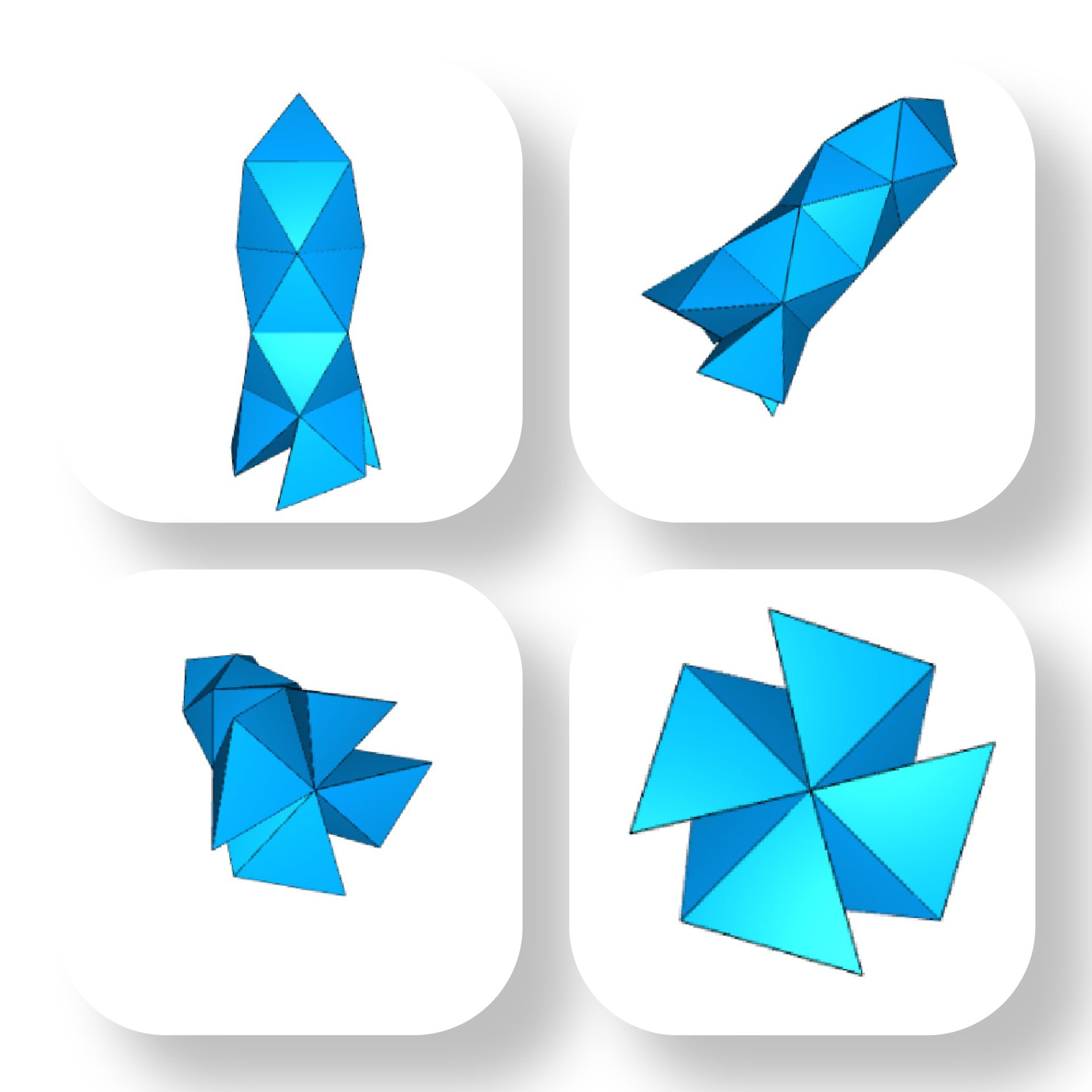}
    \caption{Embeddings of the surfaces $X^{(4,1)}$(left), $X^{(4,2)}$(middle) and $X^{(4,3)}$(right)}
    \label{embc44}
\end{figure}
Next, we analyse the automorphism group of $X^{(n,k)}$ further. We make use of the following remark to prove that the automorphism group of $X^{(n,k)}$ and the symmetry group of $\phi_l(X^{(n,k)})$ are both cyclic groups of order $n$. 
\begin{remark}\label{vertexfaithful}
    Let $X$ be a vertex faithful simplicial surface and $v_1,v_2,v_3,$ resp. $w_1,w_2,w_3$ be vertices of $X$ all incident to the same face. Then there exists at most one automorphism $\phi\in \Aut(X)$ with $\phi(v_i)=w_i.$
    So, a automorphism $\phi$ of a vertex faithful simplicial surface $X$ is uniquely identified by the images of the vertices of a face of $X.$
\end{remark}
\begin{prop}
Let $n,k,l$ be defined as in the above construction and $\phi_l$ be an embedding of the simplicial surface $X^{(n,k)}.$ Then  
\[
\Aut(\phi_l(X^{(n,k)})) \cong \Aut(X^{(n,k)}) \cong C_n. 
\]
\end{prop}
\begin{proof}
By construction $\langle M_\alpha \rangle$ is a subgroup of the symmetry group $\Aut(\phi_l(X^{(n,k))}).$ Since $\langle M_\alpha \rangle \cong C_n,$ it suffices to show that 
the automorphism group of $X^{(n,k)}$ is also cyclic of order $n.$
Let therefore $\phi$ be an automorphism of $X^{(n,k)}$ and $(k+2)n+1,\ldots,(k+3)n$ the vertices of degree 3 in $X^{(n,k)}.$ Note, that there is exactly one vertex of degree $3n,$ namely $(k+3)n+2$ and that the image of a vertex under $\phi$ has to be a vertex with the same vertex degree. Thus, we conclude: 
\begin{itemize}
    \item $\phi((k+2)n+1)=(k+2)n+i$ for $i \in \{1,\ldots,n\},$
    \item $\phi((k+3)n+1)=(k+3)n+1.$
\end{itemize}
Since $\{(k+1)n+1,(k+2)n+1,(k+3)n+1\}$ is a face of $X^{(n,k)},$ the same holds for its images under $\phi.$ So, $\phi((k+1)n+1)$ has to be a vertex of degree 4, such that $\{\phi((k+1)n+1),(k+2)n+i,(k+3)n+1\}$ forms a face of the simplicial surface. This leads to $\phi ((k+1)n+1)=(k+1)n+i.$
With remark \ref{vertexfaithful}, the restriction of $\phi$ to the vertices of the surface is given by 
$((1,\ldots,n)\ldots((k+2)n+1,\ldots,(k+3)n))^i.$
Thus, we obtain
$C_n \cong \langle M_\alpha \rangle \leq\Aut(\phi_l(X^{(n,k)}))\hookrightarrow \Aut(X^{(n,k)})\cong C_n.$
\end{proof}
\subsection{Construction of dihedral family}
In this subsection, we present the detailed construction of the simplicial surfaces $Y^{(n,k)}$. We compute the incidence structure of the surfaces and the corresponding embeddings consisting of equilateral triangles, as in the cyclic case. Let therefore $n,k,l,\alpha,\rho,h,M_{\alpha}$ be defined as in subsection \ref{cyclic_embedding} and $\Gamma_{D_n}$ the cubic graph constructed in \ref{cyclic_dihedral_graph} with dihedral automorphism group. Furthermore, we prove that the surface $Y^{(n,k)}$ and the constructed embeddings have a underlying dihedral symmetry.

As described in the above construction, we define the set of faces of the simplicial surface $Y^{(n,k)}$ by 
\[
\{ f_{a_i},f_{b_i},f_{c_i},f_{d_i}\mid i=1,\ldots,n\}\cup \{ f_{i,j}, \overline{f_{i,j}}\mid i=1,\ldots,n,j=1,\ldots,k\}.
\]
Moreover, let the vertices of the faces  $f_{a_1},\ldots,f_{d_1}$ be defined by
\begin{align*}
    &f_{a_1}=\{kn+1,(k+1)n+1,(k+2)n+2\},f_{b_1}=\{kn+2,(k+1)n+1,(kn+2)+2\},\\
    &f_{c_1}=\{kn+1,kn+2,(k+1)n+1\},f_{d_1}=\{1,2,(k+2)n+1\}.
\end{align*}
and for $0< j \leq k$ the vertices of the faces $f_{1,j}$ and $\overline{f_{1,j}}$ be given by 
\[
f_{1,j}=\{(j-1)n+1,(j-1)n+2,jn+1\},\overline{f_{1,j}}=\{(j-1)n+2,jn+1,jn+2\}.
\]
Here, the permutation $h=(1,\ldots,n)\ldots ((k+1)n+1,\ldots,(k+2)n)$ gives rise to the vertices of the remaining faces of $Y^{(n,k)}.$
More precisely, the vertices of the faces $f_{a_{i+1}},\ldots f_{d_{i+1}},$ are given by ${(f_{a_1})}^{h^i},\ldots,{(f_{h_1})}^{h^i}$ respectively.  
If $k>0,$ we define the vertices of the faces $f_{i+1,j}$ resp. $\overline{f_{i+1,j}}$ by ${({f_{1,j}})}^{h^i}$ resp. ${(\overline{{f_{1,j}}})}^{h^i}.$   

For example, for $n=4$ and $k=0$, we obtain the following vertices of faces:
\begin{align*}
\{&\{1,5,10\},\{2,6,10\},\{3,7,10\},\{4,8,10\},\{2,5,10\},\{3,6,10\},\{4,7,10\},\{1,8,10\},\\
&\{1,2,5\},\{2,3,6\},\{3,4,7\},\{1,4,8\},\{1,2,9\},\{2,3,9 \},\{3,4,9 \}  \{ 1,4,9\}\}.
\end{align*}
\begin{remark}
The simplicial surface $Y^{(n,k)}$ satisfies the following properties: 
    \begin{itemize}
        \item The surface $X^{(n,k)}$ has Euler-Characteristic 2.
\item The face graph of the simplicial surface $Y^{(n,0)}$ is isomorphic to $G_{D_n}$.
 \end{itemize}
\end{remark}
Next, we embed the simplicial surface $Y^{(n,k)}$ into $\mathbb{R}^3$ by assigning 3D-coordinates to the vertices of the simplicial surface as follows:

Let $n,k,l,\alpha,\rho,h,M_\alpha $ be defined as in the construction of the cyclic embeddings. 
If $\cos(\alpha)\leq \frac{1}{2},$ then an embedding $\phi$ of the simplicial surface $Y^{(n,k)}$ is given as follows:
For $0<i\leq k $ we define the following images:
\begin{align*}
&    \phi(in+1)=-ih \left(
    \begin{array}{ccc}
         0&0  &1
    \end{array}\right)^t+\left(
    \begin{array}{ccc}
         \cos(i\frac{\alpha}{2})&\sin(i\frac{\alpha}{2})  &0
    \end{array}\right)^t\\
    &    \phi(in+2)=-ih \left(
    \begin{array}{ccc}
         0&0  &1
    \end{array}\right)^t+\left(
    \begin{array}{ccc}
         \cos((i+1)\frac{\alpha}{2})&\sin((i+1)\frac{\alpha}{2})  &0
    \end{array}\right)^t.
\end{align*}
Furthermore, the images of the other vertices are given by
\begin{align*}
    \phi((k+2)n+1)=&\rho \left(
    \begin{array}{ccc}
         0&0  &1
    \end{array}
\right)^t,
    \phi((k+2)n+2)=-(kh+\rho) \left(
    \begin{array}{ccc}
         0&0  &1
    \end{array}
\right)^t\\
\phi((k+1)n+1)=&\frac{1}{3}(\phi(kn+1)+\phi(kn+2)+\phi((k+3)n+2))\\
&+\frac{\phi((k+3)n+2)-\phi(kn+1))\times \phi(kn+2)-\phi((k+3)n+2))}{\|\phi((k+3)n+2)-\phi(kn+1))\times \phi(kn+2)-\phi((k+3)n+2))\|}.
\end{align*}
Hence, we can define the images of the other vertices by
\begin{align*}
& \phi ((in+j))=\phi ((in+1)^{g^j})={M_\alpha}^j \phi(in+1)
\end{align*}
where $i\in \{1,\ldots,k\}, j\in \{1,\ldots,n\}.$ Since the euclidean norm is invariant under the multiplication with orthogonal matrices, it suffices to show the edges of the constructed surfaces that are incident to the faces $f_{a_1},f_{b_1},f_{c_1},f_{d_1},f_1^i,f_2^i$ all have length 1. Since these edges have length 1 by construction, the above construction gives rise to an embedding of $Y^{(n,k)}$ constructed from congruent triangles.
For simplicity, we denote the embedding of the simplicial surface $Y^{(n,k)}_l$ that can be constructed by using the above procedure by $\psi_l^{(n,k)}.$ 
Figure \ref{embd45} shows the corresponding embedding for $n=4$ and $n=5.$
\begin{figure}[H]
\begin{minipage}{.5\textwidth}
    \centering
    \includegraphics[height=4cm]{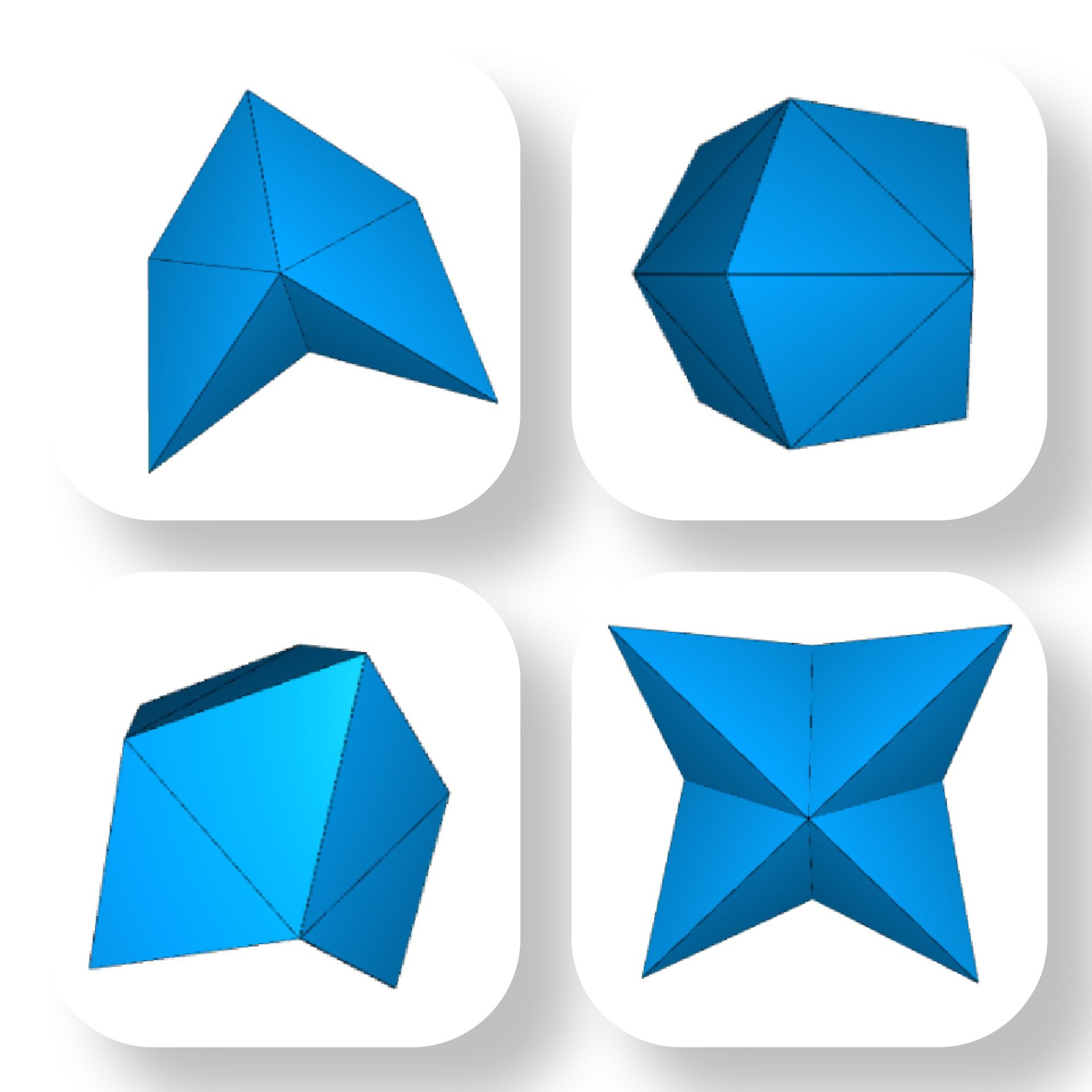}
    \includegraphics[height=4cm]{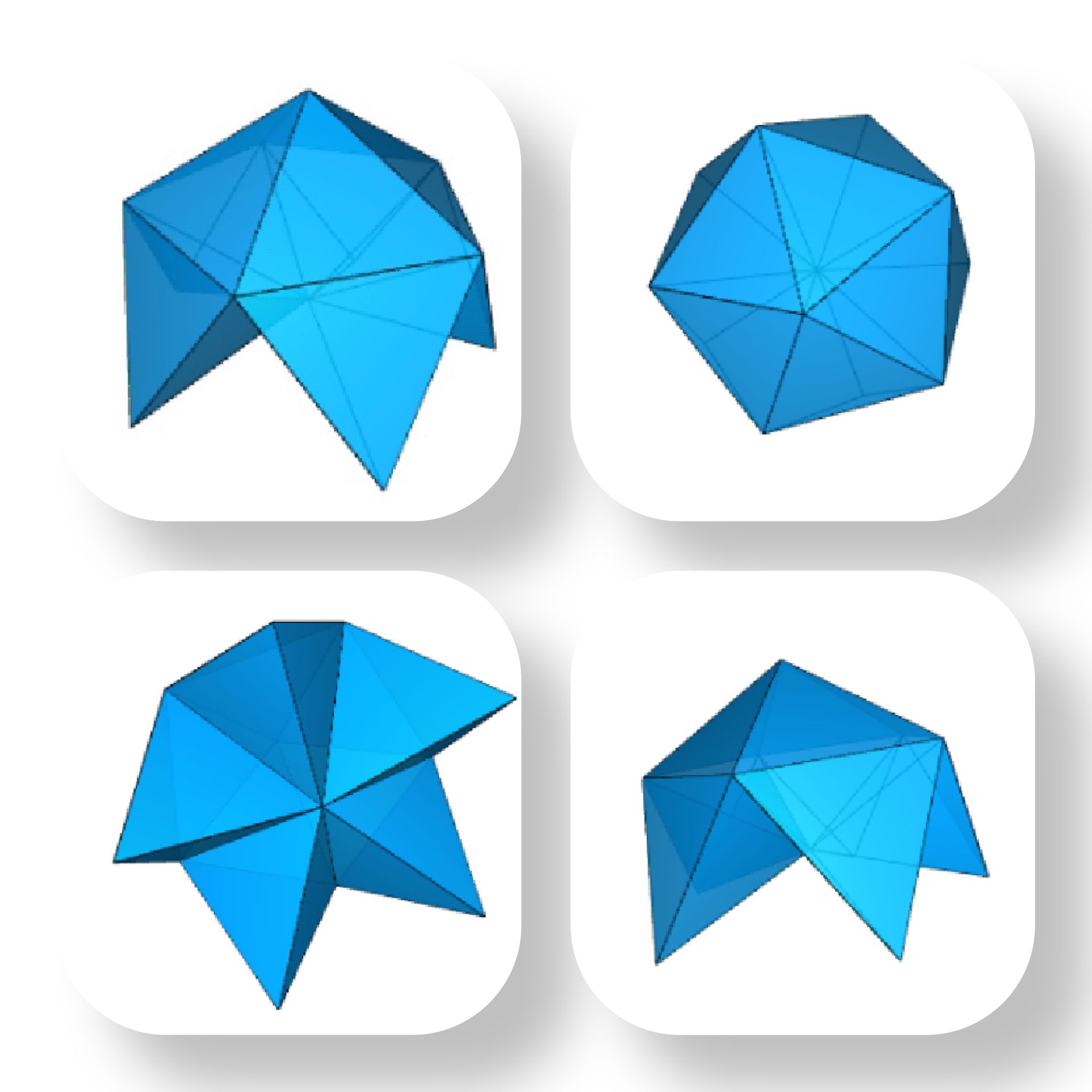}
\vspace{0.4cm}
\subcaption{}
\label{embd45}
\end{minipage}
\begin{minipage}{.5\textwidth}
    \centering
    \includegraphics[height=4cm]{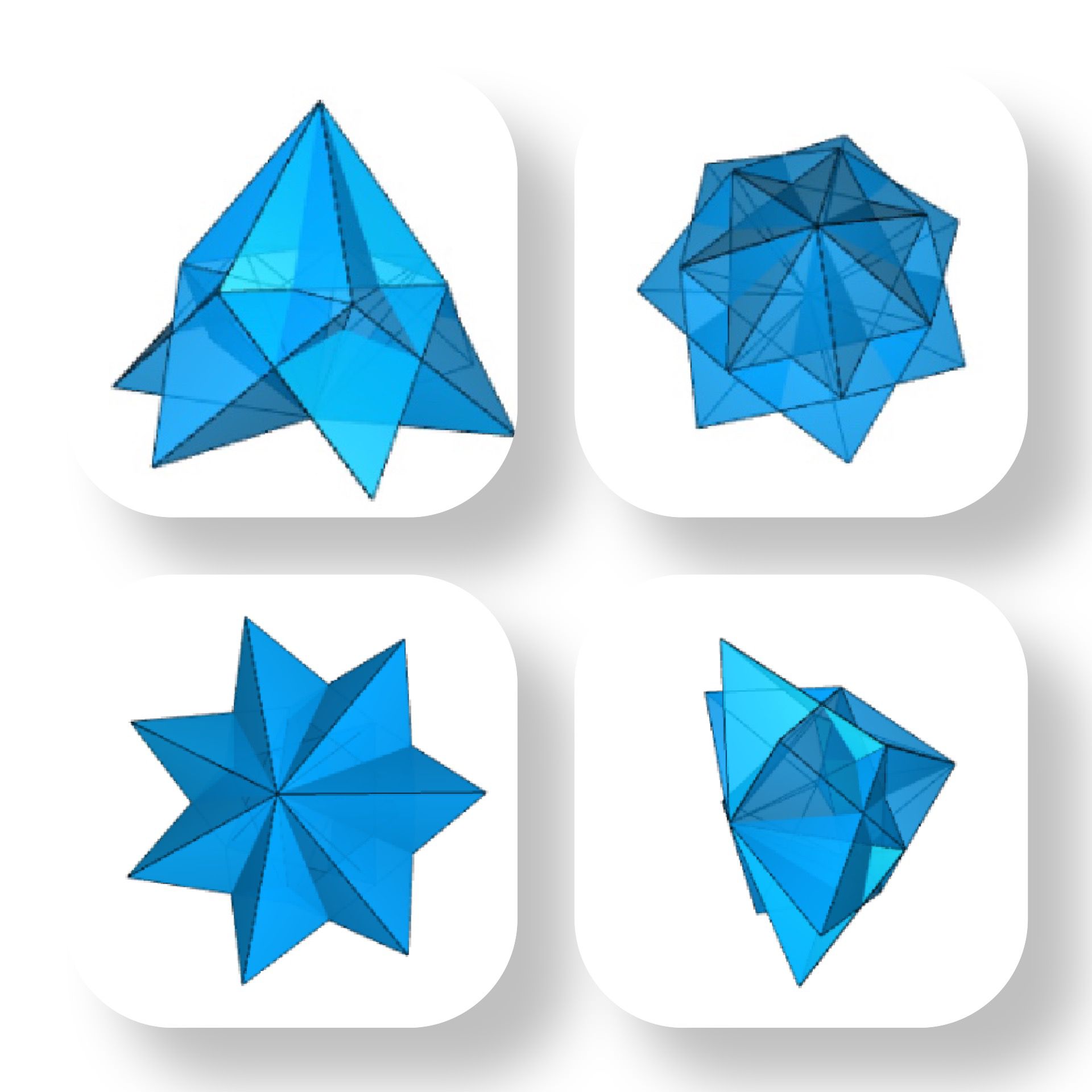}
    \includegraphics[height=4cm]{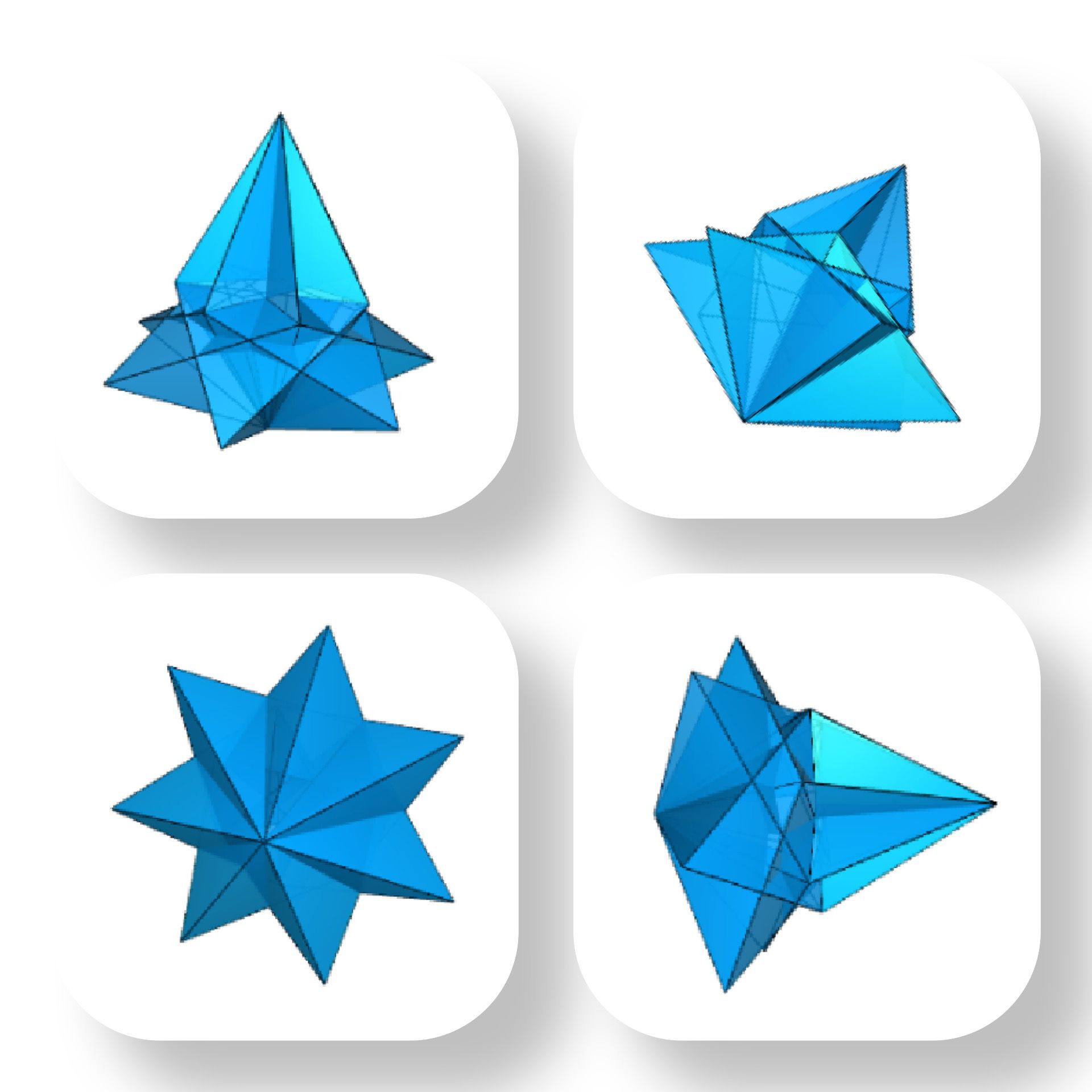}
\vspace{0.4cm}
\subcaption{}
\label{embd7}
\end{minipage}\caption{(a) Embeddings of the surfaces $Y^{(4,0)}$ (left) and $Y^{(5,0)}$ (right) (b) Two different embeddings of the surface $Y^{(7,0)}$}
\end{figure}
Note, as presented in the cyclic case the above construction gives rise 
\[
\vert \{k\mid 1\leq k \leq \frac{n}{2},\, gcd(n,k)=1, \cos(\frac{2\pi k}{n})\leq n\} \vert
\]
different embeddings of $Y^{(n,k)}$ that can not be transformed into each other by using rigid Euclidean motions. So for example, the above construction yields two different embeddings of the surface $Y^{(7,0)},$ see Figure \ref{embd7}. Figure \ref{embd44}, shows two members of the family $(X^{(4,k)})_{k \in \mathbb{N}_0}.$
\begin{figure}[H]
    \centering
    \includegraphics[height=4cm]{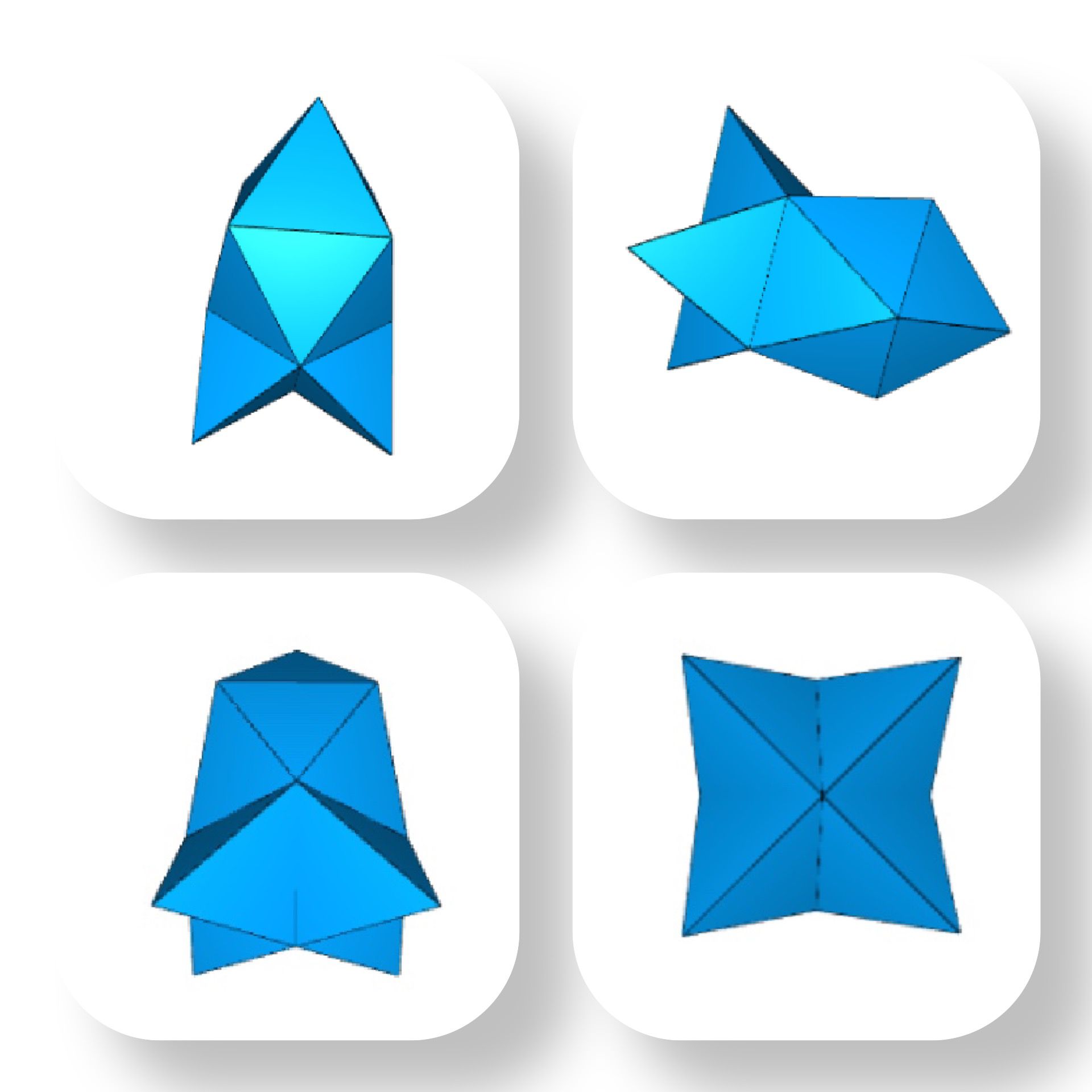}
    \includegraphics[height=4cm]{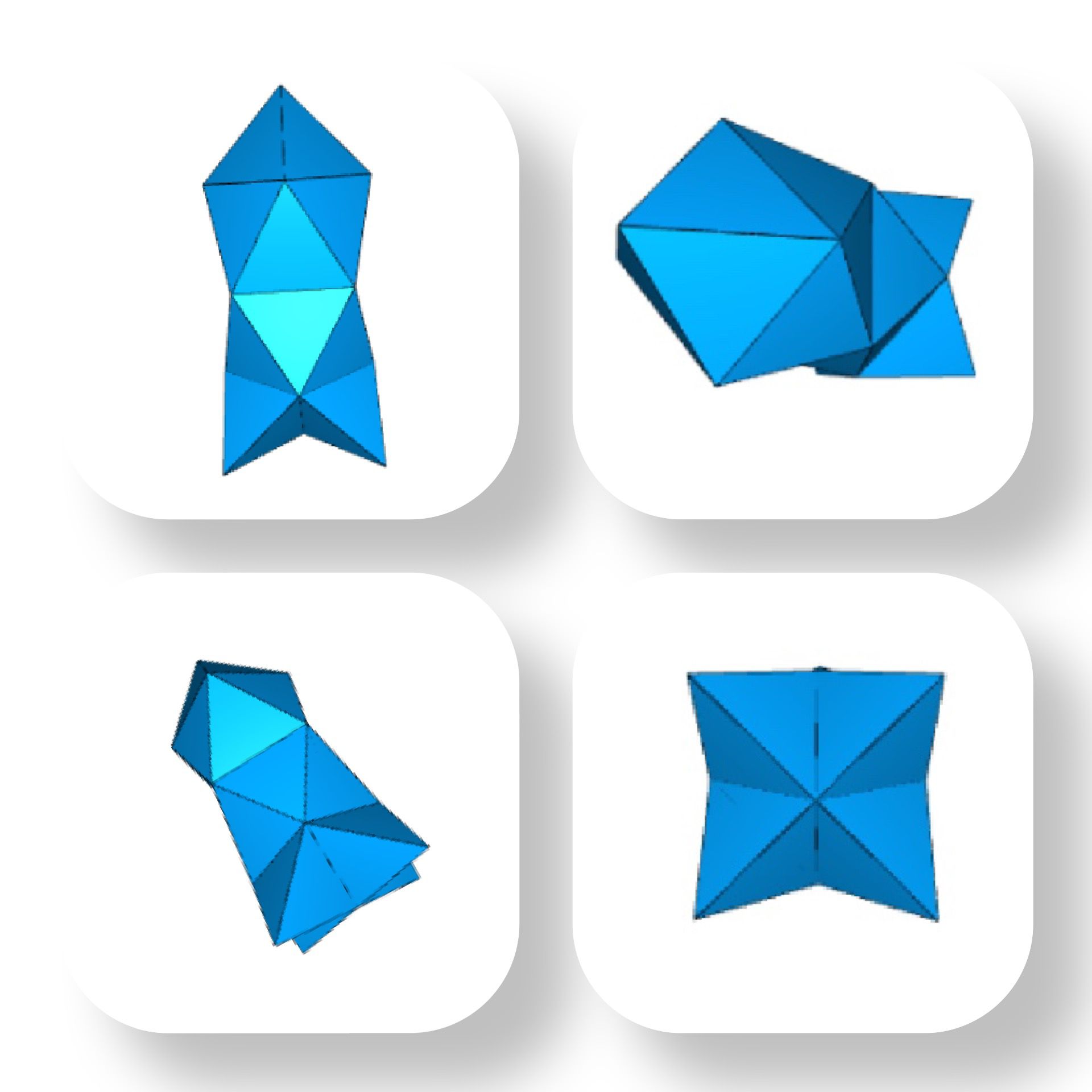}
    \includegraphics[height=4cm]{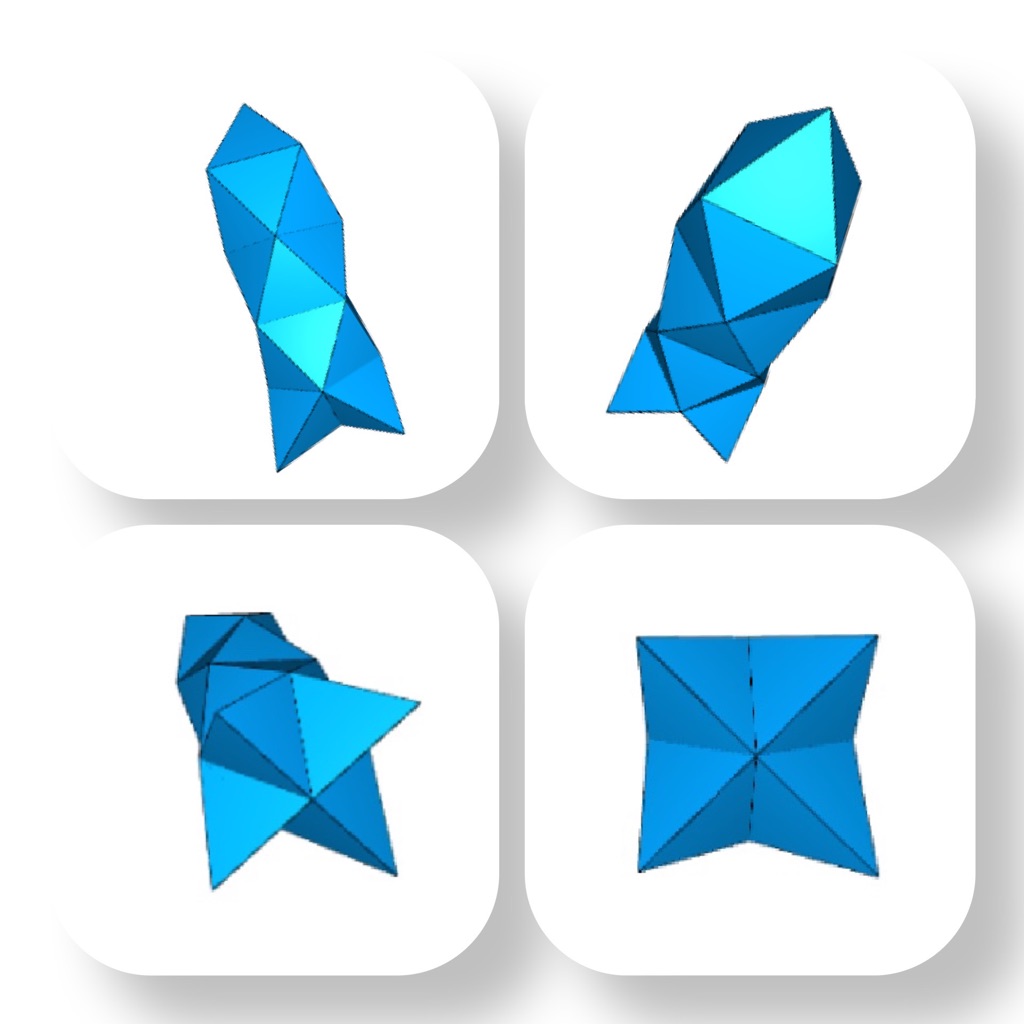}
    \caption{Embeddings of the surfaces $Y^{(4,1)}$(left),$Y^{(4,2)}$(middle) and $Y^{(4,3)}$(right)}
    \label{embd44}
\end{figure}
Next, we show that the surfaces and the embeddings have dihedral automorphism groups.
\begin{prop}
Let $n,k,l$ be deined as in the above construction and let $\psi_l^{(n,k)}$ be the corresponding embedding of the simplicial surface $Y^{(n,k)}.$ Then  
\[
\Aut(\psi_l(Y^{(n,k)})) \cong \Aut(\psi^{(n,k)}) \cong D_n. 
\]
\end{prop}
\begin{proof}
Let $\alpha$ be defined by $\frac{2\pi l}{n}.$ Furthermore let $S$ be defined by 
\[\
\left(\begin{tabular}{ccc}
    1&0&0  \\
    0&-1&0 \\
    0&0&1 \\ 
\end{tabular}\right).
\]
Then $S$ defines a symmetry of the embedding of $Y^{(n,k)}$ with $S^2=I_3.$ 
Thus, $\langle M_\alpha ,S\rangle$ is a subgroup of the symmetry group $\Aut(\psi_l(X^{(n,k))}).$ Since $\langle M_\alpha ,S\rangle \cong D_n,$ it suffices to show that 
the automorphism group of $Y^{(n,k)}$ is also dihedral of order $2n.$
Let $\phi$ be an automorphism of $Y^{(n,k)}$ and $(k+1)n+1,\ldots,(k+2)n$ the vertices of degree 3 in the simplical surface.  Since vertices that are mapped onto each other under an automorphism of the surface must have the same vertex degree, $\phi$ permutes the vertices of degree 3, i.e.\ there exists an $1\leq i \leq n$ with $\phi((k+1)n+1)=(k+1)n+i.$
Furthermore, the image of the face $\{kn+1,kn+2,(k+1)n+i\}$ under $\phi$ must again be a face of the surface $Y^{(n,k)}$. Examining the vertex degrees of the surfaces leads to $\phi(kn+1)=kn+i$ or $\phi(kn+1)=kn+i+1.$

If $\phi(kn+1)=kn+i,$ then $\phi(kn+2)=kn+i+1$ and the restriction of $\phi$ to the vertices of the vertices is given by 
\[
((1,\ldots,n)\ldots((k+2)n+1,\ldots,(k+3)n))^i.
\]
If $\psi(kn+1)=kn+i+1,$ then $\phi(kn+2)=kn+i$ and the restriction of $\phi$ to the vertices of the surface is given by 
\[
((1,\ldots,n)\ldots((k+2)n+1,\ldots,(k+3)n))^i\phi_S.
\]
whereby $\phi_S$ is an involution that arises from $S$ by embedding the symmetry group of the embedding into the automorphism group of the simplicial surface. 
Hence the automorphism group of $Y^{(n,k)}$ is dihedral.
\end{proof}

\section{Acknowledgments}
The authors gratefully acknowledge the funding by the Deutsche Forschungsgemeinschaft (DFG, German Research Foundation) in the framework of the Collaborative Research Centre CRC/TRR 280 “Design Strategies for Material-Minimized Carbon Reinforced Concrete Structures – Principles of a New Approach to Construction” (project ID 417002380).
The second author was partially supported by the FY2022 JSPS Postdoctoral Fellowship for Research in Japan (Short-term), ID PE22747. The authors thank Alice C. Niemeyer for useful comments and fruitful discussions.
\newpage
\printbibliography
\end{document}

%% file: tikz_header_graph.tex
\usepackage{amsmath, amssymb, amsfonts} 

\usepackage{ifthen}

\usepackage{tikz}
\usetikzlibrary{calc}
\usetikzlibrary{positioning}
\usetikzlibrary{shapes}
\usetikzlibrary{patterns}
\usepackage{circuitikz}

\makeatletter
\edef\texforht{TT\noexpand\fi
  \@ifpackageloaded{tex4ht}
    {\noexpand\iftrue}
    {\noexpand\iffalse}}
\makeatother

\makeatletter
\newif\iftikz@node@phantom
\tikzset{
  phantom/.is if=tikz@node@phantom,
  text/.code=%
    \edef\tikz@temp{#1}%
    \ifx\tikz@temp\tikz@nonetext
      \tikz@node@phantomtrue
    \else
      \tikz@node@phantomfalse
      \let\tikz@textcolor\tikz@temp
    \fi
}
\usepackage{etoolbox}
\patchcmd\tikz@fig@continue{\tikz@node@transformations}{%
  \iftikz@node@phantom
    \setbox\pgfnodeparttextbox\hbox{}
  \fi\tikz@node@transformations}{}{}
\makeatother

\newcommand{\tikzAngleOfLine}{\tikz@AngleOfLine}
\def\tikz@AngleOfLine(#1)(#2)#3{%
  \pgfmathanglebetweenpoints{%
    \pgfpointanchor{#1}{center}}{%
    \pgfpointanchor{#2}{center}}
  \pgfmathsetmacro{#3}{\pgfmathresult}%
}

\tikzset{ 
    vertexNodePlain/.style = {fill=none, shape=circle, inner sep=0pt, minimum size=2pt, text=none},
    vertexNodePlain/.default=white,
    vertexPlain/labels/.style = {
        vertexNode/.style={vertexNodePlain=##1},
        vertexLabel/.style={gray}
    },
    vertexPlain/nolabels/.style = {
        vertexNode/.style={vertexNodePlain=##1},
        vertexLabel/.style={text=none}
    },
    vertexPlain/.style = vertexPlain/#1,
    vertexPlain/.default=labels
}
\tikzset{
    vertexNodeNormal/.style = {fill=none, shape=circle, inner sep=0pt, minimum size=4pt, text=none},
    vertexNodeNormal/.default = blue,
    vertexNormal/labels/.style = {
        vertexNode/.style={vertexNodeNormal=##1},
        vertexLabel/.style={blue}
    },
    vertexNormal/nolabels/.style = {
        vertexNode/.style={vertexNodeNormal=##1},
        vertexLabel/.style={text=none}
    },
    vertexNormal/.style = vertexNormal/#1,
    vertexNormal/.default=labels
}
\tikzset{
    vertexNodeBallShading/pdf/.style = {ball color=#1},
    vertexNodeBallShading/svg/.style = {fill=#1},
    vertexNodeBallShading/.code = {
        \if\texforht
            \tikzset{vertexNodeBallShading/svg=white}
        \else
            \tikzset{vertexNodeBallShading/pdf=white}
        \fi
    },
    vertexNodeBall/.style = {shape=circle, vertexNodeBallShading=#1, inner sep=2pt, outer sep=0pt, minimum size=3pt, font=\tiny},
    vertexNodeBall/.default = white,
    vertexBall/labels/.style = {
        vertexNode/.style={vertexNodeBall=##1, text=black},
        vertexLabel/.style={text=none}
    },
    vertexBall/nolabels/.style = {
        vertexNode/.style={vertexNodeBall=##1, text=none},
        vertexLabel/.style={text=none}
    },
    vertexBall/.style = vertexBall/#1,
    vertexBall/.default=labels
}
\tikzset{ 
    vertexStyle/.style={vertexNormal=#1},
    vertexStyle/.default = labels
}

\newcommand{\vertexLabelR}[4][]{
    \ifthenelse{ \equal{#1}{} }
        { \node[vertexNode] at (#2) {#4}; }
        { \node[vertexNode=#1] at (#2) {#4}; }
    \node[vertexLabel, #3] at (#2) {#4};
}
\newcommand{\vertexLabelA}[4][]{
    \ifthenelse{ \equal{#1}{} }
        { \node[vertexNode] at (#2) {#4}; }
        { \node[vertexNode=#1] at (#2) {#4}; }
    \node[vertexLabel] at (#3) {#4};
}

\newcommand{\edgeLabelColor}{blue!20!white}
\tikzset{
    edgeLineNone/.style = {draw=none},
    edgeLineNone/.default=black,
    edgeNone/labels/.style = {
        edge/.style = {edgeLineNone=##1},
        edgeLabel/.style = {fill=\edgeLabelColor,font=\small}
    },
    edgeNone/nolabels/.style = {
        edge/.style = {edgeLineNone=##1},
        edgeLabel/.style = {text=none}
    },
    edgeNone/.style = edgeNone/#1,
    edgeNone/.default = labels
}
\tikzset{
    edgeLinePlain/.style={line join=round, draw=#1},
    edgeLinePlain/.default=black,
    edgePlain/labels/.style = {
        edge/.style={edgeLinePlain=##1},
        edgeLabel/.style={fill=\edgeLabelColor,font=\small}
    },
    edgePlain/nolabels/.style = {
        edge/.style={edgeLinePlain=##1},
        edgeLabel/.style={text=none}
    },
    edgePlain/.style = edgePlain/#1,
    edgePlain/.default = labels
}
\tikzset{
    edgeLineDouble/.style = {very thin, double=#1, double distance=.8pt, line join=round},
    edgeLineDouble/.default=gray!90!white,
    edgeDouble/labels/.style = {
        edge/.style = {edgeLineDouble=##1},
        edgeLabel/.style = {fill=\edgeLabelColor,font=\small}
    },
    edgeDouble/nolabels/.style = {
        edge/.style = {edgeLineDouble=##1},
        edgeLabel/.style = {text=none}
    },
    edgeDouble/.style = edgeDouble/#1,
    edgeDouble/.default = labels
}
\tikzset{
    edgeStyle/.style = {edgePlain=#1},
    edgeStyle/.default = labels
}

\newcommand{\faceColorY}{yellow!60!white}   
\newcommand{\faceColorB}{blue!60!white}     
\newcommand{\faceColorC}{cyan!60}           
\newcommand{\faceColorR}{red!60!white}      
\newcommand{\faceColorG}{green!60!white}    
\newcommand{\faceColorO}{orange!50!yellow!70!white} 

\newcommand{\faceColor}{\faceColorY}
\newcommand{\faceColorSwap}{\faceColorC}

\tikzset{
    face/.style = {fill=#1},
    face/.default = \faceColor,
    faceY/.style = {face=\faceColorY},
    faceB/.style = {face=\faceColorB},
    faceC/.style = {face=\faceColorC},
    faceR/.style = {face=\faceColorR},
    faceG/.style = {face=\faceColorG},
    faceO/.style = {face=\faceColorO}
}
\tikzset{
    faceStyle/labels/.style = {
        faceLabel/.style = {}
    },
    faceStyle/nolabels/.style = {
        faceLabel/.style = {text=none}
    },
    faceStyle/.style = faceStyle/#1,
    faceStyle/.default = labels
}
\tikzset{ face/.style={fill=#1} }
\tikzset{ faceSwap/.code=
    \ifdefined\swapColors
        \tikzset{face=\faceColorSwap}
    \else
        \tikzset{face=\faceColor}
    \fi
}

%% file: petersen2.tex
\begin{tikzpicture}[vertexBall, edgeDouble=nolabels, faceStyle=nolabels, scale=2]

\coordinate (V1) at ( 1.902,0.618 );
\coordinate (V2) at ( 0.,2.);
\coordinate (V3) at ( -1.902,0.618 );
\coordinate (V4) at (-1.174,-1.618 );
\coordinate (V5) at ( 1.174,-1.618);
\coordinate (V6) at ( 0.7416,0.2408727272727273);
\coordinate (V7) at (0.,0.7795636363636363 );
\coordinate (V8) at (-0.7416,0.240872727272727);
\coordinate (V9) at ( -0.4572,-0.6306545454545456 );
\coordinate (V10) at (0.4572,-0.6306545454545456);
\draw[edge=] (V1) -- node[edgeLabel] {$1$} (V2);
\draw[edge=] (V1) -- node[edgeLabel] {$2$} (V5);
\draw[edge=] (V1) -- node[edgeLabel] {$3$} (V6);
\draw[edge=] (V2) -- node[edgeLabel] {$4$} (V3);
\draw[edge=] (V2) -- node[edgeLabel] {$5$} (V7);
\draw[edge=] (V3) -- node[edgeLabel] {$6$} (V4);
\draw[edge=] (V3) -- node[edgeLabel] {$7$} (V8);
\draw[edge=] (V4) -- node[edgeLabel] {$8$} (V5);
\draw[edge=] (V4) -- node[edgeLabel] {$9$} (V9);
\draw[edge=] (V5) -- node[edgeLabel] {$10$} (V10);
\draw[edge=] (V6) -- node[edgeLabel] {$11$} (V8);
\draw[edge=] (V6) -- node[edgeLabel] {$12$} (V9);
\draw[edge=] (V7) -- node[edgeLabel] {$13$} (V9);
\draw[edge=] (V7) -- node[edgeLabel] {$14$} (V10);
\draw[edge=] (V8) -- node[edgeLabel] {$15$} (V10);
\vertexLabelR[]{V1}{left}{$ $}
\vertexLabelR[]{V2}{left}{$ $}
\vertexLabelR[]{V3}{left}{$ $}
\vertexLabelR[]{V4}{left}{$ $}
\vertexLabelR[]{V5}{left}{$ $}
\vertexLabelR[]{V6}{left}{$ $}
\vertexLabelR[]{V7}{left}{$ $}
\vertexLabelR[]{V8}{left}{$ $}
\vertexLabelR[]{V9}{left}{$ $}
\vertexLabelR[]{V10}{left}{$ $}
\node at (1.902+0.2,0.618 ) {$1$};
\node at (0.+0.2,2. ) {$2$};
\node at ( -1.902-0.2,0.618) {$3$};
\node at (-1.174-0.2,-1.618) {$4$};
\node at (1.174+0.2,-1.618) {$5$};
\node at (0.7416,0.2408727272727273+0.2 ) {$6$};
\node at (0.+0.2,0.7795636363636363 ) {$7$};
\node at (-0.7416,0.2408727272727273 +0.2) {$8$};
\node at (-0.4572-0.2,-0.6306545454545456 ) {$9$};
\node at (0.4572+0.2,-0.6306545454545456 ) {$10$};
\end{tikzpicture}

%% file: PetersenSurface.tex
\begin{tikzpicture}[vertexBall, edgeDouble, faceStyle, scale=2]

\coordinate (V1_1) at (0., 0.);
\coordinate (V2_1) at (1., 0.);
\coordinate (V2_2) at (-1.5, -0.8660254037844388);
\coordinate (V3_1) at (-0.5, 0.8660254037844384);
\coordinate (V3_2) at (1.5, -0.8660254037844384);
\coordinate (V3_3) at (-1., 0.);
\coordinate (V4_1) at (1.5, 0.8660254037844388);
\coordinate (V4_2) at (-0.4999999999999998, -0.8660254037844388);
\coordinate (V5_1) at (0.4999999999999999, 0.8660254037844386);
\coordinate (V5_2) at (0., -1.732050807568877);
\coordinate (V6_1) at (0.5000000000000001, -0.8660254037844386);
\coordinate (V6_2) at (0., 1.732050807568877);

\fill[face]  (V2_1) -- (V5_1) -- (V1_1) -- cycle;
\node[faceLabel] at (barycentric cs:V2_1=1,V5_1=1,V1_1=1) {$1$};
\fill[face]  (V1_1) -- (V6_1) -- (V2_1) -- cycle;
\node[faceLabel] at (barycentric cs:V1_1=1,V6_1=1,V2_1=1) {$2$};
\fill[face]  (V1_1) -- (V4_2) -- (V6_1) -- cycle;
\node[faceLabel] at (barycentric cs:V1_1=1,V4_2=1,V6_1=1) {$3$};
\fill[face]  (V1_1) -- (V3_3) -- (V4_2) -- cycle;
\node[faceLabel] at (barycentric cs:V1_1=1,V3_3=1,V4_2=1) {$4$};
\fill[face]  (V5_1) -- (V3_1) -- (V1_1) -- cycle;
\node[faceLabel] at (barycentric cs:V5_1=1,V3_1=1,V1_1=1) {$5$};
\fill[face]  (V2_1) -- (V4_1) -- (V5_1) -- cycle;
\node[faceLabel] at (barycentric cs:V2_1=1,V4_1=1,V5_1=1) {$6$};
\fill[face]  (V6_1) -- (V3_2) -- (V2_1) -- cycle;
\node[faceLabel] at (barycentric cs:V6_1=1,V3_2=1,V2_1=1) {$7$};
\fill[face]  (V4_2) -- (V5_2) -- (V6_1) -- cycle;
\node[faceLabel] at (barycentric cs:V4_2=1,V5_2=1,V6_1=1) {$8$};
\fill[face]  (V3_3) -- (V2_2) -- (V4_2) -- cycle;
\node[faceLabel] at (barycentric cs:V3_3=1,V2_2=1,V4_2=1) {$9$};
\fill[face]  (V5_1) -- (V6_2) -- (V3_1) -- cycle;
\node[faceLabel] at (barycentric cs:V5_1=1,V6_2=1,V3_1=1) {$10$};

\draw[edge] (V2_1) -- node[edgeLabel] {$1$} (V1_1);
\draw[edge] (V1_1) -- node[edgeLabel] {$2$} (V5_1);
\draw[edge] (V5_1) -- node[edgeLabel] {$3$} (V2_1);
\draw[edge] (V6_1) -- node[edgeLabel] {$4$} (V1_1);
\draw[edge] (V2_1) -- node[edgeLabel] {$5$} (V6_1);
\draw[edge] (V4_2) -- node[edgeLabel] {$6$} (V1_1);
\draw[edge] (V6_1) -- node[edgeLabel] {$7$} (V4_2);
\draw[edge] (V1_1) -- node[edgeLabel] {$8$} (V3_1);
\draw[edge] (V3_3) -- node[edgeLabel] {$8$} (V1_1);
\draw[edge] (V4_2) -- node[edgeLabel] {$9$} (V3_3);
\draw[edge] (V3_1) -- node[edgeLabel] {$10$} (V5_1);
\draw[edge] (V5_1) -- node[edgeLabel] {$11$} (V4_1);
\draw[edge] (V5_2) -- node[edgeLabel] {$11$} (V4_2);
\draw[edge] (V4_1) -- node[edgeLabel] {$12$} (V2_1);
\draw[edge] (V4_2) -- node[edgeLabel] {$12$} (V2_2);
\draw[edge] (V2_1) -- node[edgeLabel] {$13$} (V3_2);
\draw[edge] (V2_2) -- node[edgeLabel] {$13$} (V3_3);
\draw[edge] (V3_2) -- node[edgeLabel] {$14$} (V6_1);
\draw[edge] (V3_1) -- node[edgeLabel] {$14$} (V6_2);
\draw[edge] (V6_1) -- node[edgeLabel] {$15$} (V5_2);
\draw[edge] (V6_2) -- node[edgeLabel] {$15$} (V5_1);

\vertexLabelR{V1_1}{left}{$1$}
\vertexLabelR{V2_1}{left}{$2$}
\vertexLabelR{V2_2}{left}{$2$}
\vertexLabelR{V3_1}{left}{$3$}
\vertexLabelR{V3_2}{left}{$3$}
\vertexLabelR{V3_3}{left}{$3$}
\vertexLabelR{V4_1}{left}{$4$}
\vertexLabelR{V4_2}{left}{$4$}
\vertexLabelR{V5_1}{left}{$5$}
\vertexLabelR{V5_2}{left}{$5$}
\vertexLabelR{V6_1}{left}{$6$}
\vertexLabelR{V6_2}{left}{$6$}

\end{tikzpicture}

%% file: A5_16_3.tex
\begin{tikzpicture}[vertexBall, edgeDouble=nolabels, faceStyle=nolabels, scale=1]

\coordinate (V1) at (4.045 , 2.935);
\coordinate (V2) at (-0.1748420622902369 , -0.7357691168895439);
\coordinate (V3) at (-1.729960457502944 , -0.5616977212668609);
\coordinate (V4) at (-1.369909617149586 , -0.4453090771813964);
\coordinate (V5) at (-0.5744810618107785 , 0.4933941411439507);
\coordinate (V6) at (5. , 0.);
\coordinate (V7) at (1.545 , 4.755);
\coordinate (V8) at (0.646606599730662 , -0.3943890453940057);
\coordinate (V9) at (0. , -1.439383005716221);
\coordinate (V10) at (0. , -1.818480065000847);
\coordinate (V11) at (-0.646606599730662 , -0.3943890453940057);
\coordinate (V12) at (-1.545 , 4.755);
\coordinate (V13) at (1.069166739496281 , 1.470937753767285);
\coordinate (V14) at (-4.045 , -2.935);
\coordinate (V15) at (0.7546549248225898 , -0.0617625526108966);
\coordinate (V16) at (-0.2915906744286138 , 0.6985265737504954);
\coordinate (V17) at (-1.545 , -4.755);
\coordinate (V18) at (0.8466668331343571 , 1.165000580039507);
\coordinate (V19) at (-0.8466668331343571 , 1.165000580039507);
\coordinate (V20) at (4.045 , -2.935);
\coordinate (V21) at (0.2915906744286138 , 0.6985265737504954);
\coordinate (V22) at (-0.7546549248225898 , -0.0617625526108966);
\coordinate (V23) at (1.545 , -4.755);
\coordinate (V24) at (-1.069166739496281 , 1.470937753767285);
\coordinate (V25) at (-5. , 0.);
\coordinate (V26) at (0.1748420622902369 , -0.7357691168895439);
\coordinate (V27) at (1.369909617149586 , -0.4453090771813964);
\coordinate (V28) at (1.729960457502944 , -0.5616977212668609);
\coordinate (V29) at (0.5744810618107785 , 0.4933941411439507);
\coordinate (V30) at (-4.045 , 2.935);
\draw[edge] (V1) -- node[edgeLabel] {$1$} (V6);
\draw[edge] (V1) -- node[edgeLabel] {$2$} (V7);
\draw[edge] (V1) -- node[edgeLabel] {$3$} (V13);
\draw[edge] (V2) -- node[edgeLabel] {$4$} (V5);
\draw[edge] (V2) -- node[edgeLabel] {$5$} (V9);
\draw[edge] (V2) -- node[edgeLabel] {$6$} (V26);
\draw[edge] (V3) -- node[edgeLabel] {$7$} (V4);
\draw[edge] (V3) -- node[edgeLabel] {$8$} (V14);
\draw[edge] (V3) -- node[edgeLabel] {$9$} (V25);
\draw[edge] (V4) -- node[edgeLabel] {$10$} (V11);
\draw[edge] (V4) -- node[edgeLabel] {$11$} (V22);
\draw[edge] (V5) -- node[edgeLabel] {$12$} (V16);
\draw[edge] (V5) -- node[edgeLabel] {$13$} (V19);
\draw[edge] (V6) -- node[edgeLabel] {$14$} (V20);
\draw[edge] (V6) -- node[edgeLabel] {$15$} (V28);
\draw[edge] (V7) -- node[edgeLabel] {$16$} (V12);
\draw[edge] (V7) -- node[edgeLabel] {$17$} (V13);
\draw[edge] (V8) -- node[edgeLabel] {$18$} (V11);
\draw[edge] (V8) -- node[edgeLabel] {$19$} (V15);
\draw[edge] (V8) -- node[edgeLabel] {$20$} (V27);
\draw[edge] (V9) -- node[edgeLabel] {$21$} (V10);
\draw[edge] (V9) -- node[edgeLabel] {$22$} (V26);
\draw[edge] (V10) -- node[edgeLabel] {$23$} (V17);
\draw[edge] (V10) -- node[edgeLabel] {$24$} (V23);
\draw[edge] (V11) -- node[edgeLabel] {$25$} (V22);
\draw[edge] (V12) -- node[edgeLabel] {$26$} (V24);
\draw[edge] (V12) -- node[edgeLabel] {$27$} (V30);
\draw[edge] (V13) -- node[edgeLabel] {$28$} (V18);
\draw[edge] (V14) -- node[edgeLabel] {$29$} (V17);
\draw[edge] (V14) -- node[edgeLabel] {$30$} (V25);
\draw[edge] (V15) -- node[edgeLabel] {$31$} (V16);
\draw[edge] (V15) -- node[edgeLabel] {$32$} (V27);
\draw[edge] (V16) -- node[edgeLabel] {$33$} (V19);
\draw[edge] (V17) -- node[edgeLabel] {$34$} (V23);
\draw[edge] (V18) -- node[edgeLabel] {$35$} (V21);
\draw[edge] (V18) -- node[edgeLabel] {$36$} (V29);
\draw[edge] (V19) -- node[edgeLabel] {$37$} (V24);
\draw[edge] (V20) -- node[edgeLabel] {$38$} (V23);
\draw[edge] (V20) -- node[edgeLabel] {$39$} (V28);
\draw[edge] (V21) -- node[edgeLabel] {$40$} (V22);
\draw[edge] (V21) -- node[edgeLabel] {$41$} (V29);
\draw[edge] (V24) -- node[edgeLabel] {$42$} (V30);
\draw[edge] (V25) -- node[edgeLabel] {$43$} (V30);
\draw[edge] (V26) -- node[edgeLabel] {$44$} (V29);
\draw[edge] (V27) -- node[edgeLabel] {$45$} (V28);
\vertexLabelR{V1}{left}{$ $}
\vertexLabelR{V2}{left}{$ $}
\vertexLabelR{V3}{left}{$ $}
\vertexLabelR{V4}{left}{$ $}
\vertexLabelR{V5}{left}{$ $}
\vertexLabelR{V6}{left}{$ $}
\vertexLabelR{V7}{left}{$ $}
\vertexLabelR{V8}{left}{$ $}
\vertexLabelR{V9}{left}{$ $}
\vertexLabelR{V10}{left}{$ $}
\vertexLabelR{V11}{left}{$ $}
\vertexLabelR{V12}{left}{$ $}
\vertexLabelR{V13}{left}{$ $}
\vertexLabelR{V14}{left}{$ $}
\vertexLabelR{V15}{left}{$ $}
\vertexLabelR{V16}{left}{$ $}
\vertexLabelR{V17}{left}{$ $}
\vertexLabelR{V18}{left}{$ $}
\vertexLabelR{V19}{left}{$ $}
\vertexLabelR{V20}{left}{$ $}
\vertexLabelR{V21}{left}{$ $}
\vertexLabelR{V22}{left}{$ $}
\vertexLabelR{V23}{left}{$ $}
\vertexLabelR{V24}{left}{$ $}
\vertexLabelR{V25}{left}{$ $}
\vertexLabelR{V26}{left}{$ $}
\vertexLabelR{V27}{left}{$ $}
\vertexLabelR{V28}{left}{$ $}
\vertexLabelR{V29}{left}{$ $}
\vertexLabelR{V30}{left}{$ $}

\end{tikzpicture}

%% file: torus.tex
\begin{tikzpicture}[vertexBall, edgeDouble, faceStyle, scale=2]

\coordinate (V1_1) at (0., 0.);
\coordinate (V2_1) at (1., 0.);
\coordinate (V3_1) at (0.4999999999999999, 0.8660254037844386);
\coordinate (V4_1) at (-0.5, 0.8660254037844384);
\coordinate (V4_2) at (1.5, -0.8660254037844384);
\coordinate (V5_1) at (-0.9999999999999996, 0.);
\coordinate (V5_2) at (1.5, 0.8660254037844388);
\coordinate (V6_1) at (0.5000000000000001, -0.8660254037844386);
\coordinate (V6_2) at (0.9999999999999996, 1.732050807568877);
\coordinate (V6_3) at (-1.5, 0.8660254037844382);
\coordinate (V7_1) at (-0.4999999999999997, -0.8660254037844383);
\coordinate (V7_2) at (2., 0.);
\coordinate (V7_3) at (0., 1.732050807568877);

\fill[face]  (V2_1) -- (V3_1) -- (V1_1) -- cycle;
\node[faceLabel] at (barycentric cs:V2_1=1,V3_1=1,V1_1=1) {$1$};
\fill[face]  (V3_1) -- (V4_1) -- (V1_1) -- cycle;
\node[faceLabel] at (barycentric cs:V3_1=1,V4_1=1,V1_1=1) {$2$};
\fill[face]  (V6_1) -- (V4_2) -- (V2_1) -- cycle;
\node[faceLabel] at (barycentric cs:V6_1=1,V4_2=1,V2_1=1) {$3$};
\fill[face]  (V2_1) -- (V5_2) -- (V3_1) -- cycle;
\node[faceLabel] at (barycentric cs:V2_1=1,V5_2=1,V3_1=1) {$4$};
\fill[face]  (V5_2) -- (V6_2) -- (V3_1) -- cycle;
\node[faceLabel] at (barycentric cs:V5_2=1,V6_2=1,V3_1=1) {$5$};
\fill[face]  (V3_1) -- (V7_3) -- (V4_1) -- cycle;
\node[faceLabel] at (barycentric cs:V3_1=1,V7_3=1,V4_1=1) {$6$};
\fill[face]  (V6_2) -- (V7_3) -- (V3_1) -- cycle;
\node[faceLabel] at (barycentric cs:V6_2=1,V7_3=1,V3_1=1) {$7$};
\fill[face]  (V4_2) -- (V7_2) -- (V2_1) -- cycle;
\node[faceLabel] at (barycentric cs:V4_2=1,V7_2=1,V2_1=1) {$8$};
\fill[face]  (V2_1) -- (V7_2) -- (V5_2) -- cycle;
\node[faceLabel] at (barycentric cs:V2_1=1,V7_2=1,V5_2=1) {$9$};
\fill[face]  (V1_1) -- (V6_1) -- (V2_1) -- cycle;
\node[faceLabel] at (barycentric cs:V1_1=1,V6_1=1,V2_1=1) {$10$};
\fill[face]  (V5_1) -- (V7_1) -- (V1_1) -- cycle;
\node[faceLabel] at (barycentric cs:V5_1=1,V7_1=1,V1_1=1) {$11$};
\fill[face]  (V4_1) -- (V5_1) -- (V1_1) -- cycle;
\node[faceLabel] at (barycentric cs:V4_1=1,V5_1=1,V1_1=1) {$12$};
\fill[face]  (V4_1) -- (V6_3) -- (V5_1) -- cycle;
\node[faceLabel] at (barycentric cs:V4_1=1,V6_3=1,V5_1=1) {$13$};
\fill[face]  (V1_1) -- (V7_1) -- (V6_1) -- cycle;
\node[faceLabel] at (barycentric cs:V1_1=1,V7_1=1,V6_1=1) {$14$};

\draw[edge] (V2_1) -- node[edgeLabel] {$1$} (V1_1);
\draw[edge] (V1_1) -- node[edgeLabel] {$2$} (V3_1);
\draw[edge] (V1_1) -- node[edgeLabel] {$3$} (V4_1);
\draw[edge] (V1_1) -- node[edgeLabel] {$4$} (V5_1);
\draw[edge] (V6_1) -- node[edgeLabel] {$5$} (V1_1);
\draw[edge] (V1_1) -- node[edgeLabel] {$6$} (V7_1);
\draw[edge] (V3_1) -- node[edgeLabel] {$7$} (V2_1);
\draw[edge] (V2_1) -- node[edgeLabel] {$8$} (V4_2);
\draw[edge] (V5_2) -- node[edgeLabel] {$9$} (V2_1);
\draw[edge] (V2_1) -- node[edgeLabel] {$10$} (V6_1);
\draw[edge] (V7_2) -- node[edgeLabel] {$11$} (V2_1);
\draw[edge] (V4_1) -- node[edgeLabel] {$12$} (V3_1);
\draw[edge] (V3_1) -- node[edgeLabel] {$13$} (V5_2);
\draw[edge] (V3_1) -- node[edgeLabel] {$14$} (V6_2);
\draw[edge] (V7_3) -- node[edgeLabel] {$15$} (V3_1);
\draw[edge] (V5_1) -- node[edgeLabel] {$16$} (V4_1);
\draw[edge] (V4_2) -- node[edgeLabel] {$17$} (V6_1);
\draw[edge] (V6_3) -- node[edgeLabel] {$17$} (V4_1);
\draw[edge] (V7_2) -- node[edgeLabel] {$18$} (V4_2);
\draw[edge] (V4_1) -- node[edgeLabel] {$18$} (V7_3);
\draw[edge] (V6_2) -- node[edgeLabel] {$19$} (V5_2);
\draw[edge] (V5_1) -- node[edgeLabel] {$19$} (V6_3);
\draw[edge] (V7_1) -- node[edgeLabel] {$20$} (V5_1);
\draw[edge] (V5_2) -- node[edgeLabel] {$20$} (V7_2);
\draw[edge] (V6_1) -- node[edgeLabel] {$21$} (V7_1);
\draw[edge] (V7_3) -- node[edgeLabel] {$21$} (V6_2);

\vertexLabelR{V1_1}{left}{$1$}
\vertexLabelR{V2_1}{left}{$2$}
\vertexLabelR{V3_1}{left}{$3$}
\vertexLabelR{V4_1}{left}{$4$}
\vertexLabelR{V4_2}{left}{$4$}
\vertexLabelR{V5_1}{left}{$5$}
\vertexLabelR{V5_2}{left}{$5$}
\vertexLabelR{V6_1}{left}{$6$}
\vertexLabelR{V6_2}{left}{$6$}
\vertexLabelR{V6_3}{left}{$6$}
\vertexLabelR{V7_1}{left}{$7$}
\vertexLabelR{V7_2}{left}{$7$}
\vertexLabelR{V7_3}{left}{$7$}

\end{tikzpicture}

%% file: frucht_graph_part4.tex
	\begin{tikzpicture}[vertexBall, edgeDouble=nolabels, faceStyle=nolabels, scale=0.6]

	\coordinate (V4) at (0 , 0);
	
	\coordinate (V9) at (10 , 3);
	\coordinate (V10) at (10 , -2);
	\coordinate (V11) at (10, -4);
	\coordinate (V12) at (10 , -6);

	\draw[edge=black] (V4) -- node[edgeLabel] {$1$} (V9);
	\draw[edge=black] (V4) -- node[edgeLabel] {$2$} (V10);
	\draw[edge=black] (V4) -- node[edgeLabel] {$3$} (V11);
	\draw[edge=black] (V4) -- node[edgeLabel] {$4$} (V12);
    \draw[dotted] (V10) -- node[edgeLabel] {$5$} (V11);


	\vertexLabelR[]{V4}{left}{$ $}
	\vertexLabelR[]{V9}{left}{$ $}
	\vertexLabelR[]{V10}{left}{$ $}
	\vertexLabelR[]{V11}{left}{$ $}
	\vertexLabelR[]{V12}{left}{$ $}

	\def\x{0.5}
	
	\node at (0-\x , 0) {$g$};
	
	\node at (10+3*\x , 3) {$g\cdot g_1$};
	\node at (10+3*\x , -2) {$g\cdot g_1^{-1}$};
	\node at (10+3*\x, -4) {$g\cdot g_n$};
	\node at (10+3*\x , -6) {$g\cdot g_n^{-1}$};

	\end{tikzpicture}

%% file: frucht_graph_part1.tex
	\begin{tikzpicture}[vertexBall, edgeDouble=nolabels, faceStyle=nolabels, scale=0.5]

	\coordinate (V1) at (3 , 3);
	\coordinate (V2) at (3 , 0);
	\coordinate (V3) at (8 , 0);
	\coordinate (V4) at (0 , 0);
	\coordinate (V5) at (8 , 3);
	\coordinate (V6) at (8 , -2);
	\coordinate (V7) at (8 , -4);
	\coordinate (V8) at (8 , -6);
	\coordinate (V9) at (10 , 3);
	\coordinate (V10) at (10 , -2);
	\coordinate (V11) at (10, -4);
	\coordinate (V12) at (10 , -6);

	\draw[edge=red] (V1) -- node[edgeLabel] {$1$} (V4);
	\draw[edge=blue] (V2) -- node[edgeLabel] {$2$} (V1);
	\draw[edge=green] (V2) -- node[edgeLabel] {$3$} (V4);

	\draw[edge=blue] (V4) -- node[edgeLabel] {$4$} (V8);
	\draw[edge=green] (V8) -- node[edgeLabel] {$5$} (V7);
	\draw[edge=red] (V2) -- node[edgeLabel] {$6$} (V3);

	\draw[edge=blue] (V5) -- node[edgeLabel] {$4$} (V3);
	\draw[edge=green] (V1) -- node[edgeLabel] {$5$} (V5);
	\draw[edge=red] (V5) -- node[edgeLabel] {$6$} (V9);

	\draw[edge=blue, dotted] (V6) -- node[edgeLabel] {$4$} (V7);
	\draw[edge=green] (V3) -- node[edgeLabel] {$5$} (V6);
	\draw[edge=red] (V7) -- node[edgeLabel] {$6$} (V11);
	\draw[edge=red] (V8) -- node[edgeLabel] {$6$} (V12);
	\draw[edge=red] (V6) -- node[edgeLabel] {$6$} (V10);

	\vertexLabelR[]{V1}{left}{$ $}
	\vertexLabelR[]{V2}{left}{$ $}
	\vertexLabelR[]{V3}{left}{$ $}
	\vertexLabelR[]{V4}{left}{$ $}
	\vertexLabelR[]{V5}{left}{$ $}
	\vertexLabelR[]{V6}{left}{$ $}
	\vertexLabelR[]{V7}{left}{$ $}
	\vertexLabelR[]{V8}{left}{$ $}
	\vertexLabelR[]{V9}{left}{$ $}
	\vertexLabelR[]{V10}{left}{$ $}
	\vertexLabelR[]{V11}{left}{$ $}
	\vertexLabelR[]{V12}{left}{$ $}

	\def\x{0.5}
	\node at (3 , 3+\x) {$x_{1,g}$};
	\node at (3 , 0-\x) {$x_{2,g}$};
	\node at (8-2*\x , 0+\x) {$x_{3,g}$};
	\node at (0-2*\x , 0) {$x_{4,g}$};
	\node at (8 , 3+\x) {$x_{5,g}$};
	\node at (8-2*\x , -2) {$x_{6,g}$};
	\node at (8-3*\x , -4+\x) {$x_{2n+3,g}$};
	\node at (8 , -6-\x) {$x_{2n+4,g}$};
	\node at (10+3*\x , 3) {$x_{6,g\cdot g_1}$};
	\node at (10+3*\x , -2) {$x_{5,g\cdot g_1^{-1}}$};
	\node at (10+4*\x, -4) {$x_{2n+4,g\cdot g_n}$};
	\node at (10+4*\x , -6) {$x_{2n+3,g\cdot g_n^{-1}}$};

	\end{tikzpicture}

%% file: frucht_graph_part2.tex
	\begin{tikzpicture}[vertexBall, edgeDouble=nolabels, faceStyle=nolabels, scale=0.5]

	\coordinate (V3) at (8 , 0);
	\coordinate (V4) at (0 , 0);
	\coordinate (V5) at (8 , 3);
	\coordinate (V6) at (8 , -2);
	\coordinate (V7) at (8 , -4);
	\coordinate (V8) at (8 , -6);
	\coordinate (V9) at (10 , 3);
	\coordinate (V10) at (10 , -2);
	\coordinate (V11) at (10, -4);
	\coordinate (V12) at (10 , -6);

	\draw[edge=blue] (V4) -- node[edgeLabel] {$4$} (V8);
	\draw[edge=green] (V8) -- node[edgeLabel] {$5$} (V7);
	\draw[edge=red] (V4) -- node[edgeLabel] {$6$} (V3);

	\draw[edge=blue] (V5) -- node[edgeLabel] {$4$} (V3);
	\draw[edge=green] (V4) -- node[edgeLabel] {$5$} (V5);
	\draw[edge=red] (V5) -- node[edgeLabel] {$6$} (V9);

	\draw[edge=blue] (V6) -- node[edgeLabel] {$4$} (V7);
	\draw[edge=green] (V3) -- node[edgeLabel] {$5$} (V6);
	\draw[edge=red] (V7) -- node[edgeLabel] {$6$} (V11);
	\draw[edge=red] (V8) -- node[edgeLabel] {$6$} (V12);
	\draw[edge=red] (V6) -- node[edgeLabel] {$6$} (V10);

	\vertexLabelR[]{V3}{left}{$ $}
	\vertexLabelR[]{V4}{left}{$ $}
	\vertexLabelR[]{V5}{left}{$ $}
	\vertexLabelR[]{V6}{left}{$ $}
	\vertexLabelR[]{V7}{left}{$ $}
	\vertexLabelR[]{V8}{left}{$ $}
	\vertexLabelR[]{V9}{left}{$ $}
	\vertexLabelR[]{V10}{left}{$ $}
	\vertexLabelR[]{V11}{left}{$ $}
	\vertexLabelR[]{V12}{left}{$ $}

	\def\x{0.5}
	\node at (8-2*\x , 0+\x) {$x_{1,g}$};
	\node at (0-2*\x , 0) {$x_{2,g}$};
	\node at (8 , 3+\x) {$x_{3,g}$};
	\node at (8-2*\x , -2) {$x_{4,g}$};
	\node at (8-2*\x , -4) {$x_{5,g}$};
	\node at (8 , -6-\x) {$x_{6,g}$};
	\node at (10+3*\x , 3) {$x_{4,g\cdot g_1}$};
	\node at (10+3*\x , -2) {$x_{3,g\cdot g_1^{-1}}$};
	\node at (10+3*\x, -4) {$x_{6,g\cdot g_2}$};
	\node at (10+3*\x , -6) {$x_{5,g\cdot g_2^{-1}}$};

	\end{tikzpicture}

%% file: frucht_graph_part3.tex
	\begin{tikzpicture}[vertexBall, edgeDouble=nolabels, faceStyle=nolabels, scale=0.5]

	\coordinate (V2) at (3 , 0);
	\coordinate (V3) at (8 , 0);
	\coordinate (V4) at (0 , 0);
	\coordinate (V5) at (8 , 3);
	\coordinate (V6) at (8 , -2);
	\coordinate (V7) at (8 , -4);
	\coordinate (V8) at (8 , -6);
	\coordinate (V9) at (10 , 3);
	\coordinate (V10) at (10 , -2);
	\coordinate (V11) at (10, -4);
	\coordinate (V12) at (10 , -6);

	\coordinate (V13) at (3 , 1.5);
	\coordinate (V14) at (1.5 , 1.5);
	\coordinate (V15) at (5.5 , 3);

	\draw[edge=red] (V13) -- node[edgeLabel] {$1$} (V15);
	\draw[edge=blue] (V14) -- node[edgeLabel] {$2$} (V15);
	\draw[edge=green] (V13) -- node[edgeLabel] {$3$} (V14);	

	\draw[edge=red] (V14) -- node[edgeLabel] {$1$} (V4);
	\draw[edge=blue] (V2) -- node[edgeLabel] {$2$} (V13);
	\draw[edge=green] (V2) -- node[edgeLabel] {$3$} (V4);

	\draw[edge=blue] (V4) -- node[edgeLabel] {$4$} (V8);
	\draw[edge=green] (V8) -- node[edgeLabel] {$5$} (V7);
	\draw[edge=red] (V2) -- node[edgeLabel] {$6$} (V3);

	\draw[edge=blue] (V5) -- node[edgeLabel] {$4$} (V3);
	\draw[edge=green] (V15) -- node[edgeLabel] {$5$} (V5);
	\draw[edge=red] (V5) -- node[edgeLabel] {$6$} (V9);

	\draw[edge=blue, dotted] (V6) -- node[edgeLabel] {$4$} (V7);
	\draw[edge=green] (V3) -- node[edgeLabel] {$5$} (V6);
	\draw[edge=red] (V7) -- node[edgeLabel] {$6$} (V11);
	\draw[edge=red] (V8) -- node[edgeLabel] {$6$} (V12);
	\draw[edge=red] (V6) -- node[edgeLabel] {$6$} (V10);

	
	\vertexLabelR[]{V2}{left}{$ $}
	\vertexLabelR[]{V3}{left}{$ $}
	\vertexLabelR[]{V4}{left}{$ $}
	\vertexLabelR[]{V5}{left}{$ $}
	\vertexLabelR[]{V6}{left}{$ $}
	\vertexLabelR[]{V7}{left}{$ $}
	\vertexLabelR[]{V8}{left}{$ $}
	\vertexLabelR[]{V9}{left}{$ $}
	\vertexLabelR[]{V10}{left}{$ $}
	\vertexLabelR[]{V11}{left}{$ $}
	\vertexLabelR[]{V12}{left}{$ $}
	\vertexLabelR[]{V13}{left}{$ $}
	\vertexLabelR[]{V14}{left}{$ $}
	\vertexLabelR[]{V15}{left}{$ $}

	\def\x{0.5}
	\node at (1.5-\x , 1.5+\x) {$x_{1,g}$};
	\node at (3+2*\x , 1.5-\x) {$x_{2,g}$};
	\node at (5.5 , 3+\x) {$x_{3,g}$};

	\node at (3 , 0-\x) {$x_{4,g}$};
	\node at (8-2*\x , 0+\x) {$x_{5,g}$};
	\node at (0-2*\x , 0) {$x_{6,g}$};
	\node at (8 , 3+\x) {$x_{7,g}$};
	\node at (8-2*\x , -2) {$x_{8,g}$};
	\node at (8-3*\x , -4+\x) {$x_{2n+5,g}$};
	\node at (8 , -6-\x) {$x_{2n+6,g}$};
	\node at (10+3*\x , 3) {$x_{8,g\cdot g_1}$};
	\node at (10+3*\x , -2) {$x_{7,g\cdot g_1^{-1}}$};
	\node at (10+4*\x, -4) {$x_{2n+6,g\cdot g_n}$};
	\node at (10+4*\x , -6) {$x_{2n+5,g\cdot g_n^{-1}}$};

	\end{tikzpicture}

%% file: Q8.tex
\begin{tikzpicture}[vertexBall, edgeDouble=nolabels, faceStyle=nolabels, scale=2]

\coordinate (V1) at (5.444444444444445 , 1.083333333333333);
\coordinate (V2) at (5.555555555555555 , 0.);
\coordinate (V3) at (5.444444444444445 , -1.083333333333333);
\coordinate (V4) at (5.127777777777778 , 2.122222222222222);
\coordinate (V5) at (5.127777777777778 , -2.122222222222222);
\coordinate (V6) at (4.278883089770355 , 0.3851774530271399);
\coordinate (V7) at (3.54133611691023 , 1.027129436325678);
\coordinate (V8) at (4.616666666666667 , 3.083333333333333);
\coordinate (V9) at (-1.083333333333333 , 5.444444444444445);
\coordinate (V10) at (0. , 5.555555555555555);
\coordinate (V11) at (1.083333333333333 , 5.444444444444445);
\coordinate (V12) at (-2.122222222222222 , 5.127777777777778);
\coordinate (V13) at (2.122222222222222 , 5.127777777777778);
\coordinate (V14) at (-0.3851774530271399 , 4.278883089770355);
\coordinate (V15) at (-1.027129436325678 , 3.54133611691023);
\coordinate (V16) at (-3.083333333333333 , 4.616666666666667);
\coordinate (V17) at (-1.931221294363257 , 2.394342379958247);
\coordinate (V18) at (-2.770480167014614 , 2.73866388308977);
\coordinate (V19) at (-3.861722338204593 , 3.455688935281837);
\coordinate (V20) at (-2.518496868475992 , 2.365960334029228);
\coordinate (V21) at (-0.5046868475991649 , 2.078402922755741);
\coordinate (V22) at (-4.616666666666667 , 3.083333333333333);
\coordinate (V23) at (-3.927777777777778 , 3.927777777777778);
\coordinate (V24) at (-2.853789144050105 , 1.964874739039666);
\coordinate (V25) at (1.931221294363257 , -2.394342379958247);
\coordinate (V26) at (2.770480167014614 , -2.73866388308977);
\coordinate (V27) at (3.861722338204593 , -3.455688935281837);
\coordinate (V28) at (2.518496868475992 , -2.365960334029228);
\coordinate (V29) at (0.5046868475991649 , -2.078402922755741);
\coordinate (V30) at (4.616666666666667 , -3.083333333333333);
\coordinate (V31) at (3.927777777777778 , -3.927777777777778);
\coordinate (V32) at (2.853789144050105 , -1.964874739039666);
\coordinate (V33) at (1.083333333333333 , -5.444444444444445);
\coordinate (V34) at (0. , -5.555555555555555);
\coordinate (V35) at (-1.083333333333333 , -5.444444444444445);
\coordinate (V36) at (2.122222222222222 , -5.127777777777778);
\coordinate (V37) at (-2.122222222222222 , -5.127777777777778);
\coordinate (V38) at (0.3851774530271399 , -4.278883089770355);
\coordinate (V39) at (1.027129436325678 , -3.54133611691023);
\coordinate (V40) at (3.083333333333333 , -4.616666666666667);
\coordinate (V41) at (-5.444444444444445 , -1.083333333333333);
\coordinate (V42) at (-5.555555555555555 , 0.);
\coordinate (V43) at (-5.444444444444445 , 1.083333333333333);
\coordinate (V44) at (-5.127777777777778 , -2.122222222222222);
\coordinate (V45) at (-5.127777777777778 , 2.122222222222222);
\coordinate (V46) at (-4.278883089770355 , -0.3851774530271399);
\coordinate (V47) at (-3.54133611691023 , -1.027129436325678);
\coordinate (V48) at (-4.616666666666667 , -3.083333333333333);
\coordinate (V49) at (-2.394342379958247 , -1.931221294363257);
\coordinate (V50) at (-2.73866388308977 , -2.770480167014614);
\coordinate (V51) at (-3.455688935281837 , -3.861722338204593);
\coordinate (V52) at (-2.365960334029228 , -2.518496868475992);
\coordinate (V53) at (-2.078402922755741 , -0.5046868475991649);
\coordinate (V54) at (-3.083333333333333 , -4.616666666666667);
\coordinate (V55) at (-3.927777777777778 , -3.927777777777778);
\coordinate (V56) at (-1.964874739039666 , -2.853789144050105);
\coordinate (V57) at (2.394342379958247 , 1.931221294363257);
\coordinate (V58) at (2.73866388308977 , 2.770480167014614);
\coordinate (V59) at (3.455688935281837 , 3.861722338204593);
\coordinate (V60) at (2.365960334029228 , 2.518496868475992);
\coordinate (V61) at (2.078402922755741 , 0.5046868475991649);
\coordinate (V62) at (3.083333333333333 , 4.616666666666667);
\coordinate (V63) at (3.927777777777778 , 3.927777777777778);
\coordinate (V64) at (1.964874739039666 , 2.853789144050105);
\draw[edge=blue] (V1) -- node[edgeLabel] {$1$} (V2);
\draw[edge=red] (V1) -- node[edgeLabel] {$2$} (V4);
\draw[edge=green] (V1) -- node[edgeLabel] {$3$} (V5);
\draw[edge=red] (V2) -- node[edgeLabel] {$4$} (V3);
\draw[edge=green] (V2) -- node[edgeLabel] {$5$} (V4);
\draw[edge=blue] (V3) -- node[edgeLabel] {$6$} (V5);
\draw[edge=green] (V3) -- node[edgeLabel] {$7$} (V6);
\draw[edge=blue] (V4) -- node[edgeLabel] {$8$} (V8);
\draw[edge=red] (V5) -- node[edgeLabel] {$9$} (V30);
\draw[edge=blue] (V6) -- node[edgeLabel] {$10$} (V7);
\draw[edge=red] (V6) -- node[edgeLabel] {$11$} (V21);
\draw[edge=green] (V7) -- node[edgeLabel] {$12$} (V8);
\draw[edge=red] (V7) -- node[edgeLabel] {$13$} (V56);
\draw[edge=red] (V8) -- node[edgeLabel] {$14$} (V63);
\draw[edge=blue] (V9) -- node[edgeLabel] {$15$} (V10);
\draw[edge=red] (V9) -- node[edgeLabel] {$16$} (V12);
\draw[edge=green] (V9) -- node[edgeLabel] {$17$} (V13);
\draw[edge=red] (V10) -- node[edgeLabel] {$18$} (V11);
\draw[edge=green] (V10) -- node[edgeLabel] {$19$} (V12);
\draw[edge=blue] (V11) -- node[edgeLabel] {$20$} (V13);
\draw[edge=green] (V11) -- node[edgeLabel] {$21$} (V14);
\draw[edge=blue] (V12) -- node[edgeLabel] {$22$} (V16);
\draw[edge=red] (V13) -- node[edgeLabel] {$23$} (V62);
\draw[edge=blue] (V14) -- node[edgeLabel] {$24$} (V15);
\draw[edge=red] (V14) -- node[edgeLabel] {$25$} (V53);
\draw[edge=green] (V15) -- node[edgeLabel] {$26$} (V16);
\draw[edge=red] (V15) -- node[edgeLabel] {$27$} (V32);
\draw[edge=red] (V16) -- node[edgeLabel] {$28$} (V23);
\draw[edge=blue] (V17) -- node[edgeLabel] {$29$} (V18);
\draw[edge=red] (V17) -- node[edgeLabel] {$30$} (V20);
\draw[edge=green] (V17) -- node[edgeLabel] {$31$} (V21);
\draw[edge=red] (V18) -- node[edgeLabel] {$32$} (V19);
\draw[edge=green] (V18) -- node[edgeLabel] {$33$} (V20);
\draw[edge=blue] (V19) -- node[edgeLabel] {$34$} (V21);
\draw[edge=green] (V19) -- node[edgeLabel] {$35$} (V22);
\draw[edge=blue] (V20) -- node[edgeLabel] {$36$} (V24);
\draw[edge=blue] (V22) -- node[edgeLabel] {$37$} (V23);
\draw[edge=red] (V22) -- node[edgeLabel] {$38$} (V45);
\draw[edge=green] (V23) -- node[edgeLabel] {$39$} (V24);
\draw[edge=red] (V24) -- node[edgeLabel] {$40$} (V39);
\draw[edge=blue] (V25) -- node[edgeLabel] {$41$} (V26);
\draw[edge=red] (V25) -- node[edgeLabel] {$42$} (V28);
\draw[edge=green] (V25) -- node[edgeLabel] {$43$} (V29);
\draw[edge=red] (V26) -- node[edgeLabel] {$44$} (V27);
\draw[edge=green] (V26) -- node[edgeLabel] {$45$} (V28);
\draw[edge=blue] (V27) -- node[edgeLabel] {$46$} (V29);
\draw[edge=green] (V27) -- node[edgeLabel] {$47$} (V30);
\draw[edge=blue] (V28) -- node[edgeLabel] {$48$} (V32);
\draw[edge=red] (V29) -- node[edgeLabel] {$49$} (V46);
\draw[edge=blue] (V30) -- node[edgeLabel] {$50$} (V31);
\draw[edge=green] (V31) -- node[edgeLabel] {$51$} (V32);
\draw[edge=red] (V31) -- node[edgeLabel] {$52$} (V40);
\draw[edge=blue] (V33) -- node[edgeLabel] {$53$} (V34);
\draw[edge=red] (V33) -- node[edgeLabel] {$54$} (V36);
\draw[edge=green] (V33) -- node[edgeLabel] {$55$} (V37);
\draw[edge=red] (V34) -- node[edgeLabel] {$56$} (V35);
\draw[edge=green] (V34) -- node[edgeLabel] {$57$} (V36);
\draw[edge=blue] (V35) -- node[edgeLabel] {$58$} (V37);
\draw[edge=green] (V35) -- node[edgeLabel] {$59$} (V38);
\draw[edge=blue] (V36) -- node[edgeLabel] {$60$} (V40);
\draw[edge=red] (V37) -- node[edgeLabel] {$61$} (V54);
\draw[edge=blue] (V38) -- node[edgeLabel] {$62$} (V39);
\draw[edge=red] (V38) -- node[edgeLabel] {$63$} (V61);
\draw[edge=green] (V39) -- node[edgeLabel] {$64$} (V40);
\draw[edge=blue] (V41) -- node[edgeLabel] {$65$} (V42);
\draw[edge=red] (V41) -- node[edgeLabel] {$66$} (V44);
\draw[edge=green] (V41) -- node[edgeLabel] {$67$} (V45);
\draw[edge=red] (V42) -- node[edgeLabel] {$68$} (V43);
\draw[edge=green] (V42) -- node[edgeLabel] {$69$} (V44);
\draw[edge=blue] (V43) -- node[edgeLabel] {$70$} (V45);
\draw[edge=green] (V43) -- node[edgeLabel] {$71$} (V46);
\draw[edge=blue] (V44) -- node[edgeLabel] {$72$} (V48);
\draw[edge=blue] (V46) -- node[edgeLabel] {$73$} (V47);
\draw[edge=green] (V47) -- node[edgeLabel] {$74$} (V48);
\draw[edge=red] (V47) -- node[edgeLabel] {$75$} (V64);
\draw[edge=red] (V48) -- node[edgeLabel] {$76$} (V55);
\draw[edge=blue] (V49) -- node[edgeLabel] {$77$} (V50);
\draw[edge=red] (V49) -- node[edgeLabel] {$78$} (V52);
\draw[edge=green] (V49) -- node[edgeLabel] {$79$} (V53);
\draw[edge=red] (V50) -- node[edgeLabel] {$80$} (V51);
\draw[edge=green] (V50) -- node[edgeLabel] {$81$} (V52);
\draw[edge=blue] (V51) -- node[edgeLabel] {$82$} (V53);
\draw[edge=green] (V51) -- node[edgeLabel] {$83$} (V54);
\draw[edge=blue] (V52) -- node[edgeLabel] {$84$} (V56);
\draw[edge=blue] (V54) -- node[edgeLabel] {$85$} (V55);
\draw[edge=green] (V55) -- node[edgeLabel] {$86$} (V56);
\draw[edge=blue] (V57) -- node[edgeLabel] {$87$} (V58);
\draw[edge=red] (V57) -- node[edgeLabel] {$88$} (V60);
\draw[edge=green] (V57) -- node[edgeLabel] {$89$} (V61);
\draw[edge=red] (V58) -- node[edgeLabel] {$90$} (V59);
\draw[edge=green] (V58) -- node[edgeLabel] {$91$} (V60);
\draw[edge=blue] (V59) -- node[edgeLabel] {$92$} (V61);
\draw[edge=green] (V59) -- node[edgeLabel] {$93$} (V62);
\draw[edge=blue] (V60) -- node[edgeLabel] {$94$} (V64);
\draw[edge=blue] (V62) -- node[edgeLabel] {$95$} (V63);
\draw[edge=green] (V63) -- node[edgeLabel] {$96$} (V64);
\vertexLabelR{V1}{left}{$ $}
\vertexLabelR{V2}{left}{$ $}
\vertexLabelR{V3}{left}{$ $}
\vertexLabelR{V4}{left}{$ $}
\vertexLabelR{V5}{left}{$ $}
\vertexLabelR{V6}{left}{$ $}
\vertexLabelR{V7}{left}{$ $}
\vertexLabelR{V8}{left}{$ $}
\vertexLabelR{V9}{left}{$ $}
\vertexLabelR{V10}{left}{$ $}
\vertexLabelR{V11}{left}{$ $}
\vertexLabelR{V12}{left}{$ $}
\vertexLabelR{V13}{left}{$ $}
\vertexLabelR{V14}{left}{$ $}
\vertexLabelR{V15}{left}{$ $}
\vertexLabelR{V16}{left}{$ $}
\vertexLabelR{V17}{left}{$ $}
\vertexLabelR{V18}{left}{$ $}
\vertexLabelR{V19}{left}{$ $}
\vertexLabelR{V20}{left}{$ $}
\vertexLabelR{V21}{left}{$ $}
\vertexLabelR{V22}{left}{$ $}
\vertexLabelR{V23}{left}{$ $}
\vertexLabelR{V24}{left}{$ $}
\vertexLabelR{V25}{left}{$ $}
\vertexLabelR{V26}{left}{$ $}
\vertexLabelR{V27}{left}{$ $}
\vertexLabelR{V28}{left}{$ $}
\vertexLabelR{V29}{left}{$ $}
\vertexLabelR{V30}{left}{$ $}
\vertexLabelR{V31}{left}{$ $}
\vertexLabelR{V32}{left}{$ $}
\vertexLabelR{V33}{left}{$ $}
\vertexLabelR{V34}{left}{$ $}
\vertexLabelR{V35}{left}{$ $}
\vertexLabelR{V36}{left}{$ $}
\vertexLabelR{V37}{left}{$ $}
\vertexLabelR{V38}{left}{$ $}
\vertexLabelR{V39}{left}{$ $}
\vertexLabelR{V40}{left}{$ $}
\vertexLabelR{V41}{left}{$ $}
\vertexLabelR{V42}{left}{$ $}
\vertexLabelR{V43}{left}{$ $}
\vertexLabelR{V44}{left}{$ $}
\vertexLabelR{V45}{left}{$ $}
\vertexLabelR{V46}{left}{$ $}
\vertexLabelR{V47}{left}{$ $}
\vertexLabelR{V48}{left}{$ $}
\vertexLabelR{V49}{left}{$ $}
\vertexLabelR{V50}{left}{$ $}
\vertexLabelR{V51}{left}{$ $}
\vertexLabelR{V52}{left}{$ $}
\vertexLabelR{V53}{left}{$ $}
\vertexLabelR{V54}{left}{$ $}
\vertexLabelR{V55}{left}{$ $}
\vertexLabelR{V56}{left}{$ $}
\vertexLabelR{V57}{left}{$ $}
\vertexLabelR{V58}{left}{$ $}
\vertexLabelR{V59}{left}{$ $}
\vertexLabelR{V60}{left}{$ $}
\vertexLabelR{V61}{left}{$ $}
\vertexLabelR{V62}{left}{$ $}
\vertexLabelR{V63}{left}{$ $}
\vertexLabelR{V64}{left}{$ $}
\end{tikzpicture}

%% file: Q8_red.tex
\begin{tikzpicture}[vertexBall, edgeDouble=nolabels, faceStyle=nolabels, scale=2]

\coordinate (V3) at (4.9 , -0.975);
\coordinate (V5) at (4.615 , -1.91);
\coordinate (V6) at (3.85099478079332 , 0.3466597077244259);
\coordinate (V7) at (3.187202505219207 , 0.9244164926931107);
\coordinate (V8) at (4.155 , 2.775);
\coordinate (V11) at (0.975 , 4.9);
\coordinate (V13) at (1.91 , 4.615);
\coordinate (V14) at (-0.3466597077244259 , 3.85099478079332);
\coordinate (V15) at (-0.9244164926931107 , 3.187202505219207);
\coordinate (V16) at (-2.775 , 4.155);
\coordinate (V19) at (-3.475550104384133 , 3.110120041753654);
\coordinate (V21) at (-0.4542181628392484 , 1.870562630480167);
\coordinate (V22) at (-4.155 , 2.775);
\coordinate (V23) at (-3.535 , 3.535);
\coordinate (V24) at (-2.568410229645094 , 1.768387265135699);
\coordinate (V27) at (3.475550104384133 , -3.110120041753654);
\coordinate (V29) at (0.4542181628392484 , -1.870562630480167);
\coordinate (V30) at (4.155 , -2.775);
\coordinate (V31) at (3.535 , -3.535);
\coordinate (V32) at (2.568410229645094 , -1.768387265135699);
\coordinate (V35) at (-0.975 , -4.9);
\coordinate (V37) at (-1.91 , -4.615);
\coordinate (V38) at (0.3466597077244259 , -3.85099478079332);
\coordinate (V39) at (0.9244164926931107 , -3.187202505219207);
\coordinate (V40) at (2.775 , -4.155);
\coordinate (V43) at (-4.9 , 0.975);
\coordinate (V45) at (-4.615 , 1.91);
\coordinate (V46) at (-3.85099478079332 , -0.3466597077244259);
\coordinate (V47) at (-3.187202505219207 , -0.9244164926931107);
\coordinate (V48) at (-4.155 , -2.775);
\coordinate (V51) at (-3.110120041753654 , -3.475550104384133);
\coordinate (V53) at (-1.870562630480167 , -0.4542181628392484);
\coordinate (V54) at (-2.775 , -4.155);
\coordinate (V55) at (-3.535 , -3.535);
\coordinate (V56) at (-1.768387265135699 , -2.568410229645094);
\coordinate (V59) at (3.110120041753654 , 3.475550104384133);
\coordinate (V61) at (1.870562630480167 , 0.4542181628392484);
\coordinate (V62) at (2.775 , 4.155);
\coordinate (V63) at (3.535 , 3.535);
\coordinate (V64) at (1.768387265135699 , 2.568410229645094);
\coordinate (V65) at (4.838333333333333 , 0.9616666666666667);
\coordinate (V66) at (-0.9616666666666667 , 4.838333333333333);
\coordinate (V67) at (-2.166059498956159 , 2.249689979123174);
\coordinate (V68) at (2.166059498956159 , -2.249689979123174);
\coordinate (V69) at (0.9616666666666667 , -4.838333333333333);
\coordinate (V70) at (-4.838333333333333 , -0.9616666666666667);
\coordinate (V71) at (-2.249689979123174 , -2.166059498956159);
\coordinate (V72) at (2.249689979123174 , 2.166059498956159);
\draw[edge=green] (V5) -- node[edgeLabel] {$3$} (V65);
\draw[edge=red] (V3) -- node[edgeLabel] {$4$} (V65);
\draw[edge=blue] (V3) -- node[edgeLabel] {$6$} (V5);
\draw[edge=green] (V3) -- node[edgeLabel] {$7$} (V6);
\draw[edge=blue] (V8) -- node[edgeLabel] {$8$} (V65);
\draw[edge=red] (V5) -- node[edgeLabel] {$9$} (V30);
\draw[edge=blue] (V6) -- node[edgeLabel] {$10$} (V7);
\draw[edge=red] (V6) -- node[edgeLabel] {$11$} (V21);
\draw[edge=green] (V7) -- node[edgeLabel] {$12$} (V8);
\draw[edge=red] (V7) -- node[edgeLabel] {$13$} (V56);
\draw[edge=red] (V8) -- node[edgeLabel] {$14$} (V63);
\draw[edge=green] (V13) -- node[edgeLabel] {$17$} (V66);
\draw[edge=red] (V11) -- node[edgeLabel] {$18$} (V66);
\draw[edge=blue] (V11) -- node[edgeLabel] {$20$} (V13);
\draw[edge=green] (V11) -- node[edgeLabel] {$21$} (V14);
\draw[edge=blue] (V16) -- node[edgeLabel] {$22$} (V66);
\draw[edge=red] (V13) -- node[edgeLabel] {$23$} (V62);
\draw[edge=blue] (V14) -- node[edgeLabel] {$24$} (V15);
\draw[edge=red] (V14) -- node[edgeLabel] {$25$} (V53);
\draw[edge=green] (V15) -- node[edgeLabel] {$26$} (V16);
\draw[edge=red] (V15) -- node[edgeLabel] {$27$} (V32);
\draw[edge=red] (V16) -- node[edgeLabel] {$28$} (V23);
\draw[edge=green] (V21) -- node[edgeLabel] {$31$} (V67);
\draw[edge=red] (V19) -- node[edgeLabel] {$32$} (V67);
\draw[edge=blue] (V19) -- node[edgeLabel] {$34$} (V21);
\draw[edge=green] (V19) -- node[edgeLabel] {$35$} (V22);
\draw[edge=blue] (V24) -- node[edgeLabel] {$36$} (V67);
\draw[edge=blue] (V22) -- node[edgeLabel] {$37$} (V23);
\draw[edge=red] (V22) -- node[edgeLabel] {$38$} (V45);
\draw[edge=green] (V23) -- node[edgeLabel] {$39$} (V24);
\draw[edge=red] (V24) -- node[edgeLabel] {$40$} (V39);
\draw[edge=green] (V29) -- node[edgeLabel] {$43$} (V68);
\draw[edge=red] (V27) -- node[edgeLabel] {$44$} (V68);
\draw[edge=blue] (V27) -- node[edgeLabel] {$46$} (V29);
\draw[edge=green] (V27) -- node[edgeLabel] {$47$} (V30);
\draw[edge=blue] (V32) -- node[edgeLabel] {$48$} (V68);
\draw[edge=red] (V29) -- node[edgeLabel] {$49$} (V46);
\draw[edge=blue] (V30) -- node[edgeLabel] {$50$} (V31);
\draw[edge=green] (V31) -- node[edgeLabel] {$51$} (V32);
\draw[edge=red] (V31) -- node[edgeLabel] {$52$} (V40);
\draw[edge=green] (V37) -- node[edgeLabel] {$55$} (V69);
\draw[edge=red] (V35) -- node[edgeLabel] {$56$} (V69);
\draw[edge=blue] (V35) -- node[edgeLabel] {$58$} (V37);
\draw[edge=green] (V35) -- node[edgeLabel] {$59$} (V38);
\draw[edge=blue] (V40) -- node[edgeLabel] {$60$} (V69);
\draw[edge=red] (V37) -- node[edgeLabel] {$61$} (V54);
\draw[edge=blue] (V38) -- node[edgeLabel] {$62$} (V39);
\draw[edge=red] (V38) -- node[edgeLabel] {$63$} (V61);
\draw[edge=green] (V39) -- node[edgeLabel] {$64$} (V40);
\draw[edge=green] (V45) -- node[edgeLabel] {$67$} (V70);
\draw[edge=red] (V43) -- node[edgeLabel] {$68$} (V70);
\draw[edge=blue] (V43) -- node[edgeLabel] {$70$} (V45);
\draw[edge=green] (V43) -- node[edgeLabel] {$71$} (V46);
\draw[edge=blue] (V48) -- node[edgeLabel] {$72$} (V70);
\draw[edge=blue] (V46) -- node[edgeLabel] {$73$} (V47);
\draw[edge=green] (V47) -- node[edgeLabel] {$74$} (V48);
\draw[edge=red] (V47) -- node[edgeLabel] {$75$} (V64);
\draw[edge=red] (V48) -- node[edgeLabel] {$76$} (V55);
\draw[edge=green] (V53) -- node[edgeLabel] {$79$} (V71);
\draw[edge=red] (V51) -- node[edgeLabel] {$80$} (V71);
\draw[edge=blue] (V51) -- node[edgeLabel] {$82$} (V53);
\draw[edge=green] (V51) -- node[edgeLabel] {$83$} (V54);
\draw[edge=blue] (V56) -- node[edgeLabel] {$84$} (V71);
\draw[edge=blue] (V54) -- node[edgeLabel] {$85$} (V55);
\draw[edge=green] (V55) -- node[edgeLabel] {$86$} (V56);
\draw[edge=green] (V61) -- node[edgeLabel] {$89$} (V72);
\draw[edge=red] (V59) -- node[edgeLabel] {$90$} (V72);
\draw[edge=blue] (V59) -- node[edgeLabel] {$92$} (V61);
\draw[edge=green] (V59) -- node[edgeLabel] {$93$} (V62);
\draw[edge=blue] (V64) -- node[edgeLabel] {$94$} (V72);
\draw[edge=blue] (V62) -- node[edgeLabel] {$95$} (V63);
\draw[edge=green] (V63) -- node[edgeLabel] {$96$} (V64);
\vertexLabelR[]{V3}{left}{$$}
\vertexLabelR[]{V5}{left}{$$}
\vertexLabelR[]{V6}{left}{$$}
\vertexLabelR[]{V7}{left}{$$}
\vertexLabelR[]{V8}{left}{$$}
\vertexLabelR[]{V11}{left}{$$}
\vertexLabelR[]{V13}{left}{$$}
\vertexLabelR[]{V14}{left}{$$}
\vertexLabelR[]{V15}{left}{$$}
\vertexLabelR[]{V16}{left}{$$}
\vertexLabelR[]{V19}{left}{$$}
\vertexLabelR[]{V21}{left}{$$}
\vertexLabelR[]{V22}{left}{$$}
\vertexLabelR[]{V23}{left}{$$}
\vertexLabelR[]{V24}{left}{$$}
\vertexLabelR[]{V27}{left}{$$}
\vertexLabelR[]{V29}{left}{$$}
\vertexLabelR[]{V30}{left}{$$}
\vertexLabelR[]{V31}{left}{$$}
\vertexLabelR[]{V32}{left}{$$}
\vertexLabelR[]{V35}{left}{$$}
\vertexLabelR[]{V37}{left}{$$}
\vertexLabelR[]{V38}{left}{$$}
\vertexLabelR[]{V39}{left}{$$}
\vertexLabelR[]{V40}{left}{$$}
\vertexLabelR[]{V43}{left}{$$}
\vertexLabelR[]{V45}{left}{$$}
\vertexLabelR[]{V46}{left}{$$}
\vertexLabelR[]{V47}{left}{$$}
\vertexLabelR[]{V48}{left}{$$}
\vertexLabelR[]{V51}{left}{$$}
\vertexLabelR[]{V53}{left}{$$}
\vertexLabelR[]{V54}{left}{$$}
\vertexLabelR[]{V55}{left}{$$}
\vertexLabelR[]{V56}{left}{$$}
\vertexLabelR[]{V59}{left}{$$}
\vertexLabelR[]{V61}{left}{$$}
\vertexLabelR[]{V62}{left}{$$}
\vertexLabelR[]{V63}{left}{$$}
\vertexLabelR[]{V64}{left}{$$}
\vertexLabelR[]{V65}{left}{$$}
\vertexLabelR[]{V66}{left}{$$}
\vertexLabelR[]{V67}{left}{$$}
\vertexLabelR[]{V68}{left}{$$}
\vertexLabelR[]{V69}{left}{$$}
\vertexLabelR[]{V70}{left}{$$}
\vertexLabelR[]{V71}{left}{$$}
\vertexLabelR[]{V72}{left}{$$}
\end{tikzpicture}

%% file: FacegraphC4.tex
\begin{tikzpicture}[vertexBall, edgeDouble=nolabels, faceStyle=nolabels, scale=1.7]

\coordinate (V1) at (0.7071067811865475 , 0.7071067811865475);
\coordinate (V2) at (1.060660171779821 , 0.3535533905932737);
\coordinate (V3) at (2.121320343559642 , 2.121320343559642);
\coordinate (V4) at (2.82842712474619 , 2.82842712474619);
\coordinate (V5) at (0.353553390593274 , 1.060660171779821);
\coordinate (V6) at (1.414213562373095 , 1.414213562373095);
\coordinate (V7) at (-0.7071067811865475 , 0.7071067811865475);
\coordinate (V8) at (-0.3535533905932737 , 1.060660171779821);
\coordinate (V9) at (-2.121320343559642 , 2.121320343559642);
\coordinate (V10) at (-2.82842712474619 , 2.82842712474619);
\coordinate (V11) at (-1.060660171779821 , 0.353553390593274);
\coordinate (V12) at (-1.414213562373095 , 1.414213562373095);
\coordinate (V13) at (-0.7071067811865476 , -0.7071067811865475);
\coordinate (V14) at (-1.060660171779821 , -0.3535533905932736);
\coordinate (V15) at (-2.121320343559643 , -2.121320343559642);
\coordinate (V16) at (-2.82842712474619 , -2.828427124746189);
\coordinate (V17) at (-0.3535533905932741 , -1.060660171779821);
\coordinate (V18) at (-1.414213562373095 , -1.414213562373095);
\coordinate (V19) at (0.7071067811865475 , -0.7071067811865476);
\coordinate (V20) at (0.3535533905932736 , -1.060660171779821);
\coordinate (V21) at (2.121320343559642 , -2.121320343559643);
\coordinate (V22) at (2.828427124746189 , -2.82842712474619);
\coordinate (V23) at (1.060660171779821 , -0.3535533905932741);
\coordinate (V24) at (1.414213562373095 , -1.414213562373095);
\draw[edge] (V3) -- node[edgeLabel] {$1$} (V4);
\draw[edge] (V21) -- node[edgeLabel] {$2$} (V22);
\draw[edge] (V3) -- node[edgeLabel] {$3$} (V6);
\draw[edge] (V21) -- node[edgeLabel] {$4$} (V23);
\draw[edge] (V2) -- node[edgeLabel] {$5$} (V6);
\draw[edge] (V4) -- node[edgeLabel] {$6$} (V22);
\draw[edge] (V2) -- node[edgeLabel] {$7$} (V23);
\draw[edge] (V9) -- node[edgeLabel] {$8$} (V10);
\draw[edge] (V3) -- node[edgeLabel] {$9$} (V5);
\draw[edge] (V9) -- node[edgeLabel] {$10$} (V12);
\draw[edge] (V8) -- node[edgeLabel] {$11$} (V12);
\draw[edge] (V4) -- node[edgeLabel] {$12$} (V10);
\draw[edge] (V5) -- node[edgeLabel] {$13$} (V8);
\draw[edge] (V15) -- node[edgeLabel] {$14$} (V16);
\draw[edge] (V9) -- node[edgeLabel] {$15$} (V11);
\draw[edge] (V15) -- node[edgeLabel] {$16$} (V18);
\draw[edge] (V14) -- node[edgeLabel] {$17$} (V18);
\draw[edge] (V10) -- node[edgeLabel] {$18$} (V16);
\draw[edge] (V11) -- node[edgeLabel] {$19$} (V14);
\draw[edge] (V15) -- node[edgeLabel] {$20$} (V17);
\draw[edge] (V21) -- node[edgeLabel] {$21$} (V24);
\draw[edge] (V20) -- node[edgeLabel] {$22$} (V24);
\draw[edge] (V16) -- node[edgeLabel] {$23$} (V22);
\draw[edge] (V17) -- node[edgeLabel] {$24$} (V20);
\draw[edge] (V1) -- node[edgeLabel] {$25$} (V6);
\draw[edge] (V1) -- node[edgeLabel] {$26$} (V5);
\draw[edge] (V7) -- node[edgeLabel] {$27$} (V12);
\draw[edge] (V7) -- node[edgeLabel] {$28$} (V11);
\draw[edge] (V13) -- node[edgeLabel] {$29$} (V18);
\draw[edge] (V13) -- node[edgeLabel] {$30$} (V17);
\draw[edge] (V19) -- node[edgeLabel] {$31$} (V24);
\draw[edge] (V19) -- node[edgeLabel] {$32$} (V23);
\draw[edge] (V1) -- node[edgeLabel] {$33$} (V2);
\draw[edge] (V7) -- node[edgeLabel] {$34$} (V8);
\draw[edge] (V13) -- node[edgeLabel] {$35$} (V14);
\draw[edge] (V19) -- node[edgeLabel] {$36$} (V20);
\vertexLabelR{V1}{left}{$ $}
\vertexLabelR{V2}{left}{$ $}
\vertexLabelR{V3}{left}{$ $}
\vertexLabelR{V4}{left}{$ $}
\vertexLabelR{V5}{left}{$ $}
\vertexLabelR{V6}{left}{$ $}
\vertexLabelR{V7}{left}{$ $}
\vertexLabelR{V8}{left}{$ $}
\vertexLabelR{V9}{left}{$ $}
\vertexLabelR{V10}{left}{$ $}
\vertexLabelR{V11}{left}{$ $}
\vertexLabelR{V12}{left}{$ $}
\vertexLabelR{V13}{left}{$ $}
\vertexLabelR{V14}{left}{$ $}
\vertexLabelR{V15}{left}{$ $}
\vertexLabelR{V16}{left}{$ $}
\vertexLabelR{V17}{left}{$ $}
\vertexLabelR{V18}{left}{$ $}
\vertexLabelR{V19}{left}{$ $}
\vertexLabelR{V20}{left}{$ $}
\vertexLabelR{V21}{left}{$ $}
\vertexLabelR{V22}{left}{$ $}
\vertexLabelR{V23}{left}{$ $}
\vertexLabelR{V24}{left}{$ $}
\node at (0.7071067811865475-0.1 , 0.7071067811865475-0.1) {$a_1$};
\node at (1.060660171779821-0.2 , 0.3535533905932737) {$b_1$};
\node at (2.121320343559642+0.2 , 2.121320343559642) {$c_1$};
\node at (2.82842712474619+0.1 , 2.82842712474619+0.2) {$d_1$};
\node at (0.353553390593274 , 1.060660171779821-0.2) {$e_1$};
\node at (1.414213562373095 +0.2, 1.414213562373095) {$f_1$};
\node at (-0.7071067811865475+0.1 , 0.7071067811865475-0.1) {$a_2$};
\node at (-0.3535533905932737 , 1.060660171779821-0.2) {$b_2$};
\node at (-2.121320343559642+0.2 , 2.121320343559642) {$c_2$};
\node at (-2.82842712474619 -0.1, 2.82842712474619+0.2) {$d_2$};
\node at (-1.060660171779821 +0.2, 0.353553390593274) {$e_2$};
\node at (-1.414213562373095+0.2 , 1.414213562373095+0.1) {$f_2$};
\node at (-0.7071067811865476+0.1 , -0.7071067811865475+0.1) {$a_3$};
\node at (-1.060660171779821+0.2 , -0.3535533905932736) {$b_3$};
\node at (-2.121320343559643-0.2 , -2.121320343559642) {$c_3$};
\node at (-2.82842712474619 -0.1, -2.828427124746189-0.2) {$d_3$};
\node at (-0.3535533905932741 , -1.060660171779821+0.2) {$e_3$};
\node at (-1.414213562373095-0.2 , -1.414213562373095) {$f_3$};
\node at (0.7071067811865475-0.1 , -0.7071067811865476+0.1) {$a_4$};
\node at (0.3535533905932736 , -1.060660171779821+0.2) {$b_4$};
\node at (2.121320343559642 -0.2, -2.121320343559643) {$c_4$};
\node at (2.828427124746189 +0.1, -2.82842712474619-0.2) {$d_4$};
\node at (1.060660171779821-0.2 , -0.3535533905932741) {$e_4$};
\node at (1.414213562373095 -0.2, -1.414213562373095-0.2) {$f_4$};
\end{tikzpicture}

%% file: FacegraphC7.tex
\begin{tikzpicture}[vertexBall, edgeDouble=nolabels, faceStyle=nolabels, scale=2.5]

\coordinate (V1) at (0.704405825649691 , 0.3392239669730524);
\coordinate (V2) at (0.7860866519424037 , 0.1696119834865262);
\coordinate (V3) at (1.382853759595796 , 0.6659472721439031);
\coordinate (V4) at (1.722077726568848 , 0.8293089247293285);
\coordinate (V5) at (0.6227249993569783 , 0.5088359504595785);
\coordinate (V6) at (1.043629792622743 , 0.5025856195584777);
\coordinate (V7) at (0.1739738716752359 , 0.7622293348805765);
\coordinate (V8) at (0.3575090223697507 , 0.7203386344133406);
\coordinate (V9) at (0.3415366735441793 , 1.496369937658636);
\coordinate (V10) at (0.4253180744786508 , 1.863440239047665);
\coordinate (V11) at (-0.009561279019278857 , 0.8041200353478124);
\coordinate (V12) at (0.2577552726097076 , 1.129299636269606);
\coordinate (V13) at (-0.4874639560909119 , 0.6112604669781575);
\coordinate (V14) at (-0.3402801929023528 , 0.7286356013966028);
\coordinate (V15) at (-0.9569644937646932 , 1.199995519732393);
\coordinate (V16) at (-1.191714762601584 , 1.494363046109512);
\coordinate (V17) at (-0.6346477192794708 , 0.4938853325597122);
\coordinate (V18) at (-0.7222142249278024 , 0.9056279933552753);
\coordinate (V19) at (-0.78183148246803 , 1.776356839400251e-16);
\coordinate (V20) at (-0.7818314824680299 , 0.1882550990706334);
\coordinate (V21) at (-1.534851878750563 , 3.552713678800501e-16);
\coordinate (V22) at (-1.91136207689183 , 4.440892098500626e-16);
\coordinate (V23) at (-0.78183148246803 , -0.188255099070633);
\coordinate (V24) at (-1.158341680609296 , 2.664535259100376e-16);
\coordinate (V25) at (-0.4874639560909121 , -0.6112604669781574);
\coordinate (V26) at (-0.6346477192794712 , -0.4938853325597118);
\coordinate (V27) at (-0.9569644937646942 , -1.199995519732393);
\coordinate (V28) at (-1.191714762601585 , -1.494363046109511);
\coordinate (V29) at (-0.3402801929023533 , -0.7286356013966028);
\coordinate (V30) at (-0.7222142249278031 , -0.9056279933552752);
\coordinate (V31) at (0.1739738716752356 , -0.7622293348805766);
\coordinate (V32) at (-0.009561279019279235 , -0.8041200353478124);
\coordinate (V33) at (0.3415366735441785 , -1.496369937658636);
\coordinate (V34) at (0.4253180744786502 , -1.863440239047666);
\coordinate (V35) at (0.3575090223697505 , -0.720338634413341);
\coordinate (V36) at (0.2577552726097071 , -1.129299636269606);
\coordinate (V37) at (0.7044058256496911 , -0.3392239669730528);
\coordinate (V38) at (0.6227249993569782 , -0.5088359504595789);
\coordinate (V39) at (1.382853759595796 , -0.665947272143904);
\coordinate (V40) at (1.722077726568848 , -0.8293089247293293);
\coordinate (V41) at (0.7860866519424038 , -0.1696119834865267);
\coordinate (V42) at (1.043629792622743 , -0.5025856195584784);
\draw[edge] (V3) -- node[edgeLabel] {$1$} (V4);
\draw[edge] (V39) -- node[edgeLabel] {$2$} (V40);
\draw[edge] (V3) -- node[edgeLabel] {$3$} (V6);
\draw[edge] (V39) -- node[edgeLabel] {$4$} (V41);
\draw[edge] (V2) -- node[edgeLabel] {$5$} (V6);
\draw[edge] (V4) -- node[edgeLabel] {$6$} (V40);
\draw[edge] (V2) -- node[edgeLabel] {$7$} (V41);
\draw[edge] (V9) -- node[edgeLabel] {$8$} (V10);
\draw[edge] (V3) -- node[edgeLabel] {$9$} (V5);
\draw[edge] (V9) -- node[edgeLabel] {$10$} (V12);
\draw[edge] (V8) -- node[edgeLabel] {$11$} (V12);
\draw[edge] (V4) -- node[edgeLabel] {$12$} (V10);
\draw[edge] (V5) -- node[edgeLabel] {$13$} (V8);
\draw[edge] (V15) -- node[edgeLabel] {$14$} (V16);
\draw[edge] (V9) -- node[edgeLabel] {$15$} (V11);
\draw[edge] (V15) -- node[edgeLabel] {$16$} (V18);
\draw[edge] (V14) -- node[edgeLabel] {$17$} (V18);
\draw[edge] (V10) -- node[edgeLabel] {$18$} (V16);
\draw[edge] (V11) -- node[edgeLabel] {$19$} (V14);
\draw[edge] (V21) -- node[edgeLabel] {$20$} (V22);
\draw[edge] (V15) -- node[edgeLabel] {$21$} (V17);
\draw[edge] (V21) -- node[edgeLabel] {$22$} (V24);
\draw[edge] (V20) -- node[edgeLabel] {$23$} (V24);
\draw[edge] (V16) -- node[edgeLabel] {$24$} (V22);
\draw[edge] (V17) -- node[edgeLabel] {$25$} (V20);
\draw[edge] (V27) -- node[edgeLabel] {$26$} (V28);
\draw[edge] (V21) -- node[edgeLabel] {$27$} (V23);
\draw[edge] (V27) -- node[edgeLabel] {$28$} (V30);
\draw[edge] (V26) -- node[edgeLabel] {$29$} (V30);
\draw[edge] (V22) -- node[edgeLabel] {$30$} (V28);
\draw[edge] (V23) -- node[edgeLabel] {$31$} (V26);
\draw[edge] (V33) -- node[edgeLabel] {$32$} (V34);
\draw[edge] (V27) -- node[edgeLabel] {$33$} (V29);
\draw[edge] (V33) -- node[edgeLabel] {$34$} (V36);
\draw[edge] (V32) -- node[edgeLabel] {$35$} (V36);
\draw[edge] (V28) -- node[edgeLabel] {$36$} (V34);
\draw[edge] (V29) -- node[edgeLabel] {$37$} (V32);
\draw[edge] (V33) -- node[edgeLabel] {$38$} (V35);
\draw[edge] (V39) -- node[edgeLabel] {$39$} (V42);
\draw[edge] (V38) -- node[edgeLabel] {$40$} (V42);
\draw[edge] (V34) -- node[edgeLabel] {$41$} (V40);
\draw[edge] (V35) -- node[edgeLabel] {$42$} (V38);
\draw[edge] (V1) -- node[edgeLabel] {$43$} (V6);
\draw[edge] (V1) -- node[edgeLabel] {$44$} (V5);
\draw[edge] (V7) -- node[edgeLabel] {$45$} (V12);
\draw[edge] (V7) -- node[edgeLabel] {$46$} (V11);
\draw[edge] (V13) -- node[edgeLabel] {$47$} (V18);
\draw[edge] (V13) -- node[edgeLabel] {$48$} (V17);
\draw[edge] (V19) -- node[edgeLabel] {$49$} (V24);
\draw[edge] (V19) -- node[edgeLabel] {$50$} (V23);
\draw[edge] (V25) -- node[edgeLabel] {$51$} (V30);
\draw[edge] (V25) -- node[edgeLabel] {$52$} (V29);
\draw[edge] (V31) -- node[edgeLabel] {$53$} (V36);
\draw[edge] (V31) -- node[edgeLabel] {$54$} (V35);
\draw[edge] (V37) -- node[edgeLabel] {$55$} (V42);
\draw[edge] (V37) -- node[edgeLabel] {$56$} (V41);
\draw[edge] (V1) -- node[edgeLabel] {$57$} (V2);
\draw[edge] (V7) -- node[edgeLabel] {$58$} (V8);
\draw[edge] (V13) -- node[edgeLabel] {$59$} (V14);
\draw[edge] (V19) -- node[edgeLabel] {$60$} (V20);
\draw[edge] (V25) -- node[edgeLabel] {$61$} (V26);
\draw[edge] (V31) -- node[edgeLabel] {$62$} (V32);
\draw[edge] (V37) -- node[edgeLabel] {$63$} (V38);
\vertexLabelR{V1}{left}{$ $}
\vertexLabelR{V2}{left}{$ $}
\vertexLabelR{V3}{left}{$ $}
\vertexLabelR{V4}{left}{$ $}
\vertexLabelR{V5}{left}{$ $}
\vertexLabelR{V6}{left}{$ $}
\vertexLabelR{V7}{left}{$ $}
\vertexLabelR{V8}{left}{$ $}
\vertexLabelR{V9}{left}{$ $}
\vertexLabelR{V10}{left}{$ $}
\vertexLabelR{V11}{left}{$ $}
\vertexLabelR{V12}{left}{$ $}
\vertexLabelR{V13}{left}{$ $}
\vertexLabelR{V14}{left}{$ $}
\vertexLabelR{V15}{left}{$ $}
\vertexLabelR{V16}{left}{$ $}
\vertexLabelR{V17}{left}{$ $}
\vertexLabelR{V18}{left}{$ $}
\vertexLabelR{V19}{left}{$ $}
\vertexLabelR{V20}{left}{$ $}
\vertexLabelR{V21}{left}{$ $}
\vertexLabelR{V22}{left}{$ $}
\vertexLabelR{V23}{left}{$ $}
\vertexLabelR{V24}{left}{$ $}
\vertexLabelR{V25}{left}{$ $}
\vertexLabelR{V26}{left}{$ $}
\vertexLabelR{V27}{left}{$ $}
\vertexLabelR{V28}{left}{$ $}
\vertexLabelR{V29}{left}{$ $}
\vertexLabelR{V30}{left}{$ $}
\vertexLabelR{V31}{left}{$ $}
\vertexLabelR{V32}{left}{$ $}
\vertexLabelR{V33}{left}{$ $}
\vertexLabelR{V34}{left}{$ $}
\vertexLabelR{V35}{left}{$ $}
\vertexLabelR{V36}{left}{$ $}
\vertexLabelR{V37}{left}{$ $}
\vertexLabelR{V38}{left}{$ $}
\vertexLabelR{V39}{left}{$ $}
\vertexLabelR{V40}{left}{$ $}
\vertexLabelR{V41}{left}{$ $}
\vertexLabelR{V42}{left}{$ $}

\node at (0.704405825649691-0.1 , 0.3392239669730524) {$a_1$};
\node at (0.7860866519424037-0.1 , 0.1696119834865262) {$b_1$};
\node at (1.382853759595796 , 0.6659472721439031+0.1) {$c_1$};
\node at (1.722077726568848 +0.1 , 0.8293089247293285 +0.1) {$d_1$};
\node at (0.6227249993569783-0.1 , 0.5088359504595785) {$e_1$};
\node at (1.043629792622743 , 0.5025856195584777-0.1){$f_1$};
\node at (0.1739738716752359 , 0.7622293348805765-0.1) {$a_2$};
\node at (0.3575090223697507 , 0.7203386344133406-0.1) {$b_2$};
\node at (0.3415366735441793+0.1 , 1.496369937658636) {$c_2$};
\node at (0.4253180744786508 , 1.863440239047665+0.1) {$d_2$};
\node at (-0.009561279019278857 , 0.8041200353478124-0.1) {$e_2$};
\node at (0.2577552726097076 +0.1, 1.129299636269606) {$f_2$};
\node at (-0.4874639560909119 , 0.6112604669781575-0.1) {$a_3$};
\node at (-0.3402801929023528 , 0.7286356013966028-0.1) {$b_3$};
\node at (-0.9569644937646932+0.1 , 1.199995519732393) {$c_3$};
\node at (-1.191714762601584 , 1.494363046109512+0.1) {$d_3$};
\node at (-0.6346477192794708 , 0.4938853325597122-0.1) {$e_3$};
\node at (-0.7222142249278024+0.1 , 0.9056279933552753) {$f_3$};
\node at (-0.78183148246803 +0.1, 1.776356839400251e-16) {$a_4$};
\node at (-0.7818314824680299+0.1 , 0.1882550990706334) {$b_4$};
\node at (-1.534851878750563 , 3.552713678800501e-16+0.1) {$c_4$};
\node at (-1.91136207689183-0.1 , 4.440892098500626e-16) {$d_4$};
\node at (-0.78183148246803 +0.1, -0.188255099070633) {$e_4$};
\node at (-1.158341680609296 , 2.664535259100376e-16+0.1) {$f_4$};
\node at (-0.4874639560909121 , -0.6112604669781574+0.1) {$a_5$};
\node at (-0.6346477192794712 , -0.4938853325597118+0.1) {$b_5$};
\node at (-0.9569644937646942 -0.1, -1.199995519732393) {$c_5$};
\node at (-1.191714762601585 , -1.494363046109511-0.1) {$d_5$};
\node at (-0.3402801929023533 , -0.7286356013966028+0.1) {$e_5$};
\node at (-0.7222142249278031-0.1 , -0.9056279933552752) {$f_5$};
\node at (0.1739738716752356 , -0.7622293348805766+0.1) {$a_6$};
\node at (-0.009561279019279235 , -0.8041200353478124+0.1) {$b_6$};
\node at (0.3415366735441785 , -1.496369937658636-0.1) {$c_6$};
\node at (0.4253180744786502 , -1.863440239047666-0.1) {$d_6$};
\node at (0.3575090223697505 , -0.720338634413341+0.1) {$e_6$};
\node at (0.2577552726097071-0.1 , -1.129299636269606) {$f_6$};
\node at (0.7044058256496911-0.1 , -0.3392239669730528) {$a_7$};
\node at (0.6227249993569782-0.1 , -0.5088359504595789) {$b_7$};
\node at (1.382853759595796 , -0.665947272143904-0.1) {$c_7$};
\node at (1.722077726568848+0.1 , -0.8293089247293293) {$d_7$};
\node at (0.7860866519424038-0.1 , -0.1696119834865267) {$e_7$};
\node at (1.043629792622743 , -0.5025856195584784-0.1) {$f_7$};
\end{tikzpicture}

%% file: FacegraphD4.tex
\begin{tikzpicture}[vertexBall, edgeDouble=nolabels, faceStyle=nolabels, scale=2]

\coordinate (V1) at (0.353553390593274 , 1.060660171779821);
\coordinate (V2) at (1.060660171779821 , 0.3535533905932737);
\coordinate (V3) at (1.414213562373095 , 1.414213562373095);
\coordinate (V4) at (2.121320343559642 , 2.121320343559642);
\coordinate (V5) at (-1.060660171779821 , 0.353553390593274);
\coordinate (V6) at (-0.3535533905932737 , 1.060660171779821);
\coordinate (V7) at (-1.414213562373095 , 1.414213562373095);
\coordinate (V8) at (-2.121320343559642 , 2.121320343559642);
\coordinate (V9) at (-0.3535533905932741 , -1.060660171779821);
\coordinate (V10) at (-1.060660171779821 , -0.3535533905932736);
\coordinate (V11) at (-1.414213562373095 , -1.414213562373095);
\coordinate (V12) at (-2.121320343559643 , -2.121320343559642);
\coordinate (V13) at (1.060660171779821 , -0.3535533905932741);
\coordinate (V14) at (0.3535533905932736 , -1.060660171779821);
\coordinate (V15) at (1.414213562373095 , -1.414213562373095);
\coordinate (V16) at (2.121320343559642 , -2.121320343559643);
\draw[edge] (V3) -- node[edgeLabel] {$1$} (V4);
\draw[edge] (V15) -- node[edgeLabel] {$2$} (V16);
\draw[edge] (V2) -- node[edgeLabel] {$3$} (V3);
\draw[edge] (V13) -- node[edgeLabel] {$4$} (V15);
\draw[edge] (V4) -- node[edgeLabel] {$5$} (V16);
\draw[edge] (V2) -- node[edgeLabel] {$6$} (V13);
\draw[edge] (V7) -- node[edgeLabel] {$7$} (V8);
\draw[edge] (V1) -- node[edgeLabel] {$8$} (V3);
\draw[edge] (V6) -- node[edgeLabel] {$9$} (V7);
\draw[edge] (V4) -- node[edgeLabel] {$10$} (V8);
\draw[edge] (V1) -- node[edgeLabel] {$11$} (V6);
\draw[edge] (V11) -- node[edgeLabel] {$12$} (V12);
\draw[edge] (V5) -- node[edgeLabel] {$13$} (V7);
\draw[edge] (V10) -- node[edgeLabel] {$14$} (V11);
\draw[edge] (V8) -- node[edgeLabel] {$15$} (V12);
\draw[edge] (V5) -- node[edgeLabel] {$16$} (V10);
\draw[edge] (V9) -- node[edgeLabel] {$17$} (V11);
\draw[edge] (V14) -- node[edgeLabel] {$18$} (V15);
\draw[edge] (V12) -- node[edgeLabel] {$19$} (V16);
\draw[edge] (V9) -- node[edgeLabel] {$20$} (V14);
\draw[edge] (V1) -- node[edgeLabel] {$21$} (V2);
\draw[edge] (V5) -- node[edgeLabel] {$22$} (V6);
\draw[edge] (V9) -- node[edgeLabel] {$23$} (V10);
\draw[edge] (V13) -- node[edgeLabel] {$24$} (V14);
\vertexLabelR{V1}{left}{$ $}
\vertexLabelR{V2}{left}{$ $}
\vertexLabelR{V3}{left}{$ $}
\vertexLabelR{V4}{left}{$ $}
\vertexLabelR{V5}{left}{$ $}
\vertexLabelR{V6}{left}{$ $}
\vertexLabelR{V7}{left}{$ $}
\vertexLabelR{V8}{left}{$ $}
\vertexLabelR{V9}{left}{$ $}
\vertexLabelR{V10}{left}{$ $}
\vertexLabelR{V11}{left}{$ $}
\vertexLabelR{V12}{left}{$ $}
\vertexLabelR{V13}{left}{$ $}
\vertexLabelR{V14}{left}{$ $}
\vertexLabelR{V15}{left}{$ $}
\vertexLabelR{V16}{left}{$ $}

\node at (0.353553390593274 , 1.060660171779821-0.2) {$a_1$};
\node at (1.060660171779821-0.2 , 0.3535533905932737) {$b_1$};
\node at (1.414213562373095-0.2 , 1.414213562373095+0.1) {$c_1$};
\node at (2.121320343559642 +0.1, 2.12132034355964+0.2) {$d_1$};
\node at (-1.060660171779821+0.2 , 0.353553390593274) {$a_2$};
\node at (-0.3535533905932737 , 1.060660171779821-0.2) {$b_2$};
\node at (-1.414213562373095+0.1 , 1.414213562373095+0.2) {$c_2$};
\node at (-2.121320343559642-0.1 , 2.121320343559642+0.2){$d_2$};
\node at (-0.3535533905932741 , -1.060660171779821+0.2) {$a_3$};
\node at (-1.060660171779821 +0.2, -0.3535533905932736) {$b_3$};
\node at (-1.414213562373095-0.2 , -1.414213562373095+0.1) {$c_3$};
\node at (-2.121320343559643-0.1 , -2.121320343559642-0.2) {$d_3$};
\node at (1.060660171779821-0.2 , -0.3535533905932741) {$a_4$};
\node at (0.3535533905932736 , -1.060660171779821+0.2) {$b_4$};
\node at (1.414213562373095+0.2 , -1.414213562373095+0.1) {$c_4$};
\node at (2.121320343559642+0.1 , -2.121320343559643-0.2) {$d_4$};
\end{tikzpicture}

%% file: FacegraphD7.tex
\begin{tikzpicture}[vertexBall, edgeDouble=nolabels, faceStyle=nolabels, scale=3]

\coordinate (V1) at (0.6227249993569783 , 0.5088359504595785);
\coordinate (V2) at (0.7860866519424037 , 0.1696119834865262);
\coordinate (V3) at (1.043629792622743 , 0.5025856195584777);
\coordinate (V4) at (1.382853759595796 , 0.6659472721439031);
\coordinate (V5) at (-0.009561279019278857 , 0.8041200353478124);
\coordinate (V6) at (0.3575090223697507 , 0.7203386344133406);
\coordinate (V7) at (0.2577552726097076 , 1.129299636269606);
\coordinate (V8) at (0.3415366735441793 , 1.496369937658636);
\coordinate (V9) at (-0.6346477192794708 , 0.4938853325597122);
\coordinate (V10) at (-0.3402801929023528 , 0.7286356013966028);
\coordinate (V11) at (-0.7222142249278024 , 0.9056279933552753);
\coordinate (V12) at (-0.9569644937646932 , 1.199995519732393);
\coordinate (V13) at (-0.78183148246803 , -0.188255099070633);
\coordinate (V14) at (-0.7818314824680299 , 0.1882550990706334);
\coordinate (V15) at (-1.158341680609296 , 2.664535259100376e-16);
\coordinate (V16) at (-1.534851878750563 , 3.552713678800501e-16);
\coordinate (V17) at (-0.3402801929023533 , -0.7286356013966028);
\coordinate (V18) at (-0.6346477192794712 , -0.4938853325597118);
\coordinate (V19) at (-0.7222142249278031 , -0.9056279933552752);
\coordinate (V20) at (-0.9569644937646942 , -1.199995519732393);
\coordinate (V21) at (0.3575090223697505 , -0.720338634413341);
\coordinate (V22) at (-0.009561279019279235 , -0.8041200353478124);
\coordinate (V23) at (0.2577552726097071 , -1.129299636269606);
\coordinate (V24) at (0.3415366735441785 , -1.496369937658636);
\coordinate (V25) at (0.7860866519424038 , -0.1696119834865267);
\coordinate (V26) at (0.6227249993569782 , -0.5088359504595789);
\coordinate (V27) at (1.043629792622743 , -0.5025856195584784);
\coordinate (V28) at (1.382853759595796 , -0.665947272143904);
\draw[edge] (V3) -- node[edgeLabel] {$1$} (V4);
\draw[edge] (V27) -- node[edgeLabel] {$2$} (V28);
\draw[edge] (V2) -- node[edgeLabel] {$3$} (V3);
\draw[edge] (V25) -- node[edgeLabel] {$4$} (V27);
\draw[edge] (V4) -- node[edgeLabel] {$5$} (V28);
\draw[edge] (V2) -- node[edgeLabel] {$6$} (V25);
\draw[edge] (V7) -- node[edgeLabel] {$7$} (V8);
\draw[edge] (V1) -- node[edgeLabel] {$8$} (V3);
\draw[edge] (V6) -- node[edgeLabel] {$9$} (V7);
\draw[edge] (V4) -- node[edgeLabel] {$10$} (V8);
\draw[edge] (V1) -- node[edgeLabel] {$11$} (V6);
\draw[edge] (V11) -- node[edgeLabel] {$12$} (V12);
\draw[edge] (V5) -- node[edgeLabel] {$13$} (V7);
\draw[edge] (V10) -- node[edgeLabel] {$14$} (V11);
\draw[edge] (V8) -- node[edgeLabel] {$15$} (V12);
\draw[edge] (V5) -- node[edgeLabel] {$16$} (V10);
\draw[edge] (V15) -- node[edgeLabel] {$17$} (V16);
\draw[edge] (V9) -- node[edgeLabel] {$18$} (V11);
\draw[edge] (V14) -- node[edgeLabel] {$19$} (V15);
\draw[edge] (V12) -- node[edgeLabel] {$20$} (V16);
\draw[edge] (V9) -- node[edgeLabel] {$21$} (V14);
\draw[edge] (V19) -- node[edgeLabel] {$22$} (V20);
\draw[edge] (V13) -- node[edgeLabel] {$23$} (V15);
\draw[edge] (V18) -- node[edgeLabel] {$24$} (V19);
\draw[edge] (V16) -- node[edgeLabel] {$25$} (V20);
\draw[edge] (V13) -- node[edgeLabel] {$26$} (V18);
\draw[edge] (V23) -- node[edgeLabel] {$27$} (V24);
\draw[edge] (V17) -- node[edgeLabel] {$28$} (V19);
\draw[edge] (V22) -- node[edgeLabel] {$29$} (V23);
\draw[edge] (V20) -- node[edgeLabel] {$30$} (V24);
\draw[edge] (V17) -- node[edgeLabel] {$31$} (V22);
\draw[edge] (V21) -- node[edgeLabel] {$32$} (V23);
\draw[edge] (V26) -- node[edgeLabel] {$33$} (V27);
\draw[edge] (V24) -- node[edgeLabel] {$34$} (V28);
\draw[edge] (V21) -- node[edgeLabel] {$35$} (V26);
\draw[edge] (V1) -- node[edgeLabel] {$36$} (V2);
\draw[edge] (V5) -- node[edgeLabel] {$37$} (V6);
\draw[edge] (V9) -- node[edgeLabel] {$38$} (V10);
\draw[edge] (V13) -- node[edgeLabel] {$39$} (V14);
\draw[edge] (V17) -- node[edgeLabel] {$40$} (V18);
\draw[edge] (V21) -- node[edgeLabel] {$41$} (V22);
\draw[edge] (V25) -- node[edgeLabel] {$42$} (V26);
\vertexLabelR{V2}{left} {$ $}
\vertexLabelR{V2}{left}{$ $}
\vertexLabelR{V3}{left}{$ $}
\vertexLabelR{V4}{left}{$ $}
\vertexLabelR{V5}{left}{$ $}
\vertexLabelR{V6}{left}{$ $}
\vertexLabelR{V7}{left}{$ $}
\vertexLabelR{V8}{left}{$ $}
\vertexLabelR{V9}{left}{$ $}
\vertexLabelR{V10}{left}{$ $}
\vertexLabelR{V11}{left}{$ $}
\vertexLabelR{V12}{left}{$ $}
\vertexLabelR{V13}{left}{$ $}
\vertexLabelR{V14}{left}{$ $}
\vertexLabelR{V15}{left}{$ $}
\vertexLabelR{V16}{left}{$ $}
\vertexLabelR{V17}{left}{$ $}
\vertexLabelR{V18}{left}{$ $}
\vertexLabelR{V19}{left}{$ $}
\vertexLabelR{V20}{left}{$ $}
\vertexLabelR{V21}{left}{$ $}
\vertexLabelR{V22}{left}{$ $}
\vertexLabelR{V23}{left}{$ $}
\vertexLabelR{V24}{left}{$ $}
\vertexLabelR{V25}{left}{$ $}
\vertexLabelR{V26}{left}{$ $}
\vertexLabelR{V27}{left}{$ $}
\vertexLabelR{V28}{left}{$ $}

\node at (0.6227249993569783-0.1 , 0.5088359504595785) {$a_1$};
\node at (0.7860866519424037-0.1 , 0.1696119834865262) {$b_1$};
\node at (1.043629792622743 , 0.5025856195584777+0.1) {$c_1$};
\node at (1.382853759595796+0.1 , 0.6659472721439031) {$d_1$};
\node at (-0.009561279019278857 , 0.8041200353478124-0.1) {$a_2$};
\node at (0.3575090223697507 , 0.7203386344133406-0.1) {$b_2$};
\node at (0.2577552726097076+0.1 , 1.129299636269606) {$c_2$};
\node at (0.3415366735441793 , 1.496369937658636+0.1) {$d_2$};
\node at (-0.6346477192794708+0.1 , 0.4938853325597122-0.1) {$a_3$};
\node at (-0.3402801929023528+0.1 , 0.7286356013966028-0.1) {$b_3$};
\node at (-0.7222142249278024+0.1 , 0.9056279933552753+0.1) {$c_3$};
\node at (-0.9569644937646932 , 1.199995519732393+0.1) {$d_3$};
\node at (-0.78183148246803+0.1 , -0.188255099070633) {$a_4$};
\node at (-0.7818314824680299+0.1 , 0.1882550990706334) {$b_4$};
\node at (-1.158341680609296 , 0.1+2.664535259100376e-16) {$c_4$};
\node at (-1.534851878750563-0.1 , 3.552713678800501e-16) {$d_4$};
\node at (-0.3402801929023533 , -0.7286356013966028+0.1) {$a_5$};
\node at (-0.6346477192794712+0.1 , -0.4938853325597118+0.1) {$b_5$};
\node at (-0.7222142249278031+0.15 , -0.9056279933552752) {$c_5$};
\node at (-0.9569644937646942 , -1.199995519732393-0.1) {$d_5$};
\node at (0.3575090223697505 , -0.720338634413341+0.1) {$a_6$};
\node at (-0.009561279019279235 , -0.8041200353478124+0.1) {$b_6$};
\node at (0.2577552726097071+0.1 , -1.129299636269606) {$c_6$};
\node at (0.3415366735441785 , -1.496369937658636-0.1) {$d_6$};
\node at (0.7860866519424038-0.1 , -0.166119834865267) {$a_7$};
\node at (0.6227249993569782 -0.1, -0.5088359504595789+0.1) {$b_7$};
\node at (1.043629792622743 , -0.5025856195584784+0.1) {$c_7$};
\node  at (1.382853759595796+0.1 , -0.665947272143904) {$d_7$};
\end{tikzpicture}

%% file: facegraph_octahedron.tex
\begin{tikzpicture}[vertexBall, edgeDouble=nolabels, faceStyle=nolabels, scale=4]

\coordinate (V1) at (0.,0);
\coordinate (V2) at (0 , 1.);
\coordinate (V3) at (1,1);
\coordinate (V4) at (1,0);
\coordinate (V5) at (1/3,1/3);
\coordinate (V6) at (1/3,2/3);
\coordinate (V7) at (2/3,2/3);
\coordinate (V8) at (2/3,1/3);
\draw[edge=] (V1) -- node[edgeLabel] {$1$} (V2);
\draw[edge=] (V1) -- node[edgeLabel] {$2$} (V5);
\draw[edge=] (V1) -- node[edgeLabel] {$3$} (V4);
\draw[edge=] (V2) -- node[edgeLabel] {$4$} (V3);
\draw[edge=] (V2) -- node[edgeLabel] {$5$} (V6);
\draw[edge=] (V3) -- node[edgeLabel] {$6$} (V4);
\draw[edge=] (V3) -- node[edgeLabel] {$7$} (V7);
\draw[edge=] (V4) -- node[edgeLabel] {$8$} (V8);
\draw[edge=] (V5) -- node[edgeLabel] {$9$} (V6);
\draw[edge=] (V5) -- node[edgeLabel] {$10$} (V8);
\draw[edge=] (V6) -- node[edgeLabel] {$11$} (V7);
\draw[edge=] (V7) -- node[edgeLabel] {$12$} (V8);

\vertexLabelR[]{V1}{left}{$ $}
\vertexLabelR[]{V2}{left}{$ $}
\vertexLabelR[]{V3}{left}{$ $}
\vertexLabelR[]{V4}{left}{$ $}
\vertexLabelR[]{V5}{left}{$ $}
\vertexLabelR[]{V6}{left}{$ $}
\vertexLabelR[]{V7}{left}{$ $}
\vertexLabelR[]{V8}{left}{$ $}

\node at (0.-0.1,0) {$1$};
\node at (0-0.1 , 1.) {$2$};
\node at (1+0.1,1) {$3$};
\node at (1+0.1,0) {$4$};
\node at (1/3-0.1,1/3) {$5$};
\node at (1/3-0.1,2/3) {$6$};
\node at (2/3+0.1,2/3) {$7$};
\node at (2/3+0.1,1/3) {$8$};
\end{tikzpicture}

%% file: A5_60_31.tex
\begin{tikzpicture}[vertexBall, edgeDouble=nolabels, faceStyle=nolabels, scale=2]

\coordinate (V1) at (3.965158301548441 , 2.877067889998683);
\coordinate (V2) at (-3.991159733603308 , -0.2657893543082068);
\coordinate (V3) at (-1.431293023358366 , 2.439549196324122);
\coordinate (V4) at (-3.413483104569645 , -0.5900280457889908);
\coordinate (V5) at (1.868283956781413 , -4.063190256046742);
\coordinate (V6) at (-1.4868939088375 , -0.1978547544587713);
\coordinate (V7) at (3.900768512385266 , 2.187236844236042);
\coordinate (V8) at (-3.386063735041599 , -2.1294535407555);
\coordinate (V9) at (4.898979485566356 , 0.);
\coordinate (V10) at (1.878805187934591 , 2.114258987397255);
\coordinate (V11) at (-3.109107291094604 , -1.527564025650765);
\coordinate (V12) at (0.2764313410772409 , -2.814886447739986);
\coordinate (V13) at (4.441774107998704 , 0.5202333836965094);
\coordinate (V14) at (-1.31883113651155 , -0.7146218813943871);
\coordinate (V15) at (-3.073655384396769 , 2.559812996679395);
\coordinate (V16) at (-1.086942783673777 , 1.033709526423889);
\coordinate (V17) at (-0.4933401410509959 , -3.428792135027696);
\coordinate (V18) at (3.386063735041599 , -2.1294535407555);
\coordinate (V19) at (-0.2764313410772409 , -2.814886447739986);
\coordinate (V20) at (-3.965158301548441 , 2.877067889998683);
\coordinate (V21) at (-1.616736726893802 , 3.063684441308001);
\coordinate (V22) at (1.513873982240009 , 4.659204391942553);
\coordinate (V23) at (-4.441774107998704 , 0.5202333836965094);
\coordinate (V24) at (3.991159733603308 , -0.2657893543082068);
\coordinate (V25) at (1.4868939088375 , -0.1978547544587713);
\coordinate (V26) at (-2.762627346492922 , -0.6065394829765657);
\coordinate (V27) at (1.31883113651155 , -0.7146218813943871);
\coordinate (V28) at (3.965158301548441 , -2.877067889998683);
\coordinate (V29) at (0.4933401410509959 , -3.428792135027696);
\coordinate (V30) at (-0.8771662462291147 , 4.385268449761808);
\coordinate (V31) at (3.073655384396769 , 2.559812996679395);
\coordinate (V32) at (-3.288608371262429 , -3.030685562783226);
\coordinate (V33) at (0.6474516389441235 , 1.353072937881979);
\coordinate (V34) at (-2.592110785022198 , -1.131795775823362);
\coordinate (V35) at (-1.513873982240009 , -4.659204391942553);
\coordinate (V36) at (-1.878805187934591 , 2.114258987397255);
\coordinate (V37) at (1.486465059042501 , 3.713545694918237);
\coordinate (V38) at (-3.965158301548441 , -2.877067889998683);
\coordinate (V39) at (3.288608371262429 , -3.030685562783226);
\coordinate (V40) at (1.616736726893802 , 3.063684441308001);
\coordinate (V41) at (0.9806423188398485 , -3.877929942959311);
\coordinate (V42) at (3.413483104569645 , -0.5900280457889908);
\coordinate (V43) at (-0.6474516389441235 , 1.353072937881979);
\coordinate (V44) at (-1.513873982240009 , 4.659204391942553);
\coordinate (V45) at (2.762627346492922 , -0.6065394829765657);
\coordinate (V46) at (-1.868283956781413 , -4.063190256046742);
\coordinate (V47) at (2.592110785022198 , -1.131795775823362);
\coordinate (V48) at (2.4152474873967 , 2.483259868524421);
\coordinate (V49) at (-1.486465059042501 , 3.713545694918237);
\coordinate (V50) at (1.086942783673777 , 1.033709526423889);
\coordinate (V51) at (1.513873982240009 , -4.659204391942553);
\coordinate (V52) at (0.271975138130117 , -1.475137120487144);
\coordinate (V53) at (1.431293023358366 , 2.439549196324122);
\coordinate (V54) at (-3.900768512385266 , 2.187236844236042);
\coordinate (V55) at (3.109107291094604 , -1.527564025650765);
\coordinate (V56) at (0.8771662462291147 , 4.385268449761808);
\coordinate (V57) at (-4.898979485566356 , 0.);
\coordinate (V58) at (-2.4152474873967 , 2.483259868524421);
\coordinate (V59) at (-0.271975138130117 , -1.475137120487144);
\coordinate (V60) at (-0.9806423188398485 , -3.877929942959311);
\draw[edge=blue] (V1) -- node[edgeLabel] {$1$} (V7);
\draw[edge=red] (V1) -- node[edgeLabel] {$2$} (V9);
\draw[edge=green] (V1) -- node[edgeLabel] {$3$} (V22);
\draw[edge=red] (V2) -- node[edgeLabel] {$4$} (V4);
\draw[edge=blue] (V2) -- node[edgeLabel] {$5$} (V8);
\draw[edge=green] (V2) -- node[edgeLabel] {$6$} (V23);
\draw[edge=blue] (V3) -- node[edgeLabel] {$7$} (V6);
\draw[edge=red] (V3) -- node[edgeLabel] {$8$} (V14);
\draw[edge=green] (V3) -- node[edgeLabel] {$9$} (V21);
\draw[edge=blue] (V4) -- node[edgeLabel] {$10$} (V11);
\draw[edge=green] (V4) -- node[edgeLabel] {$11$} (V26);
\draw[edge=red] (V5) -- node[edgeLabel] {$12$} (V39);
\draw[edge=green] (V5) -- node[edgeLabel] {$13$} (V41);
\draw[edge=blue] (V5) -- node[edgeLabel] {$14$} (V51);
\draw[edge=red] (V6) -- node[edgeLabel] {$15$} (V12);
\draw[edge=green] (V6) -- node[edgeLabel] {$16$} (V33);
\draw[edge=red] (V7) -- node[edgeLabel] {$17$} (V13);
\draw[edge=green] (V7) -- node[edgeLabel] {$18$} (V31);
\draw[edge=red] (V8) -- node[edgeLabel] {$19$} (V11);
\draw[edge=green] (V8) -- node[edgeLabel] {$20$} (V32);
\draw[edge=blue] (V9) -- node[edgeLabel] {$21$} (V13);
\draw[edge=green] (V9) -- node[edgeLabel] {$22$} (V28);
\draw[edge=red] (V10) -- node[edgeLabel] {$23$} (V16);
\draw[edge=blue] (V10) -- node[edgeLabel] {$24$} (V43);
\draw[edge=green] (V10) -- node[edgeLabel] {$25$} (V48);
\draw[edge=green] (V11) -- node[edgeLabel] {$26$} (V34);
\draw[edge=blue] (V12) -- node[edgeLabel] {$27$} (V14);
\draw[edge=green] (V12) -- node[edgeLabel] {$28$} (V29);
\draw[edge=green] (V13) -- node[edgeLabel] {$29$} (V24);
\draw[edge=green] (V14) -- node[edgeLabel] {$30$} (V27);
\draw[edge=blue] (V15) -- node[edgeLabel] {$31$} (V49);
\draw[edge=green] (V15) -- node[edgeLabel] {$32$} (V54);
\draw[edge=red] (V15) -- node[edgeLabel] {$33$} (V58);
\draw[edge=blue] (V16) -- node[edgeLabel] {$34$} (V34);
\draw[edge=green] (V16) -- node[edgeLabel] {$35$} (V59);
\draw[edge=green] (V17) -- node[edgeLabel] {$36$} (V19);
\draw[edge=blue] (V17) -- node[edgeLabel] {$37$} (V29);
\draw[edge=red] (V17) -- node[edgeLabel] {$38$} (V60);
\draw[edge=blue] (V18) -- node[edgeLabel] {$39$} (V24);
\draw[edge=green] (V18) -- node[edgeLabel] {$40$} (V39);
\draw[edge=red] (V18) -- node[edgeLabel] {$41$} (V55);
\draw[edge=red] (V19) -- node[edgeLabel] {$42$} (V25);
\draw[edge=blue] (V19) -- node[edgeLabel] {$43$} (V27);
\draw[edge=green] (V20) -- node[edgeLabel] {$44$} (V44);
\draw[edge=blue] (V20) -- node[edgeLabel] {$45$} (V54);
\draw[edge=red] (V20) -- node[edgeLabel] {$46$} (V57);
\draw[edge=red] (V21) -- node[edgeLabel] {$47$} (V49);
\draw[edge=blue] (V21) -- node[edgeLabel] {$48$} (V58);
\draw[edge=red] (V22) -- node[edgeLabel] {$49$} (V44);
\draw[edge=blue] (V22) -- node[edgeLabel] {$50$} (V56);
\draw[edge=red] (V23) -- node[edgeLabel] {$51$} (V54);
\draw[edge=blue] (V23) -- node[edgeLabel] {$52$} (V57);
\draw[edge=red] (V24) -- node[edgeLabel] {$53$} (V42);
\draw[edge=green] (V25) -- node[edgeLabel] {$54$} (V43);
\draw[edge=blue] (V25) -- node[edgeLabel] {$55$} (V53);
\draw[edge=red] (V26) -- node[edgeLabel] {$56$} (V52);
\draw[edge=blue] (V26) -- node[edgeLabel] {$57$} (V59);
\draw[edge=red] (V27) -- node[edgeLabel] {$58$} (V53);
\draw[edge=blue] (V28) -- node[edgeLabel] {$59$} (V39);
\draw[edge=red] (V28) -- node[edgeLabel] {$60$} (V51);
\draw[edge=red] (V29) -- node[edgeLabel] {$61$} (V41);
\draw[edge=blue] (V30) -- node[edgeLabel] {$62$} (V44);
\draw[edge=green] (V30) -- node[edgeLabel] {$63$} (V49);
\draw[edge=red] (V30) -- node[edgeLabel] {$64$} (V56);
\draw[edge=blue] (V31) -- node[edgeLabel] {$65$} (V37);
\draw[edge=red] (V31) -- node[edgeLabel] {$66$} (V48);
\draw[edge=blue] (V32) -- node[edgeLabel] {$67$} (V38);
\draw[edge=red] (V32) -- node[edgeLabel] {$68$} (V46);
\draw[edge=blue] (V33) -- node[edgeLabel] {$69$} (V36);
\draw[edge=red] (V33) -- node[edgeLabel] {$70$} (V47);
\draw[edge=red] (V34) -- node[edgeLabel] {$71$} (V43);
\draw[edge=red] (V35) -- node[edgeLabel] {$72$} (V38);
\draw[edge=blue] (V35) -- node[edgeLabel] {$73$} (V46);
\draw[edge=green] (V35) -- node[edgeLabel] {$74$} (V51);
\draw[edge=red] (V36) -- node[edgeLabel] {$75$} (V50);
\draw[edge=green] (V36) -- node[edgeLabel] {$76$} (V58);
\draw[edge=red] (V37) -- node[edgeLabel] {$77$} (V40);
\draw[edge=green] (V37) -- node[edgeLabel] {$78$} (V56);
\draw[edge=green] (V38) -- node[edgeLabel] {$79$} (V57);
\draw[edge=blue] (V40) -- node[edgeLabel] {$80$} (V48);
\draw[edge=green] (V40) -- node[edgeLabel] {$81$} (V53);
\draw[edge=blue] (V41) -- node[edgeLabel] {$82$} (V60);
\draw[edge=green] (V42) -- node[edgeLabel] {$83$} (V45);
\draw[edge=blue] (V42) -- node[edgeLabel] {$84$} (V55);
\draw[edge=blue] (V45) -- node[edgeLabel] {$85$} (V52);
\draw[edge=red] (V45) -- node[edgeLabel] {$86$} (V59);
\draw[edge=green] (V46) -- node[edgeLabel] {$87$} (V60);
\draw[edge=blue] (V47) -- node[edgeLabel] {$88$} (V50);
\draw[edge=green] (V47) -- node[edgeLabel] {$89$} (V55);
\draw[edge=green] (V50) -- node[edgeLabel] {$90$} (V52);
\vertexLabelR[]{V1}{left}{$   $}
\vertexLabelR[]{V2}{left}{$   $}
\vertexLabelR[]{V3}{left}{$   $}
\vertexLabelR[]{V4}{left}{$   $}
\vertexLabelR[]{V5}{left}{$   $}
\vertexLabelR[]{V6}{left}{$   $}
\vertexLabelR[]{V7}{left}{$   $}
\vertexLabelR[]{V8}{left}{$   $}
\vertexLabelR[]{V9}{left}{$   $}
\vertexLabelR[]{V10}{left}{$   $}
\vertexLabelR[]{V11}{left}{$   $}
\vertexLabelR[]{V12}{left}{$   $}
\vertexLabelR[]{V13}{left}{$   $}
\vertexLabelR[]{V14}{left}{$   $}
\vertexLabelR[]{V15}{left}{$   $}
\vertexLabelR[]{V16}{left}{$   $}
\vertexLabelR[]{V17}{left}{$   $}
\vertexLabelR[]{V18}{left}{$   $}
\vertexLabelR[]{V19}{left}{$   $}
\vertexLabelR[]{V20}{left}{$   $}
\vertexLabelR[]{V21}{left}{$   $}
\vertexLabelR[]{V22}{left}{$   $}
\vertexLabelR[]{V23}{left}{$   $}
\vertexLabelR[]{V24}{left}{$   $}
\vertexLabelR[]{V25}{left}{$   $}
\vertexLabelR[]{V26}{left}{$   $}
\vertexLabelR[]{V27}{left}{$   $}
\vertexLabelR[]{V28}{left}{$   $}
\vertexLabelR[]{V29}{left}{$   $}
\vertexLabelR[]{V30}{left}{$   $}
\vertexLabelR[]{V31}{left}{$   $}
\vertexLabelR[]{V32}{left}{$   $}
\vertexLabelR[]{V33}{left}{$   $}
\vertexLabelR[]{V34}{left}{$   $}
\vertexLabelR[]{V35}{left}{$   $}
\vertexLabelR[]{V36}{left}{$   $}
\vertexLabelR[]{V37}{left}{$   $}
\vertexLabelR[]{V38}{left}{$   $}
\vertexLabelR[]{V39}{left}{$   $}
\vertexLabelR[]{V40}{left}{$   $}
\vertexLabelR[]{V41}{left}{$   $}
\vertexLabelR[]{V42}{left}{$   $}
\vertexLabelR[]{V43}{left}{$   $}
\vertexLabelR[]{V44}{left}{$   $}
\vertexLabelR[]{V45}{left}{$   $}
\vertexLabelR[]{V46}{left}{$   $}
\vertexLabelR[]{V47}{left}{$   $}
\vertexLabelR[]{V48}{left}{$   $}
\vertexLabelR[]{V49}{left}{$   $}
\vertexLabelR[]{V50}{left}{$   $}
\vertexLabelR[]{V51}{left}{$   $}
\vertexLabelR[]{V52}{left}{$   $}
\vertexLabelR[]{V53}{left}{$   $}
\vertexLabelR[]{V54}{left}{$   $}
\vertexLabelR[]{V55}{left}{$   $}
\vertexLabelR[]{V56}{left}{$   $}
\vertexLabelR[]{V57}{left}{$   $}
\vertexLabelR[]{V58}{left}{$   $}
\vertexLabelR[]{V59}{left}{$   $}
\vertexLabelR[]{V60}{left}{$   $}
\end{tikzpicture}

%% file: A5_16_2.tex
\begin{tikzpicture}[vertexBall, edgeDouble=nolabels, faceStyle=nolabels, scale=1.3]

\coordinate (V1) at (2.5 , 4.33);
\coordinate (V2) at (-2.433421036746078 , 0.05630827145951441);
\coordinate (V3) at (-0.4179927772481275 , -0.2413211633979189);
\coordinate (V4) at (1.435965956300105 , 0.8290310121039276);
\coordinate (V5) at (-1.169247992592014 , -2.136249191064039);
\coordinate (V6) at (5. , 0.);
\coordinate (V7) at (2.433486057844761 , 0.05630827145951441);
\coordinate (V8) at (0.4177141153966287 , -0.2413211633979189);
\coordinate (V9) at (-2.495 , 4.33);
\coordinate (V10) at (-5. , 0.);
\coordinate (V11) at (-1.435008645662572 , 0.8290310121039276);
\coordinate (V12) at (1.166781192140438 , -2.136249191064039);
\coordinate (V13) at (-1.99855604377276 , 0.7016876904954873);
\coordinate (V14) at (-0.2381884928267052 , -3.28110154111232);
\coordinate (V15) at (2.959506698061767 , 1.435922799439002);
\coordinate (V16) at (1.607406094962054 , 1.380622242727742);
\coordinate (V17) at (0.3904174598743416 , -2.082309933223229);
\coordinate (V18) at (-2.721563894232468 , 1.845178741673319);
\coordinate (V19) at (2.495 , -4.33);
\coordinate (V20) at (1.266714619517894 , 2.079940919604525);
\coordinate (V21) at (0.0002786618514987516 , 0.4826423267958379);
\coordinate (V22) at (-0.0009573106375334037 , -1.658062024207855);
\coordinate (V23) at (-1.264312840165001 , 2.079940919604525);
\coordinate (V24) at (-2.5 , -4.33);
\coordinate (V25) at (-2.957848588593592 , 1.435922799439002);
\coordinate (V26) at (1.99936630669481 , 0.7016876904954873);
\coordinate (V27) at (0.2343996919709173 , -3.28110154111232);
\coordinate (V28) at (-0.3928219748087795 , -2.082309933223229);
\coordinate (V29) at (2.723694585620081 , 1.845178741673319);
\coordinate (V30) at (-1.605811842949666 , 1.380622242727742);
\draw[edge=red] (V1) -- node[edgeLabel] {$1$} (V6);
\draw[edge=blue] (V1) -- node[edgeLabel] {$2$} (V9);
\draw[edge=green] (V1) -- node[edgeLabel] {$3$} (V16);
\draw[edge=blue] (V2) -- node[edgeLabel] {$4$} (V5);
\draw[edge=green] (V2) -- node[edgeLabel] {$5$} (V11);
\draw[edge=red] (V2) -- node[edgeLabel] {$6$} (V25);
\draw[edge=red] (V3) -- node[edgeLabel] {$7$} (V4);
\draw[edge=blue] (V3) -- node[edgeLabel] {$8$} (V13);
\draw[edge=green] (V3) -- node[edgeLabel] {$9$} (V28);
\draw[edge=green] (V4) -- node[edgeLabel] {$10$} (V7);
\draw[edge=blue] (V4) -- node[edgeLabel] {$11$} (V20);
\draw[edge=green] (V5) -- node[edgeLabel] {$12$} (V14);
\draw[edge=red] (V5) -- node[edgeLabel] {$13$} (V22);
\draw[edge=blue] (V6) -- node[edgeLabel] {$14$} (V19);
\draw[edge=green] (V6) -- node[edgeLabel] {$15$} (V26);
\draw[edge=red] (V7) -- node[edgeLabel] {$16$} (V12);
\draw[edge=blue] (V7) -- node[edgeLabel] {$17$} (V15);
\draw[edge=blue] (V8) -- node[edgeLabel] {$18$} (V11);
\draw[edge=green] (V8) -- node[edgeLabel] {$19$} (V17);
\draw[edge=red] (V8) -- node[edgeLabel] {$20$} (V26);
\draw[edge=red] (V9) -- node[edgeLabel] {$21$} (V10);
\draw[edge=green] (V9) -- node[edgeLabel] {$22$} (V30);
\draw[edge=green] (V10) -- node[edgeLabel] {$23$} (V13);
\draw[edge=blue] (V10) -- node[edgeLabel] {$24$} (V24);
\draw[edge=red] (V11) -- node[edgeLabel] {$25$} (V23);
\draw[edge=blue] (V12) -- node[edgeLabel] {$26$} (V22);
\draw[edge=green] (V12) -- node[edgeLabel] {$27$} (V27);
\draw[edge=red] (V13) -- node[edgeLabel] {$28$} (V18);
\draw[edge=blue] (V14) -- node[edgeLabel] {$29$} (V17);
\draw[edge=red] (V14) -- node[edgeLabel] {$30$} (V27);
\draw[edge=red] (V15) -- node[edgeLabel] {$31$} (V16);
\draw[edge=green] (V15) -- node[edgeLabel] {$32$} (V29);
\draw[edge=blue] (V16) -- node[edgeLabel] {$33$} (V21);
\draw[edge=red] (V17) -- node[edgeLabel] {$34$} (V19);
\draw[edge=blue] (V18) -- node[edgeLabel] {$35$} (V23);
\draw[edge=green] (V18) -- node[edgeLabel] {$36$} (V25);
\draw[edge=green] (V19) -- node[edgeLabel] {$37$} (V24);
\draw[edge=green] (V20) -- node[edgeLabel] {$38$} (V23);
\draw[edge=red] (V20) -- node[edgeLabel] {$39$} (V29);
\draw[edge=green] (V21) -- node[edgeLabel] {$40$} (V22);
\draw[edge=red] (V21) -- node[edgeLabel] {$41$} (V30);
\draw[edge=red] (V24) -- node[edgeLabel] {$42$} (V28);
\draw[edge=blue] (V25) -- node[edgeLabel] {$43$} (V30);
\draw[edge=blue] (V26) -- node[edgeLabel] {$44$} (V29);
\draw[edge=blue] (V27) -- node[edgeLabel] {$45$} (V28);
\vertexLabelR[]{V1}{left}{$1$}
\vertexLabelR[]{V2}{left}{$2$}
\vertexLabelR[]{V3}{left}{$3$}
\vertexLabelR[]{V4}{left}{$4$}
\vertexLabelR[]{V5}{left}{$5$}
\vertexLabelR[]{V6}{left}{$6$}
\vertexLabelR[]{V7}{left}{$7$}
\vertexLabelR[]{V8}{left}{$8$}
\vertexLabelR[]{V9}{left}{$9$}
\vertexLabelR[]{V10}{left}{$10$}
\vertexLabelR[]{V11}{left}{$11$}
\vertexLabelR[]{V12}{left}{$12$}
\vertexLabelR[]{V13}{left}{$13$}
\vertexLabelR[]{V14}{left}{$14$}
\vertexLabelR[]{V15}{left}{$15$}
\vertexLabelR[]{V16}{left}{$16$}
\vertexLabelR[]{V17}{left}{$17$}
\vertexLabelR[]{V18}{left}{$18$}
\vertexLabelR[]{V19}{left}{$19$}
\vertexLabelR[]{V20}{left}{$20$}
\vertexLabelR[]{V21}{left}{$21$}
\vertexLabelR[]{V22}{left}{$22$}
\vertexLabelR[]{V23}{left}{$23$}
\vertexLabelR[]{V24}{left}{$24$}
\vertexLabelR[]{V25}{left}{$25$}
\vertexLabelR[]{V26}{left}{$26$}
\vertexLabelR[]{V27}{left}{$27$}
\vertexLabelR[]{V28}{left}{$28$}
\vertexLabelR[]{V29}{left}{$29$}
\vertexLabelR[]{V30}{left}{$30$}
\end{tikzpicture}

%% file: SnubDodecahedronExtended.tex
\begin{tikzpicture}[vertexBall, edgeDouble=nolabels, faceStyle=nolabels, scale=2]

\coordinate (V1) at (-3.024867444084355 , 2.3000384661389);
\coordinate (V2) at (-3.804190002771964 , 1.236178960672704);
\coordinate (V3) at (-1.249191226229794 , 3.588804993352314);
\coordinate (V4) at (-4.793322814389319 , 1.557451892372555);
\coordinate (V5) at (0. , 4.);
\coordinate (V6) at (0. , 5.04);
\coordinate (V7) at (1.027450368899934 , 0.6199562399455221);
\coordinate (V8) at (0.8016299985877654 , 0.5978204959385209);
\coordinate (V9) at (1.348922823197051 , 0.3747095102317242);
\coordinate (V10) at (1.196121491022951 , 0.09640217172364007);
\coordinate (V11) at (0.7272645608865574 , 0.3332960522995778);
\coordinate (V12) at (0.7945328804952505 , 0.09336756295373699);
\coordinate (V13) at (-0.3985445576146004 , 0.03409157661039482);
\coordinate (V14) at (-0.7842428790646876 , 0.157997172874485);
\coordinate (V15) at (-0.3724539181233598 , -0.1458700753223829);
\coordinate (V16) at (-0.7068901756939217 , -0.37457479828122);
\coordinate (V17) at (-0.999924280563422 , -0.01230581731227666);
\coordinate (V18) at (-0.949783099689771 , -0.3129090339758356);
\coordinate (V19) at (1.360471431181157 , 1.178608283077925);
\coordinate (V20) at (0.5844460650436534 , 0.8114325585376703);
\coordinate (V21) at (0.8701196080138841 , 1.34271808945503);
\coordinate (V22) at (0.556826040383429 , 1.062988598598739);
\coordinate (V23) at (0.4598771310735444 , -1.946410291874651);
\coordinate (V24) at (0.08502828234262039 , -1.597739087336184);
\coordinate (V25) at (0.9738660912677428 , -1.746878598036764);
\coordinate (V26) at (0.1740888309503289 , -1.187304964589279);
\coordinate (V27) at (0.7728200784728317 , -1.167368462101511);
\coordinate (V28) at (0.4610693938681 , -1.107887636016443);
\coordinate (V29) at (-2.831914910742094 , 1.88155731730835);
\coordinate (V30) at (-3.799172515488024 , -0.07929815609712981);
\coordinate (V31) at (-3.397017208316859 , 0.1423871005362915);
\coordinate (V32) at (-3.062049152713096 , 0.9294379949028401);
\coordinate (V33) at (-0.3540033864050212 , 0.1862299718460412);
\coordinate (V34) at (-0.6975950639305047 , 0.3916134915701898);
\coordinate (V35) at (-0.2537273863400213 , 0.3092287396427789);
\coordinate (V36) at (-0.5744869910678644 , 0.5567447324347053);
\coordinate (V37) at (1.962200532260396 , 0.3870000919858013);
\coordinate (V38) at (2.410508061665681 , 0.9743977035301162);
\coordinate (V39) at (1.959867155252077 , 1.385251144653608);
\coordinate (V40) at (0.4049249180372106 , -2.770565973001286);
\coordinate (V41) at (-0.4078968717950131 , -1.753174304505928);
\coordinate (V42) at (-0.771178180443339 , -2.272726163445148);
\coordinate (V43) at (-0.531550907754291 , -2.749082325516206);
\coordinate (V44) at (-1.956073351231147 , 2.003441300513532);
\coordinate (V45) at (-1.18525504637082 , 2.536763779907878);
\coordinate (V46) at (-0.9141839795808377 , 3.274792764661261);
\coordinate (V47) at (1.184579110460706 , -3.186969144980874);
\coordinate (V48) at (1.930325951986332 , -2.552222897610642);
\coordinate (V49) at (1.670804645690585 , -1.99208730630431);
\coordinate (V50) at (-0.2864054460552463 , -0.2792345259273922);
\coordinate (V51) at (-0.1432036804211941 , -0.1396162809769138);
\coordinate (V52) at (-0.1770013369590011 , 0.09311566310093145);
\coordinate (V53) at (-0.1555085106160095 , -0.368533720473433);
\coordinate (V54) at (0.02356824877642449 , -0.3993050683996042);
\coordinate (V55) at (0.176925614868163 , -0.3587440965417026);
\coordinate (V56) at (0.08846407780218239 , -0.1793714217444056);
\coordinate (V57) at (-3.121353811074246 , -2.167291024780078);
\coordinate (V58) at (-2.66371942947027 , -2.11295972537637);
\coordinate (V59) at (-2.509444777882384 , -1.24204947838596);
\coordinate (V60) at (-2.778821805454422 , -0.3437286335629102);
\coordinate (V61) at (1.993273742734456 , -0.1638895558764308);
\coordinate (V62) at (2.76011210853702 , -0.4709364588851966);
\coordinate (V63) at (3.397003393570344 , 0.14271630625674);
\coordinate (V64) at (3.023386961916789 , 1.048394619650239);
\coordinate (V65) at (1.545784081934439 , -0.4129789002335404);
\coordinate (V66) at (2.449778098866819 , -1.355945156085774);
\coordinate (V67) at (1.540995905011161 , -0.9302320252167379);
\coordinate (V68) at (1.922528308452974 , -1.436622742127156);
\coordinate (V69) at (-0.3924175759265322 , -0.6971430599984083);
\coordinate (V70) at (-0.5878091877288945 , -0.5426604452339394);
\coordinate (V71) at (-0.2718714798896767 , 1.168796773790293);
\coordinate (V72) at (0.0605774320465923 , 1.398688805534183);
\coordinate (V73) at (-0.7001109464997288 , 1.658265558525309);
\coordinate (V74) at (-0.711364224844796 , 2.292152032394658);
\coordinate (V75) at (0.2383588650070936 , 1.985745464925585);
\coordinate (V76) at (-0.1815107589002303 , 2.593656462294778);
\coordinate (V77) at (-0.2971579871804812 , -0.9548283252265013);
\coordinate (V78) at (0.003964933851074374 , -0.9999921396188856);
\coordinate (V79) at (0.1375935345564869 , -0.7880786884875475);
\coordinate (V80) at (-0.3206012033526726 , 0.9472142674225396);
\coordinate (V81) at (0.2778916936638581 , 1.16738006090246);
\coordinate (V82) at (-0.09210816256872685 , 0.7946798640888122);
\coordinate (V83) at (0.1567121485294322 , 0.7845006708112423);
\coordinate (V84) at (-2.349457672158911 , -3.237290324752111);
\coordinate (V85) at (-2.959886280053899 , -4.07929812702488);
\coordinate (V86) at (-1.09759508691955 , -3.638033125903346);
\coordinate (V87) at (-1.184263854028487 , -3.187086306337122);
\coordinate (V88) at (-1.829135690686329 , -2.625692789542876);
\coordinate (V89) at (1.182992000008321 , -0.2013204607493062);
\coordinate (V90) at (-1.007782614043469 , 1.242728531430623);
\coordinate (V91) at (-0.8386522325944363 , 0.8582904128349377);
\coordinate (V92) at (-0.5908658545302778 , 0.806769819682296);
\coordinate (V93) at (-0.4667692699716827 , -1.105498280690704);
\coordinate (V94) at (-0.8708716333176526 , -1.096167230983778);
\coordinate (V95) at (-1.359674573795039 , -1.466725963967121);
\coordinate (V96) at (-1.377012007197084 , -2.205411057384781);
\coordinate (V97) at (-1.516127202285222 , 1.304361263795727);
\coordinate (V98) at (0.3424275131703118 , 0.2067447658926243);
\coordinate (V99) at (0.3958458761374177 , 0.05749819427599565);
\coordinate (V100) at (0.3870373051871031 , -0.1010055661511051);
\coordinate (V101) at (0.7920289216511606 , -0.1126507313251892);
\coordinate (V102) at (-1.814772838897314 , 0.8405947556347146);
\coordinate (V103) at (-1.311452384317772 , 0.4899924934804933);
\coordinate (V104) at (-1.024235051481984 , 0.6252539958414489);
\coordinate (V105) at (3.799102782411756 , -0.08257147613585403);
\coordinate (V106) at (3.804190002771964 , 1.236178960672704);
\coordinate (V107) at (3.026847971335596 , 2.297431469798737);
\coordinate (V108) at (2.831735026946159 , 1.881828030710097);
\coordinate (V109) at (0.2155751350572285 , 0.3369382156198042);
\coordinate (V110) at (0.1979229182827637 , 0.02874923335385554);
\coordinate (V111) at (0.06763704977153011 , 0.3942400658205652);
\coordinate (V112) at (0.03381795144166661 , 0.1971201312912739);
\coordinate (V113) at (-2.522767538663147 , 0.6290023417027089);
\coordinate (V114) at (3.119484191757324 , -2.169981192862316);
\coordinate (V115) at (2.663920382260002 , -2.112706367903435);
\coordinate (V116) at (-0.9108570510778876 , -0.7812422367625133);
\coordinate (V117) at (2.349457672158911 , -3.237290324752111);
\coordinate (V118) at (2.959886280053899 , -4.07929812702488);
\coordinate (V119) at (4.793322814389319 , 1.557451892372555);
\coordinate (V120) at (0.3517521305326663 , 0.718519616061893);
\coordinate (V121) at (1.094465686371412 , -3.638975798402011);
\coordinate (V122) at (-1.195608082803177 , 0.1025734484928296);
\coordinate (V123) at (-1.793429847921778 , -0.1536534431220798);
\coordinate (V124) at (-1.493158570433431 , -0.5748717105069557);
\coordinate (V125) at (-1.075260799205943 , -0.5327421643637718);
\coordinate (V126) at (0.771591241335168 , 1.845168544143015);
\coordinate (V127) at (1.300394935870282 , 2.47971228386741);
\coordinate (V128) at (0.9138653575613026 , 3.274881693779083);
\coordinate (V129) at (-0.06251564412122056 , 3.199389284572934);
\coordinate (V130) at (0.3023429517296782 , -0.2618945198727639);
\coordinate (V131) at (0.5415053100514363 , -0.5888735001561012);
\coordinate (V132) at (0.816060025051816 , -0.5779671578147234);
\coordinate (V133) at (0.9522371120617287 , -0.3053595952517931);
\coordinate (V134) at (-2.399791888622925 , 0.03160524164784561);
\coordinate (V135) at (2.045947961743858 , 1.911569234382095);
\coordinate (V136) at (0.3341563731236057 , -0.7268696707807236);
\coordinate (V137) at (0.9068087795925752 , -0.7859375530243);
\coordinate (V138) at (-0.09071677712555659 , 0.3895772918792266);
\coordinate (V139) at (1.252281473455167 , 3.587727847989722);
\coordinate (V140) at (-1.708655304718638 , -1.039469600160033);
\draw[edge=] (V1) -- node[edgeLabel] {$1$} (V2);
\draw[edge=] (V1) -- node[edgeLabel] {$2$} (V3);
\draw[edge=] (V2) -- node[edgeLabel] {$3$} (V4);
\draw[edge=] (V5) -- node[edgeLabel] {$4$} (V6);
\draw[edge=] (V3) -- node[edgeLabel] {$5$} (V5);
\draw[edge=] (V4) -- node[edgeLabel] {$6$} (V6);
\draw[edge=] (V7) -- node[edgeLabel] {$7$} (V8);
\draw[edge=] (V7) -- node[edgeLabel] {$8$} (V9);
\draw[edge=] (V9) -- node[edgeLabel] {$9$} (V10);
\draw[edge=] (V11) -- node[edgeLabel] {$10$} (V12);
\draw[edge=] (V8) -- node[edgeLabel] {$11$} (V11);
\draw[edge=] (V10) -- node[edgeLabel] {$12$} (V12);
\draw[edge=] (V13) -- node[edgeLabel] {$13$} (V14);
\draw[edge=] (V13) -- node[edgeLabel] {$14$} (V15);
\draw[edge=] (V15) -- node[edgeLabel] {$15$} (V16);
\draw[edge=] (V17) -- node[edgeLabel] {$16$} (V18);
\draw[edge=] (V14) -- node[edgeLabel] {$17$} (V17);
\draw[edge=] (V16) -- node[edgeLabel] {$18$} (V18);
\draw[edge=] (V7) -- node[edgeLabel] {$19$} (V19);
\draw[edge=] (V8) -- node[edgeLabel] {$20$} (V20);
\draw[edge=] (V21) -- node[edgeLabel] {$21$} (V22);
\draw[edge=] (V19) -- node[edgeLabel] {$22$} (V21);
\draw[edge=] (V20) -- node[edgeLabel] {$23$} (V22);
\draw[edge=] (V23) -- node[edgeLabel] {$24$} (V24);
\draw[edge=] (V23) -- node[edgeLabel] {$25$} (V25);
\draw[edge=] (V24) -- node[edgeLabel] {$26$} (V26);
\draw[edge=] (V27) -- node[edgeLabel] {$27$} (V28);
\draw[edge=] (V25) -- node[edgeLabel] {$28$} (V27);
\draw[edge=] (V26) -- node[edgeLabel] {$29$} (V28);
\draw[edge=] (V1) -- node[edgeLabel] {$30$} (V29);
\draw[edge=] (V30) -- node[edgeLabel] {$31$} (V31);
\draw[edge=] (V2) -- node[edgeLabel] {$32$} (V30);
\draw[edge=] (V29) -- node[edgeLabel] {$33$} (V32);
\draw[edge=] (V31) -- node[edgeLabel] {$34$} (V32);
\draw[edge=] (V13) -- node[edgeLabel] {$35$} (V33);
\draw[edge=] (V14) -- node[edgeLabel] {$36$} (V34);
\draw[edge=] (V35) -- node[edgeLabel] {$37$} (V36);
\draw[edge=] (V33) -- node[edgeLabel] {$38$} (V35);
\draw[edge=] (V34) -- node[edgeLabel] {$39$} (V36);
\draw[edge=] (V37) -- node[edgeLabel] {$40$} (V38);
\draw[edge=] (V9) -- node[edgeLabel] {$41$} (V37);
\draw[edge=] (V19) -- node[edgeLabel] {$42$} (V39);
\draw[edge=] (V38) -- node[edgeLabel] {$43$} (V39);
\draw[edge=] (V23) -- node[edgeLabel] {$44$} (V40);
\draw[edge=] (V41) -- node[edgeLabel] {$45$} (V42);
\draw[edge=] (V24) -- node[edgeLabel] {$46$} (V41);
\draw[edge=] (V40) -- node[edgeLabel] {$47$} (V43);
\draw[edge=] (V42) -- node[edgeLabel] {$48$} (V43);
\draw[edge=] (V44) -- node[edgeLabel] {$49$} (V45);
\draw[edge=] (V29) -- node[edgeLabel] {$50$} (V44);
\draw[edge=] (V3) -- node[edgeLabel] {$51$} (V46);
\draw[edge=] (V45) -- node[edgeLabel] {$52$} (V46);
\draw[edge=] (V47) -- node[edgeLabel] {$53$} (V48);
\draw[edge=] (V40) -- node[edgeLabel] {$54$} (V47);
\draw[edge=] (V25) -- node[edgeLabel] {$55$} (V49);
\draw[edge=] (V48) -- node[edgeLabel] {$56$} (V49);
\draw[edge=] (V50) -- node[edgeLabel] {$57$} (V51);
\draw[edge=] (V15) -- node[edgeLabel] {$58$} (V50);
\draw[edge=] (V33) -- node[edgeLabel] {$59$} (V52);
\draw[edge=] (V51) -- node[edgeLabel] {$60$} (V52);
\draw[edge=] (V50) -- node[edgeLabel] {$61$} (V53);
\draw[edge=] (V53) -- node[edgeLabel] {$62$} (V54);
\draw[edge=] (V55) -- node[edgeLabel] {$63$} (V56);
\draw[edge=] (V54) -- node[edgeLabel] {$64$} (V55);
\draw[edge=] (V51) -- node[edgeLabel] {$65$} (V56);
\draw[edge=] (V30) -- node[edgeLabel] {$66$} (V57);
\draw[edge=] (V57) -- node[edgeLabel] {$67$} (V58);
\draw[edge=] (V59) -- node[edgeLabel] {$68$} (V60);
\draw[edge=] (V58) -- node[edgeLabel] {$69$} (V59);
\draw[edge=] (V31) -- node[edgeLabel] {$70$} (V60);
\draw[edge=] (V61) -- node[edgeLabel] {$71$} (V62);
\draw[edge=] (V37) -- node[edgeLabel] {$72$} (V61);
\draw[edge=] (V63) -- node[edgeLabel] {$73$} (V64);
\draw[edge=] (V62) -- node[edgeLabel] {$74$} (V63);
\draw[edge=] (V38) -- node[edgeLabel] {$75$} (V64);
\draw[edge=] (V61) -- node[edgeLabel] {$76$} (V65);
\draw[edge=] (V62) -- node[edgeLabel] {$77$} (V66);
\draw[edge=] (V67) -- node[edgeLabel] {$78$} (V68);
\draw[edge=] (V65) -- node[edgeLabel] {$79$} (V67);
\draw[edge=] (V66) -- node[edgeLabel] {$80$} (V68);
\draw[edge=] (V53) -- node[edgeLabel] {$81$} (V69);
\draw[edge=] (V69) -- node[edgeLabel] {$82$} (V70);
\draw[edge=] (V16) -- node[edgeLabel] {$83$} (V70);
\draw[edge=] (V71) -- node[edgeLabel] {$84$} (V72);
\draw[edge=] (V71) -- node[edgeLabel] {$85$} (V73);
\draw[edge=] (V73) -- node[edgeLabel] {$86$} (V74);
\draw[edge=] (V72) -- node[edgeLabel] {$87$} (V75);
\draw[edge=] (V75) -- node[edgeLabel] {$88$} (V76);
\draw[edge=] (V74) -- node[edgeLabel] {$89$} (V76);
\draw[edge=] (V77) -- node[edgeLabel] {$90$} (V78);
\draw[edge=] (V69) -- node[edgeLabel] {$91$} (V77);
\draw[edge=] (V54) -- node[edgeLabel] {$92$} (V79);
\draw[edge=] (V78) -- node[edgeLabel] {$93$} (V79);
\draw[edge=] (V71) -- node[edgeLabel] {$94$} (V80);
\draw[edge=] (V72) -- node[edgeLabel] {$95$} (V81);
\draw[edge=] (V80) -- node[edgeLabel] {$96$} (V82);
\draw[edge=] (V82) -- node[edgeLabel] {$97$} (V83);
\draw[edge=] (V81) -- node[edgeLabel] {$98$} (V83);
\draw[edge=] (V57) -- node[edgeLabel] {$99$} (V84);
\draw[edge=] (V84) -- node[edgeLabel] {$100$} (V85);
\draw[edge=] (V4) -- node[edgeLabel] {$101$} (V85);
\draw[edge=] (V86) -- node[edgeLabel] {$102$} (V87);
\draw[edge=] (V84) -- node[edgeLabel] {$103$} (V86);
\draw[edge=] (V58) -- node[edgeLabel] {$104$} (V88);
\draw[edge=] (V87) -- node[edgeLabel] {$105$} (V88);
\draw[edge=] (V65) -- node[edgeLabel] {$106$} (V89);
\draw[edge=] (V10) -- node[edgeLabel] {$107$} (V89);
\draw[edge=] (V73) -- node[edgeLabel] {$108$} (V90);
\draw[edge=] (V90) -- node[edgeLabel] {$109$} (V91);
\draw[edge=] (V80) -- node[edgeLabel] {$110$} (V92);
\draw[edge=] (V91) -- node[edgeLabel] {$111$} (V92);
\draw[edge=] (V41) -- node[edgeLabel] {$112$} (V93);
\draw[edge=] (V93) -- node[edgeLabel] {$113$} (V94);
\draw[edge=] (V95) -- node[edgeLabel] {$114$} (V96);
\draw[edge=] (V94) -- node[edgeLabel] {$115$} (V95);
\draw[edge=] (V42) -- node[edgeLabel] {$116$} (V96);
\draw[edge=] (V90) -- node[edgeLabel] {$117$} (V97);
\draw[edge=] (V44) -- node[edgeLabel] {$118$} (V97);
\draw[edge=] (V45) -- node[edgeLabel] {$119$} (V74);
\draw[edge=] (V98) -- node[edgeLabel] {$120$} (V99);
\draw[edge=] (V11) -- node[edgeLabel] {$121$} (V98);
\draw[edge=] (V99) -- node[edgeLabel] {$122$} (V100);
\draw[edge=] (V100) -- node[edgeLabel] {$123$} (V101);
\draw[edge=] (V12) -- node[edgeLabel] {$124$} (V101);
\draw[edge=] (V97) -- node[edgeLabel] {$125$} (V102);
\draw[edge=] (V103) -- node[edgeLabel] {$126$} (V104);
\draw[edge=] (V102) -- node[edgeLabel] {$127$} (V103);
\draw[edge=] (V91) -- node[edgeLabel] {$128$} (V104);
\draw[edge=] (V63) -- node[edgeLabel] {$129$} (V105);
\draw[edge=] (V105) -- node[edgeLabel] {$130$} (V106);
\draw[edge=] (V106) -- node[edgeLabel] {$131$} (V107);
\draw[edge=] (V107) -- node[edgeLabel] {$132$} (V108);
\draw[edge=] (V64) -- node[edgeLabel] {$133$} (V108);
\draw[edge=] (V77) -- node[edgeLabel] {$134$} (V93);
\draw[edge=] (V26) -- node[edgeLabel] {$135$} (V78);
\draw[edge=] (V98) -- node[edgeLabel] {$136$} (V109);
\draw[edge=] (V99) -- node[edgeLabel] {$137$} (V110);
\draw[edge=] (V109) -- node[edgeLabel] {$138$} (V111);
\draw[edge=] (V111) -- node[edgeLabel] {$139$} (V112);
\draw[edge=] (V110) -- node[edgeLabel] {$140$} (V112);
\draw[edge=] (V102) -- node[edgeLabel] {$141$} (V113);
\draw[edge=] (V32) -- node[edgeLabel] {$142$} (V113);
\draw[edge=] (V105) -- node[edgeLabel] {$143$} (V114);
\draw[edge=] (V114) -- node[edgeLabel] {$144$} (V115);
\draw[edge=] (V66) -- node[edgeLabel] {$145$} (V115);
\draw[edge=] (V94) -- node[edgeLabel] {$146$} (V116);
\draw[edge=] (V70) -- node[edgeLabel] {$147$} (V116);
\draw[edge=] (V114) -- node[edgeLabel] {$148$} (V117);
\draw[edge=] (V117) -- node[edgeLabel] {$149$} (V118);
\draw[edge=] (V106) -- node[edgeLabel] {$150$} (V119);
\draw[edge=] (V118) -- node[edgeLabel] {$151$} (V119);
\draw[edge=] (V109) -- node[edgeLabel] {$152$} (V120);
\draw[edge=] (V20) -- node[edgeLabel] {$153$} (V120);
\draw[edge=] (V47) -- node[edgeLabel] {$154$} (V121);
\draw[edge=] (V117) -- node[edgeLabel] {$155$} (V121);
\draw[edge=] (V48) -- node[edgeLabel] {$156$} (V115);
\draw[edge=] (V17) -- node[edgeLabel] {$157$} (V122);
\draw[edge=] (V122) -- node[edgeLabel] {$158$} (V123);
\draw[edge=] (V123) -- node[edgeLabel] {$159$} (V124);
\draw[edge=] (V124) -- node[edgeLabel] {$160$} (V125);
\draw[edge=] (V18) -- node[edgeLabel] {$161$} (V125);
\draw[edge=] (V75) -- node[edgeLabel] {$162$} (V126);
\draw[edge=] (V21) -- node[edgeLabel] {$163$} (V126);
\draw[edge=] (V22) -- node[edgeLabel] {$164$} (V81);
\draw[edge=] (V126) -- node[edgeLabel] {$165$} (V127);
\draw[edge=] (V127) -- node[edgeLabel] {$166$} (V128);
\draw[edge=] (V128) -- node[edgeLabel] {$167$} (V129);
\draw[edge=] (V76) -- node[edgeLabel] {$168$} (V129);
\draw[edge=] (V86) -- node[edgeLabel] {$169$} (V121);
\draw[edge=] (V43) -- node[edgeLabel] {$170$} (V87);
\draw[edge=] (V100) -- node[edgeLabel] {$171$} (V130);
\draw[edge=] (V55) -- node[edgeLabel] {$172$} (V130);
\draw[edge=] (V56) -- node[edgeLabel] {$173$} (V110);
\draw[edge=] (V85) -- node[edgeLabel] {$174$} (V118);
\draw[edge=] (V130) -- node[edgeLabel] {$175$} (V131);
\draw[edge=] (V131) -- node[edgeLabel] {$176$} (V132);
\draw[edge=] (V132) -- node[edgeLabel] {$177$} (V133);
\draw[edge=] (V101) -- node[edgeLabel] {$178$} (V133);
\draw[edge=] (V103) -- node[edgeLabel] {$179$} (V122);
\draw[edge=] (V34) -- node[edgeLabel] {$180$} (V104);
\draw[edge=] (V123) -- node[edgeLabel] {$181$} (V134);
\draw[edge=] (V113) -- node[edgeLabel] {$182$} (V134);
\draw[edge=] (V127) -- node[edgeLabel] {$183$} (V135);
\draw[edge=] (V39) -- node[edgeLabel] {$184$} (V135);
\draw[edge=] (V131) -- node[edgeLabel] {$185$} (V136);
\draw[edge=] (V79) -- node[edgeLabel] {$186$} (V136);
\draw[edge=] (V27) -- node[edgeLabel] {$187$} (V137);
\draw[edge=] (V132) -- node[edgeLabel] {$188$} (V137);
\draw[edge=] (V28) -- node[edgeLabel] {$189$} (V136);
\draw[edge=] (V82) -- node[edgeLabel] {$190$} (V138);
\draw[edge=] (V35) -- node[edgeLabel] {$191$} (V138);
\draw[edge=] (V36) -- node[edgeLabel] {$192$} (V92);
\draw[edge=] (V107) -- node[edgeLabel] {$193$} (V139);
\draw[edge=] (V5) -- node[edgeLabel] {$194$} (V139);
\draw[edge=] (V6) -- node[edgeLabel] {$195$} (V119);
\draw[edge=] (V111) -- node[edgeLabel] {$196$} (V138);
\draw[edge=] (V83) -- node[edgeLabel] {$197$} (V120);
\draw[edge=] (V59) -- node[edgeLabel] {$198$} (V140);
\draw[edge=] (V124) -- node[edgeLabel] {$199$} (V140);
\draw[edge=] (V60) -- node[edgeLabel] {$200$} (V134);
\draw[edge=] (V67) -- node[edgeLabel] {$201$} (V137);
\draw[edge=] (V49) -- node[edgeLabel] {$202$} (V68);
\draw[edge=] (V128) -- node[edgeLabel] {$203$} (V139);
\draw[edge=] (V108) -- node[edgeLabel] {$204$} (V135);
\draw[edge=] (V52) -- node[edgeLabel] {$205$} (V112);
\draw[edge=] (V95) -- node[edgeLabel] {$206$} (V140);
\draw[edge=] (V88) -- node[edgeLabel] {$207$} (V96);
\draw[edge=] (V89) -- node[edgeLabel] {$208$} (V133);
\draw[edge=] (V116) -- node[edgeLabel] {$209$} (V125);
\draw[edge=] (V46) -- node[edgeLabel] {$210$} (V129);
\vertexLabelR[]{V1}{left}{$   $}
\vertexLabelR[]{V2}{left}{$   $}
\vertexLabelR[]{V3}{left}{$   $}
\vertexLabelR[]{V4}{left}{$   $}
\vertexLabelR[]{V5}{left}{$   $}
\vertexLabelR[]{V6}{left}{$   $}
\vertexLabelR[]{V7}{left}{$   $}
\vertexLabelR[]{V8}{left}{$   $}
\vertexLabelR[]{V9}{left}{$   $}
\vertexLabelR[]{V10}{left}{$   $}
\vertexLabelR[]{V11}{left}{$   $}
\vertexLabelR[]{V12}{left}{$   $}
\vertexLabelR[]{V13}{left}{$   $}
\vertexLabelR[]{V14}{left}{$   $}
\vertexLabelR[]{V15}{left}{$   $}
\vertexLabelR[]{V16}{left}{$   $}
\vertexLabelR[]{V17}{left}{$   $}
\vertexLabelR[]{V18}{left}{$   $}
\vertexLabelR[]{V19}{left}{$   $}
\vertexLabelR[]{V20}{left}{$   $}
\vertexLabelR[]{V21}{left}{$   $}
\vertexLabelR[]{V22}{left}{$   $}
\vertexLabelR[]{V23}{left}{$   $}
\vertexLabelR[]{V24}{left}{$   $}
\vertexLabelR[]{V25}{left}{$   $}
\vertexLabelR[]{V26}{left}{$   $}
\vertexLabelR[]{V27}{left}{$   $}
\vertexLabelR[]{V28}{left}{$   $}
\vertexLabelR[]{V29}{left}{$   $}
\vertexLabelR[]{V30}{left}{$   $}
\vertexLabelR[]{V31}{left}{$   $}
\vertexLabelR[]{V32}{left}{$   $}
\vertexLabelR[]{V33}{left}{$   $}
\vertexLabelR[]{V34}{left}{$   $}
\vertexLabelR[]{V35}{left}{$   $}
\vertexLabelR[]{V36}{left}{$   $}
\vertexLabelR[]{V37}{left}{$   $}
\vertexLabelR[]{V38}{left}{$   $}
\vertexLabelR[]{V39}{left}{$   $}
\vertexLabelR[]{V40}{left}{$   $}
\vertexLabelR[]{V41}{left}{$   $}
\vertexLabelR[]{V42}{left}{$   $}
\vertexLabelR[]{V43}{left}{$   $}
\vertexLabelR[]{V44}{left}{$   $}
\vertexLabelR[]{V45}{left}{$   $}
\vertexLabelR[]{V46}{left}{$   $}
\vertexLabelR[]{V47}{left}{$   $}
\vertexLabelR[]{V48}{left}{$   $}
\vertexLabelR[]{V49}{left}{$   $}
\vertexLabelR[]{V50}{left}{$   $}
\vertexLabelR[]{V51}{left}{$   $}
\vertexLabelR[]{V52}{left}{$   $}
\vertexLabelR[]{V53}{left}{$   $}
\vertexLabelR[]{V54}{left}{$   $}
\vertexLabelR[]{V55}{left}{$   $}
\vertexLabelR[]{V56}{left}{$   $}
\vertexLabelR[]{V57}{left}{$   $}
\vertexLabelR[]{V58}{left}{$   $}
\vertexLabelR[]{V59}{left}{$   $}
\vertexLabelR[]{V60}{left}{$   $}
\vertexLabelR[]{V61}{left}{$   $}
\vertexLabelR[]{V62}{left}{$   $}
\vertexLabelR[]{V63}{left}{$   $}
\vertexLabelR[]{V64}{left}{$   $}
\vertexLabelR[]{V65}{left}{$   $}
\vertexLabelR[]{V66}{left}{$   $}
\vertexLabelR[]{V67}{left}{$   $}
\vertexLabelR[]{V68}{left}{$   $}
\vertexLabelR[]{V69}{left}{$   $}
\vertexLabelR[]{V70}{left}{$   $}
\vertexLabelR[]{V71}{left}{$   $}
\vertexLabelR[]{V72}{left}{$   $}
\vertexLabelR[]{V73}{left}{$   $}
\vertexLabelR[]{V74}{left}{$   $}
\vertexLabelR[]{V75}{left}{$   $}
\vertexLabelR[]{V76}{left}{$   $}
\vertexLabelR[]{V77}{left}{$   $}
\vertexLabelR[]{V78}{left}{$   $}
\vertexLabelR[]{V79}{left}{$   $}
\vertexLabelR[]{V80}{left}{$   $}
\vertexLabelR[]{V81}{left}{$   $}
\vertexLabelR[]{V82}{left}{$   $}
\vertexLabelR[]{V83}{left}{$   $}
\vertexLabelR[]{V84}{left}{$   $}
\vertexLabelR[]{V85}{left}{$   $}
\vertexLabelR[]{V86}{left}{$   $}
\vertexLabelR[]{V87}{left}{$   $}
\vertexLabelR[]{V88}{left}{$   $}
\vertexLabelR[]{V89}{left}{$   $}
\vertexLabelR[]{V90}{left}{$   $}
\vertexLabelR[]{V91}{left}{$   $}
\vertexLabelR[]{V92}{left}{$   $}
\vertexLabelR[]{V93}{left}{$   $}
\vertexLabelR[]{V94}{left}{$   $}
\vertexLabelR[]{V95}{left}{$   $}
\vertexLabelR[]{V96}{left}{$   $}
\vertexLabelR[]{V97}{left}{$   $}
\vertexLabelR[]{V98}{left}{$   $}
\vertexLabelR[]{V99}{left}{$   $}
\vertexLabelR[]{V100}{left}{$   $}
\vertexLabelR[]{V101}{left}{$   $}
\vertexLabelR[]{V102}{left}{$   $}
\vertexLabelR[]{V103}{left}{$   $}
\vertexLabelR[]{V104}{left}{$   $}
\vertexLabelR[]{V105}{left}{$   $}
\vertexLabelR[]{V106}{left}{$   $}
\vertexLabelR[]{V107}{left}{$   $}
\vertexLabelR[]{V108}{left}{$   $}
\vertexLabelR[]{V109}{left}{$   $}
\vertexLabelR[]{V110}{left}{$   $}
\vertexLabelR[]{V111}{left}{$   $}
\vertexLabelR[]{V112}{left}{$   $}
\vertexLabelR[]{V113}{left}{$   $}
\vertexLabelR[]{V114}{left}{$   $}
\vertexLabelR[]{V115}{left}{$   $}
\vertexLabelR[]{V116}{left}{$   $}
\vertexLabelR[]{V117}{left}{$   $}
\vertexLabelR[]{V118}{left}{$   $}
\vertexLabelR[]{V119}{left}{$   $}
\vertexLabelR[]{V120}{left}{$   $}
\vertexLabelR[]{V121}{left}{$   $}
\vertexLabelR[]{V122}{left}{$   $}
\vertexLabelR[]{V123}{left}{$   $}
\vertexLabelR[]{V124}{left}{$   $}
\vertexLabelR[]{V125}{left}{$   $}
\vertexLabelR[]{V126}{left}{$   $}
\vertexLabelR[]{V127}{left}{$   $}
\vertexLabelR[]{V128}{left}{$   $}
\vertexLabelR[]{V129}{left}{$   $}
\vertexLabelR[]{V130}{left}{$   $}
\vertexLabelR[]{V131}{left}{$   $}
\vertexLabelR[]{V132}{left}{$   $}
\vertexLabelR[]{V133}{left}{$   $}
\vertexLabelR[]{V134}{left}{$   $}
\vertexLabelR[]{V135}{left}{$   $}
\vertexLabelR[]{V136}{left}{$   $}
\vertexLabelR[]{V137}{left}{$   $}
\vertexLabelR[]{V138}{left}{$   $}
\vertexLabelR[]{V139}{left}{$   $}
\vertexLabelR[]{V140}{left}{$   $}
\end{tikzpicture}

%% file: octa.tex
\begin{tikzpicture}[vertexBall, edgeDouble, faceStyle, scale=2]

\coordinate (V1_1) at (0., 0.);
\coordinate (V2_1) at (1., 0.);
\coordinate (V3_1) at (0.4999999999999999, 0.8660254037844386);
\coordinate (V4_1) at (-0.4999999999999999, 0.8660254037844385);
\coordinate (V5_1) at (0.5000000000000001, -0.8660254037844386);
\coordinate (V5_2) at (-0.9999999999999997, 0.);
\coordinate (V6_1) at (1.5, 0.8660254037844388);
\coordinate (V6_2) at (1.5, -0.8660254037844384);
\coordinate (V6_3) at (0., 1.732050807568877);
\coordinate (V6_4) at (-1.5, 0.8660254037844384);

\fill[face]  (V2_1) -- (V3_1) -- (V1_1) -- cycle;
\node[faceLabel] at (barycentric cs:V2_1=1,V3_1=1,V1_1=1) {$1$};
\fill[face]  (V5_1) -- (V6_2) -- (V2_1) -- cycle;
\node[faceLabel] at (barycentric cs:V5_1=1,V6_2=1,V2_1=1) {$8$};
\fill[face]  (V1_1) -- (V5_1) -- (V2_1) -- cycle;
\node[faceLabel] at (barycentric cs:V1_1=1,V5_1=1,V2_1=1) {$4$};
\fill[face]  (V2_1) -- (V6_1) -- (V3_1) -- cycle;
\node[faceLabel] at (barycentric cs:V2_1=1,V6_1=1,V3_1=1) {$5$};
\fill[face]  (V4_1) -- (V5_2) -- (V1_1) -- cycle;
\node[faceLabel] at (barycentric cs:V4_1=1,V5_2=1,V1_1=1) {$3$};
\fill[face]  (V3_1) -- (V6_3) -- (V4_1) -- cycle;
\node[faceLabel] at (barycentric cs:V3_1=1,V6_3=1,V4_1=1) {$6$};
\fill[face]  (V3_1) -- (V4_1) -- (V1_1) -- cycle;
\node[faceLabel] at (barycentric cs:V3_1=1,V4_1=1,V1_1=1) {$2$};
\fill[face]  (V4_1) -- (V6_4) -- (V5_2) -- cycle;
\node[faceLabel] at (barycentric cs:V4_1=1,V6_4=1,V5_2=1) {$7$};

\draw[edge] (V2_1) -- node[edgeLabel] {$1$} (V1_1);
\draw[edge] (V1_1) -- node[edgeLabel] {$2$} (V3_1);
\draw[edge] (V1_1) -- node[edgeLabel] {$3$} (V4_1);
\draw[edge] (V5_1) -- node[edgeLabel] {$4$} (V1_1);
\draw[edge] (V1_1) -- node[edgeLabel] {$4$} (V5_2);
\draw[edge] (V3_1) -- node[edgeLabel] {$5$} (V2_1);
\draw[edge] (V2_1) -- node[edgeLabel] {$6$} (V5_1);
\draw[edge] (V6_1) -- node[edgeLabel] {$7$} (V2_1);
\draw[edge] (V2_1) -- node[edgeLabel] {$7$} (V6_2);
\draw[edge] (V4_1) -- node[edgeLabel] {$8$} (V3_1);
\draw[edge] (V3_1) -- node[edgeLabel] {$9$} (V6_1);
\draw[edge] (V6_3) -- node[edgeLabel] {$9$} (V3_1);
\draw[edge] (V5_2) -- node[edgeLabel] {$10$} (V4_1);
\draw[edge] (V4_1) -- node[edgeLabel] {$11$} (V6_3);
\draw[edge] (V6_4) -- node[edgeLabel] {$11$} (V4_1);
\draw[edge] (V6_2) -- node[edgeLabel] {$12$} (V5_1);
\draw[edge] (V5_2) -- node[edgeLabel] {$12$} (V6_4);

\vertexLabelR{V1_1}{left}{$1$}
\vertexLabelR{V2_1}{left}{$2$}
\vertexLabelR{V3_1}{left}{$3$}
\vertexLabelR{V4_1}{left}{$4$}
\vertexLabelR{V5_1}{left}{$5$}
\vertexLabelR{V5_2}{left}{$5$}
\vertexLabelR{V6_1}{left}{$6$}
\vertexLabelR{V6_2}{left}{$6$}
\vertexLabelR{V6_3}{left}{$6$}
\vertexLabelR{V6_4}{left}{$6$}

\end{tikzpicture}